\documentclass[12pt]{amsart}

\usepackage[english]{babel} 
\usepackage{microtype}
\usepackage{fullpage}       
\usepackage{mathtools} 
\usepackage{amsthm}    
\usepackage{amsfonts}  
\usepackage{amssymb}
\usepackage{stmaryrd}

\usepackage[dvipsnames]{xcolor}
\usepackage{graphicx}
\usepackage{tikz-cd}
\usepackage{todonotes}

\usepackage[pagebackref]{hyperref}
\hypersetup{
  colorlinks=true,
  linkcolor=NavyBlue,
  citecolor=Maroon,
  urlcolor=NavyBlue,
  pdftitle={A topos for étale-finite Heyting algebras},
  pdfauthor={Marco Abbadini, Rodrigo Nicolau Almeida, and Igor Arrieta},
  pdfkeywords={Toposes, intuitionistic logic, Heyting algebras, Esakia duality}
}
\usepackage[capitalise]{cleveref} 

\theoremstyle{plain}
\newtheorem{theorem}{Theorem}[section]
\newtheorem{lemma}[theorem]{Lemma}
\newtheorem{proposition}[theorem]{Proposition}
\newtheorem{corollary}[theorem]{Corollary}
\newtheorem{claim}[theorem]{Claim}

\theoremstyle{definition}
\newtheorem{definition}[theorem]{Definition}
\newtheorem{notation}[theorem]{Notation}
\newtheorem{example}[theorem]{Example}
\newtheorem{question}[theorem]{Question}
\newtheorem{remark}[theorem]{Remark}

\makeatletter
\def\l@subsection{\@tocline{2}{0pt}{2.5pc}{5pc}{}}
\def\l@subsubsection{\@tocline{3}{0pt}{5pc}{7.5pc}{}}
\makeatother

\newcommand{\cat}[1]{\mathbf{#1}}
\newcommand{\N}{\mathbb{N}}
\newcommand{\op}{\mathrm{op}}
\newcommand{\id}{\mathrm{id}}
\newcommand{\E}{\mathcal{E}}
\newcommand{\V}{\mathcal{V}}
\newcommand{\Vu}{\V^{\uparrow}}
\newcommand{\tru}[1]{\llbracket #1 \rrbracket}

\newcommand{\Set}{\cat{Set}}
\newcommand{\Pos}{\cat{Pos}}
\newcommand{\DLat}{\cat{DLat}}
\newcommand{\Top}{\cat{Top}}
\newcommand{\TopLH}{\Top_{\mathrm{LH}}}
\newcommand{\HA}{\cat{HA}}
\newcommand{\Esa}{\cat{Esa}}
\newcommand{\Pries}{\cat{Pries}}
\renewcommand{\L}{\cat{L}}

\newcommand{\KM}{\mathsf{KM}}
\newcommand{\SLH}{\mathrm{SLH}}
\newcommand{\Viet}{\mathrm{Viet}}
\newcommand{\Spec}{\mathsf{Spec}}
\newcommand{\Sub}{\mathrm{Sub}}
\newcommand{\ClopUp}{\mathsf{ClopUp}}
\newcommand{\Up}{\mathrm{Up}}
\newcommand{\Typ}{\mathsf{Typ}}
\newcommand{\Types}{\mathsf{Types}}
\newcommand{\EsaEt}{\cat{Esa\acute{E}t}}

\newcommand{\cofnaturals}{\mathcal{P}_{\varnothing,\mathrm{cof}}(\N)}
\newcommand{\fincofnaturals}{\mathcal{P}_{\mathrm{fin,cof}}(\N)}

\newcommand{\Q}{\mathbb{Q}}

\newcommand{\longhookrightarrow}{\lhook\joinrel\longrightarrow}
\newcommand{\climp}{\Rightarrow}
\renewcommand{\u}{{\uparrow}}
\renewcommand{\d}{{\downarrow}}

\newcommand{\expl}[2]{%
  \underset{%
    \substack{%
      \big\uparrow\\[-0.2ex]
      \mathrlap{%
        \mbox{%
          \footnotesize
          \hspace{-1em}%
          \begin{tabular}{@{}l@{}}
            #2
          \end{tabular}%
        }%
      }%
    }%
  }{#1}%
}

\makeatletter
\let\@setaddresses\relax
\makeatother

\newcommand{\authorwithphoto}[5]{%
  \par\bigskip
  \noindent
  \begin{minipage}[c]{0.16\textwidth}
    \includegraphics[width=2cm,height=2.5cm,keepaspectratio]{#1}
  \end{minipage}%
  \hfill%
  \begin{minipage}[c]{0.80\textwidth}
    \textbf{#2}\\
    #3\\
    \textit{Email address:} \href{mailto:#4}{#4}\\
    \textit{Webpage:} \url{#5}
  \end{minipage}%
}

\title{A topos for étale-finite Heyting algebras}

\author{Marco Abbadini}
\address[Marco Abbadini]{Research Institute in Mathematics and Physics, Université catholique de Louvain, Louvain-la-Neuve, 1348, Belgium}
\email{marco.abbadini@uclouvain.be}

\author{Rodrigo Nicolau Almeida}
\address[Rodrigo Nicolau Almeida]{Institute for Logic, Language and Computation, University of Amsterdam, Amsterdam, 1098XH, The Netherlands}
\email{r.dacruzsilvapinadealmeida@uva.nl}

\author{Igor Arrieta}
\address[Igor Arrieta]{Department of Mathematics, University of the Basque Country UPV/EHU, Bilbao, 48080, Spain}
\email{igor.arrieta@ehu.eus}

\keywords{Toposes, intuitionistic logic, Heyting algebras, Esakia duality}
\subjclass[2020]{Primary: 03G30. Secondary: 06D20, 06D50, 03G10, 18B25, 18F70}

\begin{document}

\begin{abstract}
    A longstanding open problem is whether every Heyting algebra is the lattice of truth values (i.e., of subterminal objects) of some elementary topos. A positive answer is known for complete Heyting algebras (i.e., locales) via sheaves, and for Boolean algebras via a construction due to Peter Freyd.

    We extend Freyd's construction to all \emph{étale-finite} Heyting algebras, in the sense of Evgeny Kuznetsov. These are the Heyting algebras satisfying a generalisation of the law of excluded middle relative to some finite Heyting subalgebra.
    For every étale-finite Heyting algebra $H$, we use Esakia duality to construct an elementary topos whose lattice of truth values is isomorphic to $H$, thereby extending the class of Heyting algebras for which a positive answer to the Heyting-to-topos problem is known.
    
    The toposes we construct are categories of certain compact étale spaces.
    As a consequence, they are \emph{finitely propositional}: every object has a finite cover by subterminal objects.
    We show that a Heyting algebra occurs as the lattice of truth values of some finitely propositional topos if and only if it is étale-finite. 
    This exhibits an obstruction to extending the use of compact étale spaces beyond the étale-finite case.
\end{abstract}
\maketitle
\tableofcontents

\section{Introduction}

In \cite{Pitts1992}, Andrew Pitts proved the uniform interpolation property for propositional intuitionistic logic. In the introduction to that paper, the author explained how he had arrived at this result after having tried to disprove it (see \cite[Prop.~17]{vangooluniforminterpolation} for details) in an attempt to provide a negative answer to the following question:

\begin{question}[Heyting-to-topos problem]\label{q: pitts question}
    Is every Heyting algebra isomorphic to the lattice $\Sub_{\E}(1)$ of subobjects of the terminal object $1$ of some elementary topos $\E$?
\end{question}

The lattice $\Sub_{\E}(1)$ is also known as the \emph{lattice of truth values of $\E$}. 
As we know from a private communication with Pitts, he learned of the problem from Peter Freyd when Freyd was visiting Cambridge in 1980 and Pitts was a PhD student; we do not know whether it was Freyd who originally posed the problem.

The Heyting-to-topos problem has the following logical interpretation: does every intuitionistic propositional theory $T$ have a higher-order extension $T'$ such that 
\begin{enumerate}
    \item $T'$ is conservative over $T$, i.e., $T'$ does not prove any new propositional sentences, and
    \item for every \emph{higher-order} sentence $\varphi$, there is a \emph{propositional} sentence $\varphi'$ such that $T'$ proves that $\varphi$ and $\varphi'$ are equivalent?
\end{enumerate}

The proof of uniform interpolation ultimately left \cref{q: pitts question} open.
Two classes of Heyting algebras have been shown to be \emph{topos-admissible}, i.e., isomorphic to $\Sub_{\E}(1)$ for some elementary topos $\E$:
\begin{itemize}

    \item \emph{Complete Heyting algebras} (i.e., locales), where one can take the topos of sheaves on the given locale \cite[Ch.~IX.5]{MacLane1994};
    
    \item \emph{Boolean algebras}, for which a construction of Peter Freyd is known to produce a topos with the desired properties (see \cite[Exercise~9.11]{johnstonetopos}, and also \cite[Example~4.8]{Pitts2002} for a similar construction).
    
\end{itemize}
There is also an isolated example of a Heyting algebra known to be topos-admissible: $\mathbb{Q} \cap [0,1]$ (see the proof of \cite[Prop.~2.71]{Freyd1972}).\footnote{In fact, every $\omega$-categorical countable Heyting algebra that embeds elementarily into a topos-admissible Heyting algebra is topos-admissible; for the case of $\mathbb{Q} \cap [0,1]$, one uses that it embeds elementarily into $[0,1]$.}

These known cases can be combined to produce other topos-admissible Heyting algebras using the fact that the class of topos-admissible Heyting algebras is closed under arbitrary products (as one can take the Cartesian product of toposes) and homomorphic images (as one can perform the filter-quotient construction, see, e.g., \cite[Thm.~2, p.~261]{MacLane1994}).
To the best of our knowledge, these are the only Heyting algebras that, prior to the present work, have been known to be topos-admissible.

The Heyting-to-topos problem has appeared several times in the literature during the 1990s and early 2000s (see e.g.\ \cite[Introduction]{Ghilardi2002} and \cite{Pitts2002}). 
In the early 2000s, Dimitri (Dito) Pataraia began work on this problem. Part of his analysis is outlined in a lecture given by Peter Johnstone \cite{johnstonehochasminimaltoposes} on the subject of ``hochas'' (higher-order cylindric Heyting algebras), and involved adapting methods related to the step-by-step construction of finitely generated free Heyting algebras \cite{ghilardifreeheyting} to adding higher-order quantifiers. 
Although in the community there were claims of a solution, Pataraia passed away without leaving a written proof.

Over the years, the problem has been asked in different contexts (see, e.g., \cite{Mathoverflow,stackexchange}). 
In recent years, there has been renewed interest in the problem: in \cite{kuznetsovjibladze}, a state-of-the-art account of the problem was given, emphasising the Boolean case and outlining some of the difficulties in extending the technique beyond this case. In \cite{pittshigherorder}, the problem was once again recalled, leading to renewed attention from the research community (see, e.g., \cite[Question~16]{vangooluniforminterpolation}).

Progress on the problem faces two alternatives \cite{pittshigherorder} (see also \cite[Sec.~3.3]{vangooluniforminterpolation}).
On the one hand, if the answer is \emph{negative}, some free Heyting algebra will give a counterexample\footnote{This is a consequence of the filter-quotient construction.} (most likely the one on a countably infinite set of generators). On the other hand, if the answer is \emph{positive}, then, for every \emph{non-complete} Heyting algebra, one has to construct an elementary topos that is not a Grothendieck topos. Pataraia's approach via hochas---similar to the tripos-to-topos construction \cite{Hyland1980,Pitts2002}---amounts to an algebraic encoding of the structure of a minimal topos \cite{johnstonehochasminimaltoposes}. By contrast, Freyd's approach is strongly geometric and heavily reminiscent of the methods of Grothendieck topos theory. Given a Boolean algebra $B$, Freyd's elementary topos is constructed as follows: letting $X$ denote the Stone space dual to $B$, one takes the topos of local homeomorphisms $Y \to X$ with $Y$ a Stone space.

In this paper, we identify a class of Heyting algebras for which a generalisation of Freyd's approach works. We begin, in \cref{sec: preliminaries}, by recalling the definition of an elementary topos and by reviewing the basics of Esakia duality between Heyting algebras and Esakia spaces. 
\Cref{sec: etale-finite Heyting algebras} recalls the class of \emph{étale-finite Heyting algebras}, introduced by Kuznetsov in \cite{kuznetsovetale}, which includes many Heyting algebras that are neither Boolean nor complete.
Our main result is in \cref{sec:topos}: we show that every étale-finite Heyting algebra is topos-admissible. 
In \cref{sec:finitely-propositional-toposes}, we show that étale-finite Heyting algebras are the largest class one can cover with the technology of this paper. To show this, we introduce \emph{finitely propositional toposes}: these are the toposes in which every object has a finite jointly epimorphic family of morphisms from subterminal objects.
The toposes used for the main result in \cref{sec:topos} are finitely propositional, and we show that any finitely propositional topos has an étale-finite Heyting algebra as its lattice of truth values. We conclude by indicating some future directions in this area of research.

\section{Preliminaries} \label{sec: preliminaries}

\subsection{Toposes}

An \emph{elementary topos} is commonly defined as a category that has finite limits, is Cartesian closed, and has a subobject classifier.
In this paper, we will mainly use an alternative, equivalent definition (\cite[Cor.~A.2.3.4]{johnstone2002sketches}): an \emph{elementary topos} is a category that has all finite limits and where every object has a power object, in the following sense.

\begin{definition}[Power object]
    Let $\cat{C}$ be a category with finite limits, and $A\in \cat{C}$. A \emph{power object of $A$} is a triple $(P(A),{\in_{A}},i)$ where $P(A)$ and $\in_{A}$ are objects of $\cat{C}$ and $i \colon {\in_{A}}\rightarrowtail A\times P(A)$ is a monomorphism such that, for all objects $B$ and $R$ of $\cat{C}$ and every monomorphism $f \colon R\rightarrowtail A\times B$, there is a unique morphism $\tilde{f} \colon B\to P(A)$ such that there is a (necessarily unique) morphism $R \to {\in_A}$ making the following square a pullback. 
    \[
        \begin{tikzcd}
            R \arrow[d, tail, "f"'] \arrow[r]                    & \in_{A} \arrow[d, tail, "i"] \\
            A\times B \arrow[r, "\id_{A}\times \tilde{f}"'] & A\times P(A)
        \end{tikzcd}
    \]
\end{definition}

\subsection{Esakia duality}

As is customary, for a poset $X$ and a subset $S\subseteq X$, we write 
\[
    {\u}S=\{x\in X : \exists s\in S \text{ s.t.\ } s\leq x\} \quad \text{and} \quad {\d}S=\{x\in X : \exists s\in S \text{ s.t.\ } x\leq s\}.
\]
We say that $S$ is an \emph{upset} if $S={\u}S$ and a \emph{downset} if $S={\d}S$. 
If $S = \{ x \}$ is a singleton, we simply write ${\u}x$ and ${\d}x$.
By an \emph{ordered topological space}, or simply an \emph{ordered space}, we mean a poset that is also a topological space. 
Given an ordered space $X$, we denote by $\ClopUp(X)$ the set of its clopen upsets.

Hilary Priestley generalised Stone duality from Boolean algebras to bounded distributive lattices by introducing what are now called \emph{Priestley spaces} \cite{Priestley1970}.

\begin{definition}[Priestley space, see e.g.\ {\cite[p.~258]{Davey2002-lr}}]
    A \emph{Priestley space} is an ordered space $(X,\leq)$ that is compact and satisfies the \emph{Priestley separation axiom}: for all $x,y \in X$ with $x\nleq y$, there is $U\in \ClopUp(X)$ such that $x\in U$ and $y\notin U$. 
\end{definition}

The celebrated \emph{Priestley duality} \cite{Priestley1970} establishes a categorical duality between
\begin{itemize}
    \item the category $\Pries$ of Priestley spaces and continuous order-preserving maps between them, and
    \item the category $\DLat$ of bounded distributive lattices with bounded lattice homomorphisms between them.
\end{itemize}
In this duality, 
the functor
\[
\ClopUp \colon \Pries^\op \to \DLat
\]
associates to each Priestley space $X$ the bounded distributive lattice 
\[
\ClopUp(X)
\]
of its clopen upsets, ordered by inclusion.
On morphisms, it maps a morphism $f \colon X \to Y$ to the preimage function $f^{-1}[-]\colon \ClopUp(Y) \to \ClopUp(X)$.
The quasi-inverse of this functor is denoted by
\[
\Spec \colon \DLat \longrightarrow \Pries^{\op},
\]
and maps a bounded distributive lattice to the set of prime filters, equipped with an appropriate topology and order.

Esakia obtained a duality for Heyting algebras \cite{esakiatopologicalkripkemodels} by introducing what we nowadays call \emph{Esakia spaces}.

\begin{definition}[{Esakia space, see e.g.\ \cite[Def.~3.1.1]{Esakiach2019HeyAlg}}]
    An \emph{Esakia space} is a Priestley space $(X,\leq)$ such that, for every clopen subset $W$ of $X$, ${\d}W$ is clopen.
\end{definition}

\emph{Esakia duality} can be seen as a restriction of Priestley duality to the (non-full) subcategory $\HA$ of Heyting algebras and Heyting homomorphisms. 
For an Esakia space $X$, $\ClopUp(X)$ is a Heyting algebra, with the implication given by $A \to B = X\smallsetminus{\d}(A\smallsetminus B)$. 
Note that each finite poset, with the discrete topology, is an Esakia space.

Since Heyting homomorphisms preserve the implication, their dual maps between Esakia spaces satisfy a property that is stronger than order-preservation:

\begin{definition}[p-morphism]
    A \emph{p-morphism} from a poset $P$ to a poset $Q$ is an order-preserving map $f\colon P\to Q$ such that, for all $x\in P$ and $y\in Q$ with $f(x)\leq y$, there is $x'\geq x$ such that $f(x')=y$.
    \[
        \begin{tikzpicture}[>=stealth, scale=1.1]
            \node[circle, draw, fill=blue!15, inner sep=2pt, label=left:$x$] (x) at (0,0) {};
            \node[circle, draw, fill=blue!15, inner sep=2pt, label=left:$\exists x'$] (xp) at (0,1.5) {};
            \draw[dashed] (x) -- (xp);
        
            \node[circle, draw, fill=red!15, inner sep=2pt, label=right:$f(x)$] (fx) at (3,0) {};
            \node[circle, draw, fill=red!15, inner sep=2pt, label=right:$y$] (y) at (3,1.5) {};
            \draw (fx) -- (y);
        
            \draw[->] (x) -- (fx);
            \draw[->, dashed] (xp) -- (y);
        \end{tikzpicture}
    \]
\end{definition}

\begin{remark}\label{rem: remark on p-morphisms}
    It is well known that an order-preserving map is a p-morphism if and only if the image of any upset is an upset (see e.g.\ \cite[Prop.~1.4.12]{Esakiach2019HeyAlg}).
\end{remark}

Given a continuous p-morphism $f\colon X\to Y$ between Esakia spaces, the preimage function $f^{-1}[-] \colon \ClopUp(Y)\to \ClopUp(X)$ is a Heyting homomorphism. 
This provides a duality between the category $\Esa$ of Esakia spaces with continuous p-morphisms and $\HA$.

Throughout the paper, we will make use of some constructions on ordered spaces, which we summarise:

\begin{proposition}[Basic closure properties of Priestley and Esakia spaces] \label{properties.pri}\hfill
    \begin{enumerate}
        \item\label{properties.pri1}
        Let $X$ and $Y$ be Priestley spaces. The product $X\times Y$ with the product topology and order is a Priestley space. Moreover, if $X$ and $Y$ are Esakia spaces, then $X\times Y$ is an Esakia space, as well.\footnote{We warn the reader that, in general, this is \emph{not} the categorical product in the category $\Esa$ of Esakia spaces and continuous p-morphisms: while the cone $X \leftarrow X \times Y \rightarrow Y$ is a cone in the category of Esakia spaces, the unique comparison function from another cone is not, in general, a p-morphism.} 
        
        \item\label{properties.pri2} A closed subspace of a Priestley space is a Priestley space, and a closed upset of an Esakia space is an Esakia space.

        \item \label{properties.pri3} If $f\colon Y\to X$ is a surjective continuous p-morphism where $Y$ is an Esakia space and $X$ a Priestley space, then $X$ is likewise an Esakia space.
    \end{enumerate}
\end{proposition}

\begin{proof}
    For the assertions concerning Priestley spaces, see
    \cite[Ch.~11]{Davey2002-lr}. For the assertions concerning Esakia spaces, see \cite[Sec.~A.8.1]{Esakiach2019HeyAlg} for products and \cite[Thm.~3.1.2]{Esakiach2019HeyAlg} for closed upsets; the fact that Esakia spaces are closed under p-morphic images is standard (this fact is used throughout \cite[Sec.~3.3]{Esakiach2019HeyAlg}).
\end{proof}

\subsection{Vietoris spaces}

Vietoris spaces have a long history \cite{Michael1951TopologiesOS}; we will need them in the context of Stone spaces: 

\begin{definition}[Vietoris space] \label{d:Vietoris}
    Let $X$ be a Stone space. We write $\V(X)$ for the set of closed subsets of $X$.
    The set $\V(X)$ is topologised by declaring the following sets to form a subbasis:
    \[
        \llbracket W \rrbracket=\{C\in \V(X) : C\subseteq W\} \text{ and } \langle W\rangle=\{C\in \V(X) : C\cap W\neq \varnothing\},
    \]
    where $W$ is a clopen subset of $X$.
    The space $\V(X)$ is called the \emph{Vietoris space} of $X$.
    
    Furthermore, when $X$ is a Priestley space, we write $\Vu(X)$ for the subset of $\V(X)$ consisting of all closed \emph{upsets} of $X$.
\end{definition}

\begin{notation}[$\climp$, $\Box$, $\to$]\label{not:box-heyting-ops}
Let $X$ be a set. For subsets $U,V\subseteq X$, we write
\[
    U\climp V \coloneqq (X\smallsetminus U)\cup V
\]
for the classical implication between $U$ and $V$ in the Boolean algebra $\mathcal P(X)$.

Now suppose that $X$ is a poset. For $A\subseteq X$, set
\[
    \Box A \coloneqq \{x\in X : {\u}x\subseteq A\}.
\]
Thus $\Box A$ is an upset. If $X$ is an Esakia space and $A$ is clopen, then $\Box A$ is clopen as well, since
\[
    X\smallsetminus \Box A = {\d}(X\smallsetminus A).
\]

Finally, if $U,V\subseteq X$ are upsets, we write
\[
    U\to V \coloneqq \{x\in X : {\u}x\cap U \subseteq V\}.
\]
Equivalently,
\[
    U\to V=\Box(U\climp V).
\]
This is the Heyting implication in the lattice of upsets of $X$.  We will use the composite $\Box(U\climp V)$ whenever ambiguity may arise. Note also that, in an Esakia space, if $U$ and $V$ are clopen upsets, then $U\rightarrow V$ is a clopen upset, too.
\end{notation}

\begin{notation} \label{not:forall-along-map}
    Let $f\colon Y\to X$ be a continuous map between compact Hausdorff spaces. 
    For a subset $A\subseteq Y$, we define
    \[
        \forall_f[A]\coloneqq \{x\in X : f^{-1}[\{x\}]\subseteq A\} = X\smallsetminus f[Y\smallsetminus A].
    \]
    If $A$ is clopen and $f$ is a local homeomorphism (and hence open), then $f[Y\smallsetminus A]$ is open and compact, and thus clopen, and so $\forall_f[A]$ is clopen.
\end{notation}

The following lemma summarises some important facts we will need about Vietoris spaces:

\begin{lemma}[Basic properties of Vietoris spaces] \label{lem: basic facts about vietoris}
    \hfill
    \begin{enumerate}
        \item \label{eq: vietoris esakia} 
        If $X$ is a Stone space, then the space $\V(X)$ is a Stone space.
        \item \label{eq: continuous maps lift to Vietoris} If $f \colon X\to Y$ is a continuous map between Stone spaces, then the function
        \begin{align*}
            f[-] \colon \V(X)& \longrightarrow \V(Y)\\
            C & \longmapsto f[C]
        \end{align*}
        is continuous.
        \item \label{eq: restriction maps are continuous} If $W$ is a clopen subset of a Stone space $X$, then the function
        \begin{align*}
            r_{W}\colon \V(X) & \longrightarrow \V(X)\\
            C& \longmapsto C\cap W
        \end{align*}
        is continuous.
        
        \item \label{eq: ordered vietoris esakia} If $X$ is an Esakia space, then $\Vu(X)$ is a closed subset of $\V(X)$.
    \end{enumerate}
\end{lemma}

\begin{proof}
    \eqref{eq: vietoris esakia} This is a standard fact (see e.g.\ \cite[Thm.~4.2]{Michael1951TopologiesOS}). 
    
    \eqref{eq: continuous maps lift to Vietoris} This is likewise well known (for a proof see e.g.\ \cite[Lem.~3.9]{Kupke2004}).

    \eqref{eq: restriction maps are continuous}
    Let $W$ be a clopen subset of a Stone space $X$.
    The map $r_{W}\colon \V(X)\to \V(X)$ is well-defined since $W$ is closed. It is moreover continuous, since, for any subbasic clopen in $\V(X)$ of the form $\tru{W_{0}}$ for a clopen $W_{0}\subseteq X$, we have:
    \begin{equation*}
        r_{W}^{-1}[\tru{W_{0}}]=\{C\in \V(X) : C\cap W\subseteq W_{0}\}=\tru{W\climp W_{0}},
    \end{equation*}
    which is clopen; moreover, it follows that the set $r_{W}^{-1}[\langle W_{0} \rangle] = \V(X) \smallsetminus r_{W}^{-1}[\tru{X \smallsetminus W_{0}}]$ is clopen.
    
    \eqref{eq: ordered vietoris esakia} Let $X$ be an Esakia space. Let $K \in \V(X) \smallsetminus \Vu(X)$; we prove that there is a clopen subset of $\V(X)$ to which $K$ belongs and that is disjoint from $\Vu(X)$. Since $K$ is not an upset, there are $x\in K$ and $y \in X$ such that $x\leq y$ and $y\notin K$. By zero-dimensionality, there is a clopen subset $W$ such that $K\subseteq W$ and $y\notin W$. Then, we consider
    \begin{equation*}
        \tru{W}\cap \langle {\d}(X\smallsetminus W)\rangle.
    \end{equation*}
    Note that ${\d}(X\smallsetminus W)$ is clopen because $X$ is Esakia.
    Moreover, note that $K$ belongs to this set, since $K\subseteq W$ and $K \ni x \leq y \notin W$. 
    Furthermore, for every $K'\in \tru{W}\cap \langle {\d}(X\smallsetminus W)\rangle$, by definition we have $K' \subseteq W$ and there are $x'\in K'$ and $y' \in X$ such that $x'\leq y'$ and $y'\notin W$, which implies that $K'$ is not upwards closed.
\end{proof}

\section{Étale-finite Heyting algebras and étale-finite Esakia spaces} \label{sec: etale-finite Heyting algebras}

In this section, we recall the class of étale-finite Heyting algebras, recently introduced by Evgeny Kuznetsov, together with the corresponding Esakia-dual formulation. This is the class of Heyting algebras that our main theorem will show to be topos-admissible. We then collect examples, closure properties, and non-examples, including examples that lie outside the class of Heyting algebras that were previously known to be topos-admissible.

\subsection{Étale-finite Heyting algebras and étale-finite Esakia spaces}

\begin{definition}[\'Etale-finite Heyting algebra]
    Let $H_{0}$ be a finite Heyting algebra. A Heyting algebra homomorphism $f\colon H_{0}\to H$ is said to satisfy \emph{Jibladze's law} if, for every $x\in H$,
    \begin{equation*}
        \bigvee_{x_0\in H_{0}}\bigl(x\leftrightarrow f(x_0)\bigr)=1.
    \end{equation*}
    Given an arbitrary Heyting algebra $H$ and a finite Heyting algebra $H_{0}$, we say that $H$ is $H_{0}$-\emph{étale} if there is a homomorphism $f\colon H_{0}\to H$ satisfying Jibladze's law. We say that $H$ is \emph{étale-finite} if there is a finite Heyting algebra $H_{0}$ such that $H$ is $H_{0}$-étale.
\end{definition}

\begin{remark} \label{r:etale-finite-basic} \hfill
    \begin{enumerate}
        
        \item
        (Jibladze's law generalises the excluded middle)
        Jibladze's law can be seen as a generalisation of the law of excluded middle: taking $H_{0} = 2$ and $f \colon 2 \to H$ to be the unique Heyting homomorphism, the equation defining Jibladze's law becomes 
        \[ 
            (x \leftrightarrow 1) \vee (x \leftrightarrow 0) = 1,
        \]
        i.e., 
        \[
        x\vee \neg x = 1,
        \]
        which is the law of excluded middle.

        Frame-theoretic analogues of this law seem to have first appeared in print in \cite{kockreyes}, where they are attributed to Jibladze (March 1990); they have been studied more systematically by Jibladze in \cite{jibladzerelativebooleanness}.
        This motivates the terminology ``Jibladze's law'', which we propose here.
        
        \item (\'Etale-finite $\leftrightarrow$ finite subalgebra) \label{i: finite algebra condition}
        Replacing $H_0$ by the image $f[H_0]$, we observe that a Heyting algebra $H$ is étale-finite if and only if there is a finite Heyting \emph{subalgebra} $H_{0}$ of $H$ such that, for every $x \in H$, in $H$ we have
        \[
        \bigvee_{x_0\in H_{0}}(x\leftrightarrow x_0) = 1.
        \]
        In \cref{prop: finite set to etale-finite} we will show that it suffices that $H_0$ is a finite \emph{subset} of $H$, not necessarily a \emph{subalgebra}.
    \end{enumerate}
\end{remark}

Given a Heyting algebra $H$, we write $\mathsf{Var}(H)_{h\in H}$  for the variety generated by the algebra $H$ in the signature of Heyting algebras enriched by a constant for each element in $H$. The following was shown in \cite[pp.~11--12]{kuznetsovetale}:

\begin{proposition} \label{prop: characterisation of etale maps}
    For $H_{0}$ a finite Heyting algebra and $H$ an arbitrary Heyting algebra, the following are equivalent:
    \begin{enumerate}
        \item 
        $H$ is an $H_{0}$-étale Heyting algebra;
        
        \item $H$ is the Heyting reduct of some algebra in $\mathsf{Var}(H_{0})_{h\in H_{0}}$.
    \end{enumerate}
\end{proposition}

\begin{remark}[\'Etale-finite $\Rightarrow$ locally finite]\label{rem: local finiteness from etale-finiteness}
    Every étale-finite Heyting algebra is locally finite. Indeed, by \cref{prop: characterisation of etale maps}, an étale-finite Heyting algebra $H$ is the Heyting reduct of some algebra in $\mathsf{Var}(H_{0})_{h\in H_{0}}$ for $H_{0}$ a finite Heyting algebra; thus, $H \in \mathsf{Var}(H_{0})$.
    Since every algebra in a variety generated by finitely many finite algebras is locally finite \cite[Thm.~10.16]{BurrisSankappanavar}, every algebra in $\mathsf{Var}(H_{0})$---in particular, $H$---is locally finite.
\end{remark}

We will be especially concerned with the duals of homomorphisms $f\colon H_{0}\to H$ satisfying Jibladze's law. We recall from \cite{kuznetsovetale} the following notion:

\begin{definition}[Strict p-morphism]
    Let $f \colon Y\to X$ be an order-preserving map between posets. We say that $f$ is a \emph{strict p-morphism} if, for all $y\in Y$ and $x\in X$ with $f(y)\leq x$, there is a unique $y'\geq y$ such that $f(y')=x$. 
    \[
    \begin{tikzpicture}[>=stealth, scale=1.1]
        \node[circle, draw, fill=blue!15, inner sep=2pt, label=left:$y$] (y) at (0,0) {};
        \node[circle, draw, fill=blue!15, inner sep=2pt, label=left:$\exists! y'$] (yp) at (0,1.5) {};
        \draw[dashed] (y) -- (yp);
    
        \node[circle, draw, fill=red!15, inner sep=2pt, label=right:$f(y)$] (fy) at (3,0) {};
        \node[circle, draw, fill=red!15, inner sep=2pt, label=right:$x$] (x) at (3,1.5) {};
        \draw (fy) -- (x);
    
        \draw[->] (y) -- (fy);
        \draw[->, dashed] (yp) -- (x);
    \end{tikzpicture}
    \]
\end{definition}

The following is an equivalent formulation of the same concept:

\begin{remark}[Strict p-morphism $\Leftrightarrow$ order-iso on principal upsets, {\cite[Lem.~2.10]{kuznetsovetale}}]\label{rem: external characterisation of etale maps}
    An order-preserving map $f \colon Y \to X$ is a strict p-morphism if and only if, for each $y\in Y$, $f$ restricts to an order-isomorphism between ${\u}y$ and ${\u}f(y)$.
\end{remark}

\begin{definition}[\'Etale-finite Esakia space]
    We say that an Esakia space $X$ is \emph{étale-finite} if there is a continuous strict p-morphism from $X$ to some finite Esakia space.
\end{definition}

\begin{theorem}[Strictness is dual to Jibladze's law, {\cite[Cor.~4.7]{kuznetsovetale}}] \label{t:strict-dual-to-Jib}
    Let $\acute{e} \colon X \to X_0$ be a continuous p-morphism from an Esakia space $X$ to a finite Esakia space $X_0$, and let $f \colon H_0 \to H$ be the dual Heyting algebra homomorphism.
    The p-morphism $\acute{e}$ is strict if and only if $f$ satisfies Jibladze's law. 
\end{theorem}

\begin{corollary}[\'Etale-finite Esakia spaces $\leftrightarrow$ étale-finite Heyting algebras] \label{cor:etale-finite-iff-etale-finite}
    An Esakia space is étale-finite if and only if its dual Heyting algebra is étale-finite.
\end{corollary}

Although we defined étale-finite \emph{Esakia} spaces, we note that these are the same as ``étale-finite \emph{Priestley} spaces'':

\begin{lemma}[Priestley étale-finite is Esakia]\label{lem: priestley etale-finite is esakia}
    If $\acute{e}\colon X\to X_{0}$ is a continuous strict p-morphism with $X$ a Priestley space and $X_{0}$ a finite poset, then $X$ is an Esakia space.
\end{lemma}

\begin{proof}
    Since $\acute{e}$ is a continuous strict p-morphism, for each $x_{0}\in X_{0}$ the fibre $S_{x_{0}}=\acute{e}^{-1}[\{x_{0}\}]$ is clopen.
    
    Let $x_{0}\leq y_{0}$ in $X_{0}$. By strictness of the p-morphism, for every $x\in S_{x_{0}}$ there is a unique point $y\in S_{y_{0}}$ such that $x\leq y$. We denote it by $p_{x_{0},y_{0}}(x)$, thus obtaining a map $p_{x_{0},y_{0}}\colon S_{x_{0}}\to S_{y_{0}}$.
    For every $A\subseteq S_{y_{0}}$, we have
    \begin{equation} \label{eq:inverse}
        p_{x_{0},y_{0}}^{-1}[A]
        =
        S_{x_{0}}\cap {\downarrow}A.
    \end{equation}
    Indeed, $x\in S_{x_{0}}$ is sent into $A$ precisely when the unique point of $S_{y_{0}}$ above $x$ belongs to $A$; equivalently, when $x\leq a$ for some $a\in A$.
    In particular, if $A\subseteq S_{y_{0}}$ is closed, then $\d A$ is closed in $X$ (since the downward closure of a closed set is closed in a Priestley space, see e.g.\ \cite[Exercise~3.2.7.b]{Gehrke2024}); hence, the set $p_{x_{0},y_{0}}^{-1}[A] = S_{x_{0}}\cap {\downarrow}A$ is closed. Thus, $p_{x_{0},y_{0}}$ is continuous.
    
    Now let $W\subseteq X$ be clopen. Since the fibres $S_{y_{0}}$ partition $X$, we have
    \begin{equation}\label{eq:partition}
        W=\bigcup_{y_{0}\in X_{0}}(W\cap S_{y_{0}}).
    \end{equation}
    Using \eqref{eq:partition} and \eqref{eq:inverse}, we obtain
    \[
        {\downarrow}W =  \bigcup_{y_{0}\in X_{0}}\d (W\cap S_{y_{0}}) =
        \bigcup_{y_{0}\in X_{0}}
        \bigcup_{x_{0}\leq y_{0}}
        \left(S_{x_{0}}\cap {\downarrow}(W\cap S_{y_{0}})\right)
        =
        \bigcup_{y_{0}\in X_{0}}
        \bigcup_{x_{0}\leq y_{0}}
        p_{x_{0},y_{0}}^{-1}[W\cap S_{y_{0}}].
    \]
    For each $y_{0}$, the set $W\cap S_{y_{0}}$ is clopen in
    $S_{y_{0}}$. Since $p_{x_{0},y_{0}}$ is continuous, its inverse image
    is clopen in $S_{x_{0}}$, and hence clopen in $X$, because
    $S_{x_{0}}$ is clopen in $X$. Since $X_{0}$ is finite,
    ${\downarrow}W$ is a finite union of clopen subsets of $X$.
    
    Therefore, $X$ is an Esakia space.
\end{proof}

\subsection{Examples, closure properties, and non-examples}

\begin{example}[Examples of étale-finite Heyting algebras and Esakia spaces]\label{ex: etale-finite Heyting algebras}
    \hfill
    \begin{enumerate}
    
        \item (Finite Heyting algebras, finite Esakia spaces)
        All finite Heyting algebras (= finite bounded distributive lattices) and, dually, all finite Esakia spaces (= finite posets) are clearly étale-finite.
    
        \item \label{eq: Stone spaces} (Boolean algebras, Stone spaces) 
        If $X$ is a Stone space, seen as an Esakia space with the discrete order, then the unique function $X\to \bullet$ to the one-element space $\bullet$ is a continuous strict p-morphism. Therefore, every Stone space is étale-finite, and every Boolean algebra is étale-finite.
        
        \item \label{i: Alexandroff} 
        (The space $\alpha\N$ with a new isolated top, the algebra $\fincofnaturals$ with a new bottom) Consider the Alexandroff compactification $\alpha \N =\N\sqcup\{\infty\}$ of the set of natural numbers. 
        It is well known (see, e.g., \cite[Example~1.9]{Koppelberg1989-lg}) that the dual of this space is the Boolean algebra $\fincofnaturals$ of all subsets of $\N$ that are finite or cofinite. 
        Let $(\alpha \N)^{\top}=(\alpha \N)\cup \{\top\}$, where $\{\top\}$ is added as a clopen, and order the space by declaring every point in $\alpha \N$ to be below $\top$ (see Figure \ref{fig:a-fantastic-example}). 
        Then $(\alpha \N)^{\top}$ is an Esakia space; its dual Heyting algebra adds a new bottom element to $\fincofnaturals$.
        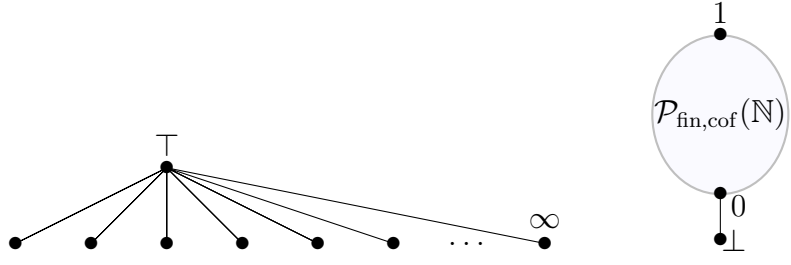
\begin{figure}[htbp]
            \centering
            \begin{tikzpicture}
                \node at (0,0) {$\bullet$};
                \node at (0,0.3) {$\top$};
                \node at (0,-1) {$\bullet$};
                \node at (-1,-1) {$\bullet$};
                \node at (-2,-1) {$\bullet$};
                \node at (1,-1) {$\bullet$};
                \node at (2,-1) {$\bullet$};
                \node at (3,-1) {$\bullet$};
                \node at (4,-1) {$\dots$};
                \node at (5,-1) {$\bullet$};
                \node at (5,-0.7) {$\infty$};
        
                \draw (5,-1) -- (0,0);
        
                \draw (0,0) -- (0,-1) -- (0,0) -- (-1,-1) -- (0,0) -- (-2,-1) -- (0,0) -- (1,-1) -- (0,0) -- (2,-1) -- (0,0) -- (3,-1);
                
            \end{tikzpicture}
            \qquad
            \begin{tikzpicture}[scale=0.6]
        
                \filldraw[color=gray!50, fill=blue!2, thick] (0,0) ellipse (1.5 and 1.75);
                
                \node at (0,0) {$\fincofnaturals$};
                \node at (0,-1.75) {$\bullet$};
                \node at (0,-2.75) {$\bullet$};
                \node at (0.4,-2) {$0$};
                \node at (0,1.75) {$\bullet$};
                \node at (0,2.25) {$1$};
                \node at (0.3,-2.75) {$\bot$};
                
                \draw (0,-2.75) -- (0,-1.75);
            
            \end{tikzpicture}      
            \caption{The space $(\alpha \N)^{\top}$ and its dual Heyting algebra.}
            \label{fig:a-fantastic-example}
        \end{figure}
        
        Let $\mathbf{2}=0<1$ be the two-element chain with the discrete topology. Consider the map $\acute{e}\colon (\alpha\N)^{\top}\to \mathbf{2}$ that sends $\top$ to $1$ and every other point to $0$. Then $\acute{e}$ is a continuous strict p-morphism. Therefore, $(\alpha \N)^\top$ is an \emph{étale-finite Esakia space}.
        
        \item More generally, let $X$ be an étale-finite Esakia space, say with $\acute{e}\colon X\to X_{0}$ a continuous strict p-morphism, and let $Y_{0}$ be a finite poset.
        We can consider $X\oplus Y_{0}$, whose underlying topological space is the disjoint union $X\sqcup Y_{0}$ of $X$ and the discrete space $Y_0$, and whose partial order is defined by letting every point of $X$ be below every point of $Y_{0}$ (see e.g., \cite[Def.~4.1.12]{bezhanishviliphdthesis}).
        Then, $X\oplus Y_{0}$ is an étale-finite Esakia space. Indeed, the map $\acute{e}^{Y_{0}}\colon X\oplus Y_{0}\to X_{0}\oplus Y_{0}$ sending $x\in X$ to $\acute{e}(x)$ and $y_{0}$ to $y_{0}$ is a continuous strict p-morphism.
        
        Dually, given an étale-finite Heyting algebra $H$ and a finite Heyting algebra $H_{0}$, their \emph{vertical sum} $H_{0}\mathbin{\overline{\oplus}} H$, defined by taking the union of the algebras and identifying $1_{H_0}$ with $0_H$ (see \cite[Def.~4.1.13]{bezhanishviliphdthesis}), is an étale-finite Heyting algebra; see Figure \ref{fig:another-fantastic-example}.
        \begin{figure}[h]
            \centering
        
            \begin{tikzpicture}
                \filldraw[color=gray!50, fill=blue!2, thick] (0,0) ellipse (3.25 and 0.25);

                \filldraw[color=gray!50, fill=blue!2, thick] (0,1.5) ellipse (2.5 and 0.25);

                \node at (0,0) {$X$};

                \node at (0,1.5) {$Y_{0}$};

                \draw[->,dotted] (0,0.3) -- (0,1.2);
                \draw[->,dotted] (-1,0.3) -- (-1,1.2);
                \draw[->,dotted] (1,0.3) -- (1,1.2);
                
            \end{tikzpicture}
            \qquad
            \begin{tikzpicture}[scale=0.5]
                \filldraw[color=gray!50, fill=blue!2, thick] (0,-3) ellipse (0.75 and 1);
                \filldraw[color=gray!50, fill=blue!2, thick] (0,0) ellipse (1.5 and 2);
                
                \node at (0,0) {$H$};
                \node at (0,-3) {$H_{0}$};
                \node at (0,-2) {$\bullet$};
                \node at (0,-4) {$\bullet$};
                \node at (0,2) {$\bullet$};
            \end{tikzpicture}
            \caption{The space $X\oplus Y_{0}$ and its dual Heyting algebra $H_{0}\mathbin{\overline{\oplus}}H$.}
            \label{fig:another-fantastic-example}
        \end{figure}
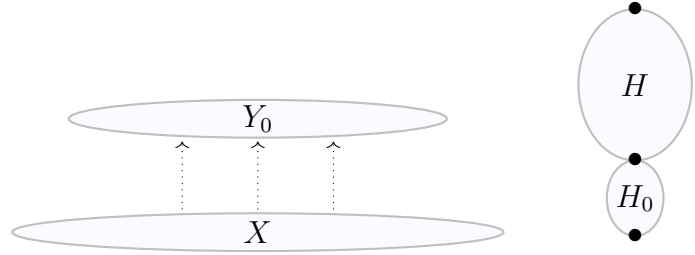

        \item (Discrete version) \label{i:discrete}
        Suppose that $\acute{e}\colon P\to P_{0}$ is a strict p-morphism from a poset $P$ to a finite poset $P_0$. We show that the inverse-image function $\acute{e}^{-1}[-]\colon \Up(P_{0})\to \Up(P)$ satisfies Jibladze's law, and hence that $\Up(P)$ is étale-finite. 
        
        First, notice that every point $p\in P$ has only finitely many successors, i.e., ${\u}p$ is finite.
        For each upset $U\in \Up(P)$, we show that, for each $p \in P$, there is an upset $U_{0}\subseteq P_{0}$ such that
        \begin{equation*}
            p\in (U\leftrightarrow \acute{e}^{-1}[U_{0}]),
        \end{equation*}
        where $\leftrightarrow$ is the biimplication in the Heyting algebra $\Up(P)$.
        Set
        \[
        U_{0}\coloneqq \acute{e}[{\u}p\cap U];
        \]
        this is an upset because ${\u}p\cap U$ is an upset and $\acute{e}$ is a p-morphism (see \cref{rem: remark on p-morphisms}). Then we show that $p\in (U\leftrightarrow \acute{e}^{-1}[U_{0}])$.
        To prove $p\in (U \rightarrow \acute{e}^{-1}[U_{0}])$, let $p' \geq p$ be such that $p'\in U$; then $p'\in {\u}p\cap U$, and so $p' \in \acute{e}^{-1}[\acute{e}[{\u}p \cap U]] = \acute{e}^{-1}[U_0]$. 
        To prove $p\in (\acute{e}^{-1}[U_{0}] \rightarrow U)$, let $p' \geq p$ be such that $p'\in \acute{e}^{-1}[U_0] = \acute{e}^{-1}[\acute{e}[{\u}p \cap U]]$; then, there is $p_0 \in {\u}p \cap U$ such that $\acute{e}(p')=\acute{e}(p_0)$; because $\acute{e}$ is a strict p-morphism, $p'=p_0$ and so $p'\in U$.
        Since $p$ was arbitrary, we have
        \[
        1_{\Up(P)} = \bigcup_{U_0 \in \Up(P_0)} U \leftrightarrow \acute{e}^{-1}[U_0].
        \]
        This shows Jibladze's law for the map $\acute{e}^{-1}[-]$.

        \item (G\"odel $\KM$-algebras of bounded depth) For the reader familiar with $\KM$-algebras, these provide a natural source of examples: all Heyting reducts of $\KM$-algebras belonging to the variety generated by a finite chain are étale-finite.
        For details, see \cref{sec: a further example}. 
    \end{enumerate}
\end{example}

\begin{lemma}[Closure properties on quotients of lattices]\label{lem: closure properties on quotients of lattices}
    \hfill
    \begin{enumerate}
        \item
        \label{i:quotient-of-finite-product}
        Let $h \colon L \to M$ be a bounded lattice homomorphism between bounded lattices, suppose that $M$ has only $0$ and $1$ as complemented elements, and $0 \neq 1$ in $M$.
        If $L$ is a finite product $L_1 \times \dots \times L_n$ of bounded distributive lattices,
        then there is $i \in \{1, \dots, n\}$ such that $h$ factors through the projection $L_1 \times \dots \times L_n \twoheadrightarrow L_{i}$.
        \[
        \begin{tikzcd}
            L_1 \times \dots \times L_n \arrow[two heads]{d}{} \arrow{r}{h}& M\\
            L_i \arrow[dashed]{ru}
        \end{tikzcd}
        \]

        \item \label{eq: quotient of product of chain} 
        Let $h \colon L \twoheadrightarrow M$ be a surjective bounded lattice homomorphism between bounded lattices, and suppose that $M$ has only $0$ and $1$ as complemented elements.
        If $L$ is an arbitrary product of chains, then $M$ is a chain.
        
        \item \label{eq: complete factor} Suppose that $h \colon L \twoheadrightarrow M$ is a surjective lattice homomorphism between lattices, that $L$ is complete and that $M$ is countable.
        Then $M$ is complete.

    \end{enumerate}
\end{lemma}
\begin{proof}
    \eqref{i:quotient-of-finite-product}
    Let $h\colon L\rightarrow M$ be a lattice homomorphism.  Consider for each $1\leq i\leq n$ the element $\chi_{i}\in L_{1}\times\dots\times L_{n}$ where $\chi_{i}(j)=1$ if and only if $i=j$, and otherwise $\chi_{i}(j)=0$. Since $\chi_{i}$ is complemented and bounded lattice homomorphisms preserve complemented elements, the element $h(\chi_{i})$ is complemented. Hence, by assumption, either $h(\chi_{i})=1$ or $h(\chi_{i})=0$.

    Note that there can be at most one $i$ such that $h(\chi_{i})=1$, since, if $i\neq j$, then $\chi_{i}\wedge\chi_{j}=0$ and so $h(\chi_{i}) \wedge h(\chi_{j})=0 \neq 1 \wedge 1$. Also there must be at least one $i$ such that $h(\chi_{i})=1$, since $\chi_{1}\vee\dots\vee \chi_{n}=1$ and so $h(\chi_{1})\vee\dots\vee h(\chi_{n}) = 1 \neq 0 \lor \dots \lor 0$. Thus, let $i$ be the unique element of $\{1, \dots, n\}$ such that $h(\chi_{i})=1$.    Consider the kernel congruence $\mathsf{Ker}(h)=\{(a,b)\in L\times L : h(a)=h(b)\}$. Then we claim that $ \mathsf{Ker}(\pi_{i})\subseteq \mathsf{Ker}(h)$.
    Let $a, b \in L$ be such that $\pi_i(a) = \pi_i(b)$.
    For every $c \in L$ we have
    \[
    c = \bigvee_{j = 1}^n \chi_j \land c.
    \]
    From this, we can deduce
    \begin{gather*}
        h(a) = h\biggl(\bigvee_{j = 1}^n \chi_j \land a\biggr)   = \bigvee_{j = 1}^n h(\chi_j) \land h(a)= h(\chi_i) \land h(a)    = h(\chi_i \land a) \\
        = h(\chi_i \land b)= h(\chi_i) \land h(b) = \bigvee_{j = 1}^n h(\chi_j) \land h(b)= h\biggl(\bigvee_{j = 1}^n \chi_j \land b\biggr)= h(b).
    \end{gather*}
    Since the projection $\pi_{i} \colon L_1 \times \dots \times L_n \to L_i$ is surjective (since each $L_j$ is nonempty) and $\mathsf{Ker}(\pi_{i})\subseteq\mathsf{Ker}(h)$, $h$ factors through $\pi_{i}$, as desired.

    \eqref{eq: quotient of product of chain} 
    Let $L=\prod_{i\in I}D_{i}$, and let $h\colon L\twoheadrightarrow M$ be a bounded lattice surjection. Let $(a_i)_{i\in I},(b_i)_{i\in I}\in L$, and let $J \coloneqq \{ i\in I : a_i\leq b_i\}$, and $(e_i)_{i\in I}\in L$ such that $e_j=1$ for $j\in J$ and $e_j=0$ if $j\notin J$. Note that then
    \begin{equation}
        \label{eq:1-equation}
        (a_{i})_{i\in I}\wedge (e_{i})_{i\in I} \leq (b_{i})_{i\in I}
    \end{equation}
    since for $i\in J$ we have $a_{i}\leq b_{i}$, and for $i\notin J$ we have $a_{i}\wedge e_{i}=0$. Moreover,  $(e_i)_{i\in I}$ is complemented, and so $h((e_i)_{i\in I})$ is complemented. Then, $h((e_i)_{i\in I})=1$ or $h((e_i)_{i\in I})=0$. If $h((e_i)_{i\in I})=1$, note that 
    \[
        h((a_i)_{i\in I})
        =
        h((a_i)_{i\in I})\wedge  h((e_i)_{i\in I})=h ((a_i\wedge e_i)_{i\in I})\leq h((b_i)_{i \in I}),
    \]
    where the first equality follows from $h((e_{i})_{i\in I})=1$, and the inequality from \eqref{eq:1-equation}. Therefore, $h((a_i)_{i \in I})\leq h((b_i)_{i \in I})$. Similarly, if  $h((e_i)_{i\in I})=0$, then $h(\lnot (e_i)_{i \in I}) = 1$, and then $h((b_i)_{i\in I})=h((b_i)_{i\in I})\wedge h(\neg (e_i)_{i\in I})=h ((b_i\wedge \neg e_i)_{i\in I})\leq h((a_i)_{i \in I})$. Hence, any two elements of $M$ are comparable, and so $M$ is a chain.

    \eqref{eq: complete factor} It is known that, if $h \colon L \twoheadrightarrow M$ is a surjective lattice homomorphism between lattices and $M$ is countable, then $h$ has a section as posets, i.e., there is an order-preserving function $i \colon M \hookrightarrow L$ such that the composite $M \xrightarrow{i} L \xrightarrow{h} M$ is the identity (see the ``weaving map'' in \cite[p.~98]{Rival1982}, which goes back to \cite[Thm.~3, p.~35]{Dean1961}); it follows that, if moreover $L$ is complete, then $M$ is complete as well.
\end{proof}

\begin{remark}[$\{$étale-finite$\}\not\subseteq \mathbb{H}\mathbb{P}(\{$Boolean$\}\cup\{$complete$\}\cup\{$chains$\})$]\label{rem: etale finite not in other classes}
    We show that there are étale-finite Heyting algebras which lie outside the class of Heyting algebras that are already known to be topos-admissible; in particular,
    we show that the étale-finite Heyting algebra $H$ from \cref{ex: etale-finite Heyting algebras}\eqref{i: Alexandroff}
    \[
        \begin{tikzpicture}[scale=0.6]
        
            \filldraw[color=gray!50, fill=blue!2, thick] (0,0) ellipse (1.5 and 1.75);
            
            \node at (0,0) {$\fincofnaturals$};
            \node at (0,-1.75) {$\bullet$};
            \node at (0,-2.75) {$\bullet$};
            \node at (0.4,-2) {$0$};
            \node at (0,1.75) {$\bullet$};
            \node at (0,2.25) {$1$};
            \node at (0.3,-2.75) {$\bot$};
            
            \draw (0,-2.75) -- (0,-1.75);
            
        \end{tikzpicture}
    \]
    is not in the closure under products and homomorphic images of the class containing all Boolean algebras, all complete Heyting algebras and all chains (including $[0,1] \cap \Q$, of course).
    This will show that our main result (\cref{thm: elementary topos from etale-finite Heyting algebra} below, which produces a topos for each étale-finite Heyting algebra) strictly extends the class of algebras known to be topos-admissible.

    We shall prove that $H$ is not a homomorphic image of a product $B \times C \times P$ of a Boolean algebra $B$, a complete Heyting algebra $C$ and an arbitrary product $P$ of chains.

    Towards a contradiction, suppose that there is a surjective Heyting homomorphism $p \colon B \times C\times P \twoheadrightarrow H$, where $B$ is a Boolean algebra, $C$ is a complete Heyting algebra, and $P$ is a product of chains.

    Notice that $H$ only has $1$ and $0$ as its complemented elements, and $0 \neq 1$. Thus, by \cref{lem: closure properties on quotients of lattices}\eqref{i:quotient-of-finite-product}, $H$ is a lattice homomorphic image of one of $B$, $C$ and $P$.  
    Since the projection preserves all existing algebraic structure, it is in fact a Heyting homomorphic image of one of $B$, $C$ and $P$.
    
    $H$ is not a homomorphic image of $B$ since $H$ is not Boolean. If it were a homomorphic image of $P$, then by \cref{lem: closure properties on quotients of lattices}\eqref{eq: quotient of product of chain} $H$ would be a chain, which is false. Finally, if $H$ were a homomorphic image of $C$, by \cref{lem: closure properties on quotients of lattices}\eqref{eq: complete factor} it would be complete, which is likewise false.
\end{remark}

\begin{remark}[Closure properties of étale-finite Heyting algebras / Esakia spaces] \label{rem:closure-properties}
    \hfill
    \begin{enumerate}
        \item \label{eq: closure properties of etale-finite} 
        It follows from \cref{prop: characterisation of etale maps} that, for a fixed finite Heyting algebra $H_0$, $H_0$-étale Heyting algebras are closed under arbitrary products and homomorphic images, as well as under subalgebras containing the image of $H_0$.
        Dually, a closed upset of an étale-finite Esakia space is an étale-finite Esakia space, and, for $X_{0}$ a fixed finite poset, an arbitrary coproduct (in $\Esa$, which is computed as in $\Pries$) of $X_{0}$-étale Esakia spaces is $X_{0}$-étale.

        \item Given two étale-finite Esakia spaces $X$ and $Y$, their topological and order-theoretic disjoint union $X\sqcup Y$ is étale-finite.
        Indeed, if $\acute{e}\colon X\to X_{0}$ and $\acute{e}'\colon Y\to Y_{0}$ are continuous strict p-morphisms to finite Esakia spaces, then the sum $\acute{e} \sqcup \acute{e}'\colon X\sqcup Y\to X_{0}\sqcup Y_{0}$ is a continuous strict p-morphism. Dually, this shows that the finite product of étale-finite Heyting algebras is étale-finite.
        
        \item 
        Given two étale-finite Esakia spaces $X$ and $Y$, their product $X \times Y$ computed in $\Pries$ is étale-finite.
        Indeed, let $\acute{e}\colon X\to X_{0}$ and $\acute{e}'\colon Y\to Y_{0}$ be continuous strict p-morphisms to finite Esakia spaces.
        First, $X\times Y$ is an Esakia space by \cref{properties.pri}. Second, the map $(\acute{e}\times \acute{e}')\colon X\times Y\to X_{0}\times Y_{0}$ is continuous and order-preserving by general facts. Finally, we note that it is a strict p-morphism: if $(\acute{e}\times \acute{e}')(x,y)\leq (x_{0},y_{0})$ then $\acute{e}(x)\leq x_{0}$, and so there is a unique $x'\geq x$ such that $\acute{e}(x')=x_{0}$, and similarly a unique $y'\geq y$ such that $\acute{e}'(y')=y_{0}$. Then $(x',y')\geq (x,y)$ is the unique pair such that $(\acute{e}\times \acute{e}')(x',y')=(x_{0},y_{0})$.

    \end{enumerate}
    
\end{remark}

In general, an \emph{infinite} product of étale-finite Heyting algebras need \emph{not} be étale-finite.
To see this, we first record the following facts.

\begin{lemma}[Jibladze's law + join-irreducible top $\Rightarrow$ finite]\label{l:join-irreducible}
    Let $H$ be a Heyting algebra whose top element $1$ is finitely join-irreducible.
    Let $S$ be a finite subset of $H$ such that, for every $x \in H$,
    \begin{equation} \label{i:join-is-1}
        \bigvee_{x_0\in S}(x\leftrightarrow x_0)=1.
    \end{equation}
    Then, $S = H$ (and hence $H$ is finite).
\end{lemma}

\begin{proof}
    The inclusion $S \subseteq H$ is obvious, so let us prove $H \subseteq S$. 
    Let $x \in H$. Since $1$ is finitely join-irreducible, from \eqref{i:join-is-1} we deduce that there is $x_0 \in S$ such that $x \leftrightarrow x_0 = 1$, i.e., such that $x = x_0$.
    Then, $x = x_0 \in S$.
\end{proof}

\begin{corollary}[\'Etale-finite + join-irreducible top $\Rightarrow$ finite] \label{c:finitely join irreducible}
    Every étale-finite Heyting algebra with a finitely join-irreducible top element is finite.
\end{corollary}

\begin{proof}
    This follows immediately from \cref{r:etale-finite-basic}\eqref{i: finite algebra condition} and \cref{l:join-irreducible}.
\end{proof}

We can now prove the claim made before \cref{l:join-irreducible}.

\begin{remark}[Infinite product of étale-finite Heyting algebras need not be étale-finite]\label{rem: infinite product of etale may fail to be etale}
    The following provides an example of an \emph{infinite} product of étale-finite Heyting algebras that is not étale-finite.
    For each $n \in \N \smallsetminus \{0\}$, let $[n] \coloneqq 1< \dots < n$ be the $n$-element linearly ordered Heyting algebra, and let $\prod_{n\in \N\smallsetminus\{0\}}[n]$ be the product of these chains. Towards a contradiction, suppose that $H_0$ is a finite Heyting algebra and $f \colon H_{0}\to \prod_{n\in \N\smallsetminus\{0\}}[n]$ a homomorphism satisfying Jibladze's law. Note that, for each $n\in \N \smallsetminus \{0\}$, letting $\pi_{n}\colon \prod_{m \in \N \smallsetminus \{0\}}[m]\to [n]$ denote the projection, for each $x\in [n]$ we have
    \begin{equation*}
         \bigvee_{h\in H_{0}}(x\leftrightarrow \pi_{n}f(h)) = 1. 
    \end{equation*}
    For every $n \geq 2$, the chain $[n]$ has a finitely join-irreducible top element and so, by \cref{l:join-irreducible},
    \[
    n = \lvert [n] \rvert \leq \lvert \pi_{n}f[H_{0}] \rvert \leq \lvert H_{0} \rvert.
    \]
    Since this holds for every $n \geq 2$, we have a contradiction.
    Thus, $\prod_{n\in \N\smallsetminus \{0\}}[n]$ is not étale-finite.
\end{remark}

We provide some non-examples of étale-finite Heyting algebras and Esakia spaces.

\begin{example}[Non-examples of étale-finite Heyting algebras and Esakia spaces]\label{ex: non-example}
    \hfill
    \begin{enumerate}
        \item \label{i:chains} (Infinite bounded chains)
        Every infinite bounded chain (such as $\omega + 1$, $1 + \omega^\op$, or $\omega + \omega^\op$) is not étale-finite; this follows immediately from \cref{c:finitely join irreducible}, since the top element of an infinite bounded chain is finitely join-irreducible.
        \newcommand{\chainvdots}{%
            \mathord{\vcenter{%
            \offinterlineskip
            \hbox to .6em{\hss$\cdot$\hss}\kern0.15ex
            \hbox to .6em{\hss$\cdot$\hss}\kern0.15ex
            \hbox to .6em{\hss$\cdot$\hss}%
          }}%
        }
        \[
        \begin{array}{ccc}
            \begin{tikzcd}[row sep=0.38em, arrows={-}, cells={nodes={inner sep=1.2pt}}]
                \bullet \\
                \chainvdots \arrow[d] \\
                \bullet \arrow[d] \\
                \bullet \arrow[d] \\
                \bullet
            \end{tikzcd}
            &
            \begin{tikzcd}[row sep=0.38em, arrows={-}, cells={nodes={inner sep=1.2pt}}]
                \bullet \arrow[d] \\
                \bullet \arrow[d] \\
                \bullet \arrow[d] \\
                \chainvdots \\
                \bullet
            \end{tikzcd}
            &
            \begin{tikzcd}[row sep=0.38em, arrows={-}, cells={nodes={inner sep=1.2pt}}]
                \bullet \arrow[d] \\
                \bullet \arrow[d] \\
                \bullet \arrow[d] \\
                \chainvdots \\
                \chainvdots \arrow[d] \\
                \bullet \arrow[d] \\
                \bullet \arrow[d] \\
                \bullet
            \end{tikzcd}
            \\[0.8em]
            \omega+1
            &
            1+\omega^\op
            &
            \omega+\omega^\op
        \end{array}
        \]
        Dually, every infinite linearly ordered Esakia space (such as $1 + \omega^\op$, $\omega + 1$, or $\omega + 1 + \omega^\op$) is not étale-finite. 
        \[
        \begin{array}{ccc}
            \begin{tikzcd}[row sep=0.38em, arrows={-}, cells={nodes={inner sep=1.2pt}}]
                \bullet \arrow[d] \\
                \bullet \arrow[d] \\
                \bullet \arrow[d] \\
                \chainvdots \\
                \bullet
            \end{tikzcd}
            &
            \begin{tikzcd}[row sep=0.38em, arrows={-}, cells={nodes={inner sep=1.2pt}}]
                \bullet \\
                \chainvdots \arrow[d] \\
                \bullet \arrow[d] \\
                \bullet \arrow[d] \\
                \bullet
            \end{tikzcd}
            &
            \begin{tikzcd}[row sep=0.38em, arrows={-}, cells={nodes={inner sep=1.2pt}}]
                \bullet \arrow[d] \\
                \bullet \arrow[d] \\
                \bullet \arrow[d] \\
                \chainvdots \\
                \bullet \\
                \chainvdots \arrow[d] \\
                \bullet \arrow[d] \\
                \bullet \arrow[d] \\
                \bullet
            \end{tikzcd}
            \\[0.8em]
            1+\omega^\op
            &
            \omega+1
            &
            \omega+1+\omega^\op
        \end{array}
        \]

        \item (Infinite Heyting algebras with greatest non-top element)\label{i:join-irreducible-top}
        Every infinite Heyting algebra with a greatest non-top element is not étale-finite.
        This follows immediately from \cref{c:finitely join irreducible}.
        So, for every infinite Heyting algebra $H$, the Heyting algebra obtained by adjoining a new top element above the original top element of $H$ is not étale-finite.
        \[
        \begin{tikzpicture}[scale=0.6]
            \filldraw[color=gray!50, fill=blue!2, thick] (0,0) ellipse (1.5 and 1.75);
        
            \node[align=center] at (0,0) {\scriptsize infinite\\[-0.2ex]\scriptsize Heyt.\ alg.};
        
            \node at (0,-1.75) {$\bullet$};
            \node at (0,1.75) {$\bullet$};
            \node at (0,2.75) {$\bullet$};
        
            \draw (0,1.75) -- (0,2.75);
        \end{tikzpicture}
        \]

        It follows that any infinite subdirectly irreducible Heyting algebra is not étale-finite, since every subdirectly irreducible Heyting algebra has a greatest non-top element (\cite[Prop.~A.1.1]{Esakiach2019HeyAlg}).
                    
        \item \label{ex: cofinite subsets} (The space $\alpha\N$ with $\infty$ as top, and its dual algebra $\cofnaturals$)
        The Esakia space $X = \alpha\mathbb{N}$, with the accumulation point $\infty$ as its top element and $\N$ discretely ordered, is not étale-finite. 
        \[
            \begin{tikzpicture}
                \node at (5,-1) {$\bullet$};
                \node at (6,-1) {$\bullet$};
                \node at (7,-1) {$\bullet$};
                \node at (7,0) {$\bullet$};
                \node at (8,-1) {$\bullet$};
                \node at (9,-1) {$\bullet$};
                \node at (10,-1) {$\dots$};
                \node at (7,0.3) {$\infty$};
                
                \draw (7,0) -- (5,-1) -- (7,0) -- (6,-1) -- (7,0) -- (7,-1) -- (7,0) -- (8,-1) -- (7,0) -- (9,-1);
            \end{tikzpicture}
        \]
        Indeed, there is no continuous strict p-morphism to a finite poset, since otherwise $\{\infty\}$ would be the preimage of the set of maximal elements of the finite poset and hence $\{\infty\}$ would be clopen, a contradiction.
        Dually, the Heyting algebra $\cofnaturals$ is not étale-finite.

        \item (Free Heyting algebras over nonempty sets)
        The free Heyting algebra over a nonempty set is not étale-finite.
        To prove this, it suffices to show that the free Heyting algebra $\mathcal{F}_{\HA}(1)$ on one generator is not étale-finite, because, by \cref{rem:closure-properties}\eqref{eq: closure properties of etale-finite}, étale-finite Heyting algebras are closed under surjective homomorphisms. The Heyting algebra $\mathcal{F}_{\HA}(1)$ is the \emph{Rieger--Nishimura lattice}, which is well known to be infinite. By \cref{rem: local finiteness from etale-finiteness}, if $\mathcal{F}_{\HA}(1)$ were étale-finite, it would be locally finite---a contradiction since it is $1$-generated, yet infinite.               
    \end{enumerate}
\end{example}

For the interested reader, in \cref{s:app-char} we characterise étale-finite Esakia spaces as those Esakia spaces where principal upsets are finite and locally constant; see \cref{p:finite-principal-upset-types}.

\section{Spectral local homeomorphisms} \label{sec:speclochomeo}

\begin{definition}[Spectral topology] \label{def:spectral-topology}
    The \emph{spectral topology} on a Priestley space is the topology of all open upsets.
\end{definition}

The clopen upsets of a Priestley space form a basis for its spectral topology \cite[Lem.~2.1]{W1975}. By default, continuity refers to continuity in the Priestley topology.

Recall that a function $f \colon Y \to X$ between topological spaces is said to be a \emph{local homeomorphism} if, for each $y\in Y$, there is an open neighbourhood $A$ of $y$ such that $f[A]$ is open, and the restriction $f{\restriction}_{A}\colon A\to f[A]$ of $f$ is a homeomorphism. It is well known that this is equivalent to being a continuous open map that is \emph{locally injective}, i.e., such that for each $y\in Y$ there is an open neighbourhood $A$ of $y$ on which $f$ is injective.

\begin{definition}[Spectral local homeomorphism]
    A \emph{spectral local homeomorphism} $f\colon Y\to X$ between Priestley spaces is a local homeomorphism with respect to the spectral topologies.
\end{definition}

In \cref{thm.priest.loc} below, we will characterise spectral local homeomorphisms purely in terms of the Priestley orders and topologies.

For convenience, we recall the following well-known fact.

\begin{lemma}[Spectral continuity $\Rightarrow$ order-preservation]\label{monot.spectral}
    A function between Priestley spaces that is continuous with respect to the spectral topologies (in particular, any spectral local homeomorphism) is order-preserving.
\end{lemma}

\begin{proof}
    Let $f \colon Y \to X$ be one such map.
    Let $y,y' \in Y$ be such that $f(y) \nleq f(y')$, and let us prove $y \not\leq y'$.
    By the Priestley separation axiom, there is a clopen upset $U$ such that $f(y)\in U$ and $f(y')\not\in U$. 
    Since $f$ is continuous with respect to the spectral topologies, $f^{-1}[U]$ is an open upset. 
    Combining the facts that $f^{-1}[U]$ is an upset and that $y \in f^{-1}[U]$ and $y' \notin f^{-1}[U]$, we obtain $y \nleq y'$.
\end{proof}

It is likewise known that the clopen upsets in a Priestley space are exactly the compact opens in its spectral topology \cite[Lem.~2.1]{W1975}. Moreover, a function between Priestley spaces is continuous and order-preserving (i.e., a morphism in $\Pries$) if and only if it is a spectral map (i.e., a map $f \colon X \to Y$ such that, for every compact open $U$ in the spectral topology of $Y$, $f^{-1}[U]$ is compact open in the spectral topology of $X$). As an immediate consequence, we obtain the following characterisation of homeomorphisms.

\begin{lemma}[Spectral homeomorphisms $\Leftrightarrow$ Priestley order-homeomorphisms]\label{homeo.spectral}
    The following are equivalent for a function $f \colon Y \to X$ between Priestley spaces.
    \begin{enumerate}
    
        \item\label{homeo1} 
        $f$ is an order-homeomorphism (with respect to the Priestley orders and topologies);
        
        \item\label{homeo2}
        $f$ is a homeomorphism with respect to the spectral topologies.
        
    \end{enumerate}
\end{lemma}

Recall from \cref{properties.pri} that a closed set of a Priestley space is again a Priestley space in the induced order and topology.

\begin{lemma}[Spectral topology on closed subspaces] \label{restriction.priest}
    Let $C$ be a closed subspace of a Priestley space $X$. The spectral topology of $C$ coincides with the subspace topology obtained by restricting to $C$ the spectral topology of $X$.
\end{lemma}

\begin{proof}
    Clearly, the restriction to $C$ of a spectral open of $X$ is a spectral open of $C$. Let us show the reverse inclusion. Let $U$ be a clopen upset of $C$. For each $x\in U$ and $y\in C\smallsetminus U$, we have $x\not\leq y$ and, by the Priestley separation property in $X$, there is a clopen upset $U_{xy}$ of $X$ such that $x\in U_{xy}$ and $y\not\in U_{xy}$. Then, for each $y\in C\smallsetminus U$, we have $U\subseteq \bigcup_{x\in U}U_{xy}$; by compactness of $U$, there is a finite subcover $U\subseteq \bigcup_{i=1}^{n_y}U_{x_i y}$. We set $V_y\coloneqq  \bigcup_{i=1}^{n_y}U_{x_i y}$; then  $C\smallsetminus U \subseteq\bigcup_{y\in C\smallsetminus U} (C\smallsetminus V_y)$. Since each $C\smallsetminus V_y$ is open in $C$, by compactness there is a finite cover $C\smallsetminus U \subseteq\bigcup_{j=1}^m (C\smallsetminus V_{y_{j}})$. Let $W\coloneqq \bigcap_{j=1}^{m} V_{y_{j}}$. Then $W$ is a clopen upset of $X$ such that $W\cap C= U$. 
\end{proof}

\begin{lemma}[Local injectivity near a closed set]\label{lem:local-inj-neighbourhood-closed-set}
    Let $f \colon Y \to X$ be a local homeomorphism, where $Y$ is compact Hausdorff and
    $X$ is Hausdorff. Let $K \subseteq Y$ be closed. If $f{\restriction}_{K}$ is injective,
    then there is an open neighbourhood $U\supseteq K$ such that
    $f{\restriction}_{U}$ is injective.
\end{lemma}
\begin{proof}
    Set
    \[
    R \coloneqq \{(y_1,y_2)\in Y\times Y : f(y_1)=f(y_2)\}
          = (f\times f)^{-1}(\Delta_X).
    \]
    Since $X$ is Hausdorff, $\Delta_X$ is closed, and hence $R$ is closed in $Y\times Y$. Let
    \[
    E \coloneqq R \smallsetminus \Delta_Y
      = \{(y_1,y_2)\in Y\times Y : f(y_1)=f(y_2),\ y_1\neq y_2\}.
    \]
    Because $f$ is a local homeomorphism, it is locally injective. Hence $\Delta_Y$ is open in
    $R$: for every $y\in Y$ there is an open neighbourhood $W_y$ of $y$ such that
    \[
    R \cap (W_y\times W_y)=\Delta_Y \cap (W_y\times W_y).
    \]
    Therefore, $E=R\smallsetminus\Delta_Y$ is closed in $R$, and hence closed in $Y\times Y$.
    
    Since $f{\restriction}_K$ is injective, we have
    \[
    E \cap (K\times K)=\varnothing.
    \]
    
    By Wallace's theorem, also known as the ``Tube Lemma'' \cite[Lem.~26.8]{munkres2000topology} (see also \cite[\href{https://stacks.math.columbia.edu/tag/005N}{Tag 005N}]{stacks-project}), we have open neighbourhoods $U_1$ and $U_2$ of $K$ such that $U_1\times U_2\subseteq (Y\times Y)\smallsetminus E$; thus, by taking $U\coloneqq U_{1}\cap U_{2}$, we obtain an open neighbourhood $U$ of $K$ such that  $(U\times U)\cap E=\varnothing$.
    
    Now if $u,v\in U$ and $f(u)=f(v)$, then $(u,v)\in R$. Since $(u,v)\notin E$, it follows
    that $u=v$. Thus $f{\restriction}_U$ is injective.
\end{proof}

The following characterises spectral local homeomorphisms purely in terms of Priestley topology and order.

\begin{theorem}[Characterisations of spectral local homeomorphisms] \label{thm.priest.loc}
    Let $f \colon Y \to X$ be a function between Priestley spaces. The following are equivalent:
    \begin{enumerate}
    
        \item \label{i:spectral-local}
        $f$ is a spectral local homeomorphism; 
    
        \item \label{i:in-terms-of-Priestley}
        for every $y\in Y$, there is a clopen upset $V\subseteq Y$ with $y\in V$ such that $f[V]$ is a clopen upset of $X$ and 
        \[
        f{\restriction}_V \colon V \longrightarrow f[V]
        \]
        is an order-homeomorphism (with respect to the Priestley order and topology);
        
        \item \label{i: Local homeomorphism decomposition} there are $n \in \N$ and clopen upsets $U_{1},\dots,U_{n}$ of $Y$ such that $Y=U_{1}\cup\dots\cup U_{n}$ and, for each $i \in \{1, \dots, n\}$, $f[U_i]$ is a clopen upset and $f {\restriction}_{U_i} \colon U_i \to f[U_i]$ is an order-homeomorphism;
        
        \item \label{i:nhb-spectral-lh-versus-strict-pm2} 
        $f$ is a strict p-morphism with respect to the Priestley orders and a local homeomorphism with respect to the Priestley topologies.
    \end{enumerate}
\end{theorem}
\begin{proof} 
    \eqref{i:spectral-local}$\Rightarrow$\eqref{i:in-terms-of-Priestley}
    Let $y\in Y$. By \eqref{i:spectral-local}, there is an open upset $A\subseteq Y$ with $y\in A$ such that $f[A]$ is an open upset in $X$ and $ f{\restriction}_A \colon A \to f[A] $ is a homeomorphism with respect to the spectral topologies. Since clopen upsets form a basis for the spectral topology, there is a clopen upset $V$ with $y\in V\subseteq A$. Then, $V$ is compact and open in the spectral topology on $Y$, and so $f[V]$ is compact by continuity of $f$ with respect to the spectral topologies; since $V$ is open in the subspace $A$, the set $f[V]$ is open in $f[A]$, and hence spectral open in $X$. Thus $f[V]$ is compact and open in the spectral topology of $X$, and hence a clopen upset. Moreover, the restriction $ f{\restriction}_V \colon V \to f[V] $ is a homeomorphism with respect to the spectral topologies. Now, since $V$ and $f[V]$ are clopen, they are closed Priestley subspaces, and so \eqref{i:in-terms-of-Priestley} follows easily from  \cref{homeo.spectral,restriction.priest}.

    \eqref{i:in-terms-of-Priestley}$\Rightarrow$\eqref{i: Local homeomorphism decomposition} By \eqref{i:in-terms-of-Priestley}, for each $y\in Y$ there is a clopen upset $U_{y}$ containing $y$ such that $f[U_y]$ is a clopen upset and $f$ restricts to an order-homeomorphism on $U_{y}$. Then, $Y = \bigcup_{y\in Y}U_{y}$. By compactness, we get a finite cover, which shows \eqref{i: Local homeomorphism decomposition}. 

    \eqref{i: Local homeomorphism decomposition}$\Rightarrow$\eqref{i:nhb-spectral-lh-versus-strict-pm2} The fact that $f$ is a local homeomorphism with respect to the Priestley topologies is immediate. Let us prove that $f$ is a strict p-morphism.
    If $y\in U_i$ and $f(y)\le x$, then $x\in f[U_i]$ because $f[U_i]$ is an upset; the unique preimage of $x$ under $f{\restriction}_{U_i}$ is the candidate $y'\ge y$; the uniqueness follows from the fact that $U_i$ is an upset, and so every $z\ge y$ is still in $U_i$.

    \eqref{i:nhb-spectral-lh-versus-strict-pm2}$\Rightarrow$\eqref{i:spectral-local} Fix $y\in Y$. Since $Y$ is a Priestley space, the principal upset ${\u} y$ is closed in $Y$ \cite[Prop.~2.6]{Priestley1984}. Since $f$ is a strict p-morphism, by \cref{rem: external characterisation of etale maps}, the restriction
    \[
    f{\restriction}_{{\u} y}\colon {\u} y \longrightarrow {\u} f(y)
    \]
    is an order-isomorphism; in particular, it is injective. By
    \cref{lem:local-inj-neighbourhood-closed-set}, there is an open neighbourhood $U$ of
    ${\u}y$ such that $f{\restriction}_{U}$ is injective. 
    
    Using a standard fact for Priestley spaces \cite[Lem.~11.21]{Davey2002-lr}, choose an open upward-closed neighbourhood $V$ of
    $y$ such that
    \[
    {\u} y \subseteq V \subseteq U.
    \]
    Then $f{\restriction}_V$ is still injective. Since $f$ is a local homeomorphism for the Priestley topologies, it is an open map. Hence, $f[V]$ is open in $X$, and
    \[
    f{\restriction}_{V} \colon V \longrightarrow f[V]
    \]
    is a homeomorphism.
    
    Since $V$ is upwards closed and $f$ is a p-morphism, the set $f[V]$ is an upset (see \cref{rem: remark on p-morphisms}). We show that $f$ moreover reflects the order. Let $v,w\in V$, and assume that $f(v)\leq f(w)$. Since $f(w)\in {\u} f(v)$ and
    \[
    f{\restriction}_{{\u}v}\colon {\u} v \longrightarrow {\u} f(v)
    \]
    is surjective, there is $u\in {\u} v$ such that $f(u)=f(w)$. As $V$ is upwards closed and
    $v\in V$, we get $u\in V$. Since $f{\restriction}_{V}$ is injective, it follows that $u=w$,
    and hence $v\leq w$.
    
    Thus, for all $v,w\in V$,
    \[
    v\leq w \quad\Longleftrightarrow\quad f(v)\leq f(w),
    \]
    and so $f{\restriction}_V$ is an order-homeomorphism onto the open upset $f[V]$.
    Since $y\in Y$ was arbitrary, $f$ is a spectral local homeomorphism.
\end{proof}

\begin{corollary}[Uniform finite fibres over principal upsets]\label{cor:finiteness}
    Let $f \colon Y\to X$ be a spectral local homeomorphism between Priestley spaces.
    If $X$ is an étale-finite Esakia space, then there is $m \in \N$ such that, for each $x\in X$, $f^{-1}[{\u}x]$ is a finite set of cardinality at most $m$.
\end{corollary}

\begin{proof}
    Choose a continuous strict p-morphism $\acute e\colon X\to X_0$ with $X_0$ finite.
    By the implication \eqref{i:spectral-local}$\Rightarrow$\eqref{i: Local homeomorphism decomposition} in \cref{thm.priest.loc}, there are $n \in \N$ and clopen upsets $U_{1},\dots,U_{n}$ of $Y$ such that $Y=U_{1}\cup\dots\cup U_{n}$ and, for each $i \in \{1, \dots, n\}$, $f[U_i]$ is a clopen upset and $f {\restriction}_{U_i} \colon U_i \to f[U_i]$ is an order-homeomorphism.
    Then, for each $x \in X$, the set $f^{-1}[{\u}x]$ has at most $n \lvert X_{0} \rvert$ elements.
\end{proof}

The following proposition shows that the existence of a spectral local homeomorphism to an Esakia space forces the domain to be Esakia.

\begin{proposition}[\'Etale over Esakia is Esakia]\label{prop: closure of Esakia under spectral}
    Let $f\colon Y\to X$ be a spectral local homeomorphism, where $Y$ is a Priestley space and $X$ is an Esakia space. Then $Y$ is also an Esakia space.
\end{proposition}
\begin{proof}
    By \cref{thm.priest.loc}\eqref{i: Local homeomorphism decomposition}, there are $n \in \N$ and clopen upsets $U_{1},\dots,U_{n}$ of $Y$ such that $Y=U_{1}\cup\dots\cup U_{n}$ and, for each $i \in \{1, \dots, n\}$, $f[U_i]$ is a clopen upset and $f {\restriction}_{U_i} \colon U_i \to f[U_i]$ is an order-homeomorphism.
    Since a closed upset of an Esakia space is an Esakia space (\cref{properties.pri}\eqref{properties.pri2}), $f[U_{i}]$ is an Esakia space; therefore, $U_{i}$ is also an Esakia space. Let $U_{1}\sqcup \dots \sqcup U_{n}$ be the coproduct of the $U_{i}$, whose elements are of the form $(x,i)$ for $x\in U_{i}$. We thus have a natural map
    \begin{equation*}
        q\colon U_{1}\sqcup \dots \sqcup U_{n} \twoheadrightarrow Y,
    \end{equation*}
    which sends $(x,i)$ to $x$. For each clopen $W\subseteq Y$, we have $q^{-1}[W]=\{(x,i) : x\in W\}=\bigsqcup_{i=1}^{n}(W\cap U_{i})$, which is clopen; thus $q$ is continuous. Moreover, $q$ is obviously surjective and a p-morphism. Thus $Y$ is an Esakia space by \cref{properties.pri}\eqref{properties.pri3}.
\end{proof}

To conclude this section, recall that the category $\TopLH$ of topological spaces and local homeomorphisms has an interesting property with respect to the ambient category $\Top$ of topological spaces and continuous functions (see e.g.\ \cite[Lem.~6.4.5]{janelidzeborceaux}):

\begin{definition}[Triangle property]
    Let $\cat{C}$ be a category.
    We say that a (not necessarily full) subcategory $\cat{L}$ of $\cat{C}$ has the \emph{triangle property} if, for all  $\cat{C}$-morphisms $Z \xrightarrow{f} Y \xrightarrow{g} X$ between objects of $\cat{L}$, if $g$ and $g f$ are $\cat{L}$-morphisms then $f$ is an $\cat{L}$-morphism.
    \[
        \begin{tikzcd}
            Z \arrow[rd, "gf"'] \arrow[rr, "f", dashed] &   & Y \arrow[ld, "g"] \\
            & X &                  
        \end{tikzcd}
    \]

\end{definition}

By the implication \eqref{i:spectral-local}$\Rightarrow$\eqref{i:nhb-spectral-lh-versus-strict-pm2} in \cref{thm.priest.loc}, every spectral local homeomorphism between Priestley spaces is continuous and order-preserving; hence, we may regard the category of Priestley spaces and spectral local homeomorphisms between them as a subcategory of $\Pries$. We now have:

\begin{proposition}[Triangle property for spectral local homeomorphisms] \label{prop: Triangle property for spectral local homeomorphisms}
    The category of Priestley spaces and spectral local homeomorphisms has the triangle property in $\Pries$.
\end{proposition}
\begin{proof}
    Let $f \colon Z \to Y$ and $g \colon Y \to X$ be continuous order-preserving maps between Priestley spaces, and suppose that $g$ and $gf$ are spectral local homeomorphisms.
    The function $f$ is continuous in the spectral topologies because it is continuous and order-preserving.
    Since local homeomorphisms satisfy the triangle property in the category of topological spaces and continuous maps (see e.g.\ \cite[Lem.~6.4.5]{janelidzeborceaux}), $f$ is a spectral local homeomorphism.
\end{proof}

\section{A topos for étale-finite Heyting algebras} \label{sec:topos}

\subsection{Esakia étale spaces}

Let $\Esa_{\SLH}$ be the category of Esakia spaces with spectral local homeomorphisms between them. 

\begin{definition}[Esakia étale space]
    Let $X$ be an Esakia space. We denote by $\EsaEt(X)$ the slice category $\Esa_{\SLH}/X$, which we refer to as the category of \emph{Esakia étale spaces over $X$}.
\end{definition}

Concretely, 
\begin{enumerate}
    \item an object of $\EsaEt(X)$ is a spectral local homeomorphism
    \[
    Y \longrightarrow X,
    \]
    where $Y$ is an Esakia space;

    \item a morphism from $Y_1 \to X$ to $Y_2 \to X$ is a spectral local homeomorphism $Y_1 \to Y_2$ such that the following diagram commutes.
    \[
    \begin{tikzcd}
        Y_1 \arrow{rr} \arrow{rd}& & Y_2 \arrow{ld}\\
        & X
    \end{tikzcd}
    \]
\end{enumerate}

Our goal is to prove that, for every étale-finite Esakia space $X$, 
\[
\EsaEt(X)
\]
is a topos whose lattice of truth values is isomorphic to $\ClopUp(X)$. 

First, we will prove that, for every Esakia space $X$, the category $\EsaEt(X)$ is finitely complete; then, we will show that, if $X$ is \emph{étale-finite}, then $\EsaEt(X)$ is a topos.
In the next section, we will also prove the converse: if $\EsaEt(X)$ is a topos, then the Esakia space $X$ is étale-finite (\cref{theorem.fp.etale}).

\begin{example}[The Boolean/Freyd case]
    Let $X$ be a Stone space. Then $\EsaEt(X)$ is isomorphic to the slice category $\cat{Stone}_{\mathrm{LH}}/X$, where $\cat{Stone}_{\mathrm{LH}}$ is the category of Stone spaces and local homeomorphisms between them. Indeed, given an object of $\EsaEt(X)$, i.e., a spectral local homeomorphism $f\colon Y\to X$ with $Y$ an Esakia space, if the order on $X$ is discrete, then so must be $Y$ (because any spectral local homeomorphism is a strict p-morphism). It was shown by Freyd (see \cite[Exercise~9.11]{johnstonetopos}) that this category is a topos.
\end{example}

\subsection{Finite limits}

We note that, while calculations of limits will turn out to be easy in slice categories, the same need not be true in the ambient category $\Esa_{\SLH}$:

\begin{example}[Lack of finite limits in the ambient category]
    $\Esa_{\SLH}$ is easily seen not to have a terminal object.
    Indeed, suppose there is a terminal object $T$, and let us derive a contradiction. Then, there is a unique morphism $!_V \colon V \to T$, where $V$ is the V-shaped three-element poset below. 
    \[
        \begin{tikzcd}[row sep=1.2em, column sep=1.4em, arrows={no head}]
            \bullet \arrow[dr] && \bullet \arrow[dl] \\
            & \bullet
        \end{tikzcd}
    \]
    Since a strict p-morphism is injective on principal upsets, $!_V$ is injective. Let $\sigma \colon V \to V$ be the function that swaps the two maximal elements of $V$. Then, $!_V\sigma$ and $!_V$ are two different morphisms from $V$ to $T$, a contradiction.
    
    We next show that, unlike $\TopLH$,\footnote{The equaliser in $\TopLH$ is well known to be the interior of the usual equaliser (see \cite[Example~A.1.2.7]{johnstone2002sketches}).} $\Esa_{\SLH}$ does not have equalisers, either.
    To see this, let $X$ be the following Stone space:
    \[
    \begin{tikzpicture}[
        scale=1,
        point/.style={circle, fill, inner sep=1.7pt},
        every node/.style={font=\small}
    ]
    
        \node[point, label=below:{$\infty$}] (inf) at (0,0) {};
        
        \draw[gray!50] (-5.2,0) -- (-0.9,0);
        
        \node at (-5.8,0) {$\mathbb{N}_0$};
        \node[point, label=below:{$0_0$}] at (-4.8,0) {};
        \node[point, label=below:{$1_0$}] at (-3.8,0) {};
        \node[point, label=below:{$2_0$}] at (-2.8,0) {};
        \node[point, label=below:{$3_0$}] at (-1.8,0) {};
        \node at (-1.1,0) {$\cdots$};
        
        \draw[gray!50] (0.9,0.75) -- (5.2,0.75);
        
        \node at (1.1,0.75) {$\cdots$};
        \node[point, label=above:{$3_1$}] at (1.8,0.75) {};
        \node[point, label=above:{$2_1$}] at (2.8,0.75) {};
        \node[point, label=above:{$1_1$}] at (3.8,0.75) {};
        \node[point, label=above:{$0_1$}] at (4.8,0.75) {};
        \node at (5.8,0.75) {$\mathbb{N}_1$};
        
        \draw[gray!50] (0.9,-0.75) -- (5.2,-0.75);
        
        \node at (1.1,-0.75) {$\cdots$};
        \node[point, label=below:{$3_2$}] at (1.8,-0.75) {};
        \node[point, label=below:{$2_2$}] at (2.8,-0.75) {};
        \node[point, label=below:{$1_2$}] at (3.8,-0.75) {};
        \node[point, label=below:{$0_2$}] at (4.8,-0.75) {};
        \node at (5.8,-0.75) {$\mathbb{N}_2$};
    
    \end{tikzpicture}
    \]
    It consists of three copies of the discrete space $\N$, which we call $\N_{i}\coloneqq \{n_{i} : n\in \N\}$ for $i\in \{0,1,2\}$, together with a limit point $\infty$.
    That is, $X$ is the Alexandroff compactification of the discrete set $\N_0 \sqcup \N_1 \sqcup \N_2$.
    The Stone space $X$ is an Esakia space when considered with equality as a partial order. 
    The identity $\id_{X}$ is moreover easily seen to be a spectral local homeomorphism, and the map $f\colon X\to X$ that permutes the two copies $\N_1$ and $\N_2$ (by mapping $n_1$ to $n_2$ and vice versa) is likewise a spectral local homeomorphism.
    Assume that $\Esa_{\SLH}$ had an equaliser $i\colon \mathsf{Eq}(f,\id_{X})\to X$ of $f$ and $\id_X$, and let us reach a contradiction.
    We denote the image of $i$ by $U$; the set $U$ is contained in $\N_0 \cup \{\infty\}$ since the equaliser equalises.
    Moreover, for every $n \in \N$, the inclusion $\{n_0\} \hookrightarrow X$ is a morphism, since it is an injective spectral local homeomorphism that equalises $f$ and $\id_X$, and so it factors through $i$. Hence, $n_0\in U$ and therefore $\N_0\subseteq U$. 
    Thus, $\N_0 \subseteq U \subseteq \N_0 \cup \{\infty\}$.
    Note that $U$ is a clopen subset of $X$ (because the map $i$ is closed and, by the implication \eqref{i:spectral-local}$\Rightarrow$\eqref{i:nhb-spectral-lh-versus-strict-pm2} in \cref{thm.priest.loc}, open).
    However, no subset $S$ of $X$ satisfying $\N_0 \subseteq S \subseteq \N_0 \cup \{\infty\}$ is clopen.
\end{example}

We now study how limits behave in slice categories. For this, we use the following category-theoretic fact:

\begin{theorem}[Finite limits in the slice under the triangle property] \label{t:the-triangle-property-is-not-that-bad-GENERAL}
    Let $\cat{C}$ be a category and $\cat{L}$ a (not necessarily full) subcategory with the triangle property.
    Suppose that $\cat{C}$ has pullbacks and that the pullback cone in $\cat{C}$ of a cospan belonging to $\cat{L}$ belongs to $\cat{L}$.
    Then, 
    \begin{enumerate}
        \item \label{t:the-triangle-property-is-not-that-bad1}
        $\cat{L}$ has pullbacks, which are computed as in $\cat{C}$; that is, the inclusion $\cat{L}\hookrightarrow \cat{C}$ creates pullbacks\footnote{A functor $F\colon \cat{C}\to\cat{D}$ is said to \emph{create limits of shape $\cat{J}$} if, whenever $D\colon \cat{J}\to \cat{C}$ is a diagram and $\{\beta_j\colon M\to FD(j) : j\in \cat{J}\}$ is a limit cone for $FD$, there is a cone $\{\alpha_j\colon L\to D(j) : j\in \cat{J}\}$ over $D$ whose image under $F$ is isomorphic to $\{\beta_j\colon M\to FD(j) : j\in \cat{J}\}$, and any such cone is a limit for $D$.  It is well known that if $F\colon \cat{C}\to\cat{D}$ creates limits of shape $\cat{J}$ and $\cat{D}$ has limits of shape $\cat{J}$, then $F$ preserves and reflects limits of shape $\cat{J}$.}.
        \item \label{t:the-triangle-property-is-not-that-bad2}
        For every $X\in \L$, the category $\cat{L}/X$ is finitely complete:
        \begin{enumerate}
        
            \item \label{i:pullback-as-in-C}
            A pullback in $\cat{L}/X$ is computed as the pullback in $\cat{C}$, i.e., the composite
            \[
                \cat{L}/X \longrightarrow \cat{L}\longhookrightarrow \cat{C}
            \]
            creates pullbacks.
        
            \item 
            A terminal object of $\cat{L}/X$ is the identity $\id_X \colon X \to X$.

            \item 
            A product of two objects of $\cat{L}/X$ is computed as the pullback in $\cat{C}$.

            \item
            Whenever $a \colon Y\to X$ and $b \colon Z\to X$ are two objects in $\cat{L}/X$, and $f,g\colon Y\to Z$ two parallel morphisms between them, the equaliser of $f$ and $g$ in $\cat{C}$ exists and is also the equaliser in $\cat{L}/X$.
        \end{enumerate}
    \end{enumerate}
\end{theorem}
\begin{proof}
    \eqref{t:the-triangle-property-is-not-that-bad1}
    Let $f\colon Y\to X$ and $g\colon Z\to X$ be morphisms in $\cat{L}$ and let
    \[
        \begin{tikzcd}
        	{Y\times_X Z}\arrow[dr, phantom, very near start, "\lrcorner"] & Z \\
        	Y & X
        	\arrow["{\pi_Z}", from=1-1, to=1-2]
        	\arrow["{\pi_Y}"', from=1-1, to=2-1]
        	\arrow["g", from=1-2, to=2-2]
        	\arrow["f"', from=2-1, to=2-2]
        \end{tikzcd}
    \]
    be the pullback diagram in $\cat{C}$. By assumption, $Y \times_X Z$ is an object of $\cat{L}$, and $\pi_Y$ and $\pi_Z$ are $\cat{L}$-morphisms. In order to conclude that the diagram is also a pullback in $\cat{L}$, we need to show that, given another cone in $\L$, the induced morphism also belongs to $\L$; this clearly follows from the triangle property.
    
    \eqref{t:the-triangle-property-is-not-that-bad2}
    Pullbacks in $\cat{L}/X$ are computed as in $\cat{L}$, and the identity $\id_X \colon X \to X$ is a terminal object of $\cat{L}/X$.
    Since $\cat{L}/X$ has pullbacks and a terminal object, it is finitely complete.
    Moreover, it is easy to check that products in $\cat{L}/X$ are the same thing as pullbacks in $\cat{L}$. 
    Finally, we prove the assertion about equalisers. The composite functor
    \[
        \cat{L}/X \longrightarrow \cat{L}\longhookrightarrow \cat{C}
    \]
    preserves pullbacks. By
    \cite[Lem.~A.1.2.9]{johnstone2002sketches} (see also
    \cite[p.~742]{Pare1990}), any pullback-preserving functor whose domain is
    finitely complete preserves equalisers. Since $\cat{L}/X$ is finitely
    complete, the composite functor above preserves equalisers; thus, the equaliser in $\cat{L}/X$, which exists because $\cat{L}/X$ is finitely complete, is also the equaliser in $\cat{C}$.
\end{proof}

From this, we can deduce the following description of limits in $\EsaEt(X)$; similar facts were proved, with the base category being $\Esa$, in \cite[Cor.~4.14]{kuznetsovetale}.

\begin{remark}[Limits in $\Pries$] \label{r:limits-in-Pries}
    Recall that the category $\Pries$ has limits, which are computed as in the category $\Set$ of sets and functions (see e.g.\ \cite[p.~69]{Gehrke2024}), as well as in the categories $\Top$ of topological spaces and continuous functions, and $\Pos$ of posets and order-preserving maps.
    In fact, the forgetful functors in the following commutative diagram preserve limits.
    \[
    \begin{tikzcd}
        &\Pries \arrow{rd}{}\arrow{dl}\\
        \Top \arrow{rd}&&\cat{Pos}\arrow{dl}\\
        & \Set
    \end{tikzcd}
    \]
\end{remark}

\begin{lemma}[Pullbacks of Esakia spaces and spectral local homeomorphisms] \label{l:pullback}
    The pullback cone in $\Pries$ of a cospan in $\Esa_\SLH$ is in $\Esa_\SLH$.
    That is, given a pullback in $\Pries$
    \[
        \begin{tikzcd}
        	{Y\times_X Z}\arrow[dr, phantom, very near start, "\lrcorner"] & Z \\
        	Y & X
        	\arrow["{\pi_Z}", from=1-1, to=1-2]
        	\arrow["{\pi_Y}"', from=1-1, to=2-1]
        	\arrow["g", from=1-2, to=2-2]
        	\arrow["f"', from=2-1, to=2-2]
        \end{tikzcd}
    \]
    with $X$, $Y$ and $Z$ Esakia spaces and 
    $f \colon Y \to X$ and $g \colon Z \to X$ spectral local homeomorphisms, we have that $Y\times_X Z$ is an Esakia space and $\pi_Y$ and $\pi_Z$ are spectral local homeomorphisms.
\end{lemma}

\begin{proof}
    The pullback is also a pullback in $\Top$.
    Note that $\pi_Y$ and $\pi_Z$ are local homeomorphisms because local homeomorphisms are stable under pullback in $\Top$ (see \cite[Example~A.1.2.7]{johnstone2002sketches}).

    We now prove that $\pi_Y$ is a strict p-morphism.
    Assume that $\pi_{Y}(y,z)\leq y'$; then $y\leq y'$, and so $f(y)\leq f(y')$, and so $g(z)\leq f(y')$; by the strict p-morphism condition there is a unique $z'\geq z$ such that $g(z')=f(y')$. Hence, $(y,z)\leq (y',z')$, and this element is unique as well. So, $\pi_{Y}$ is a strict p-morphism. Similarly, $\pi_Z$ is a strict p-morphism.  Hence $\pi_Y$ and $\pi_Z$ are spectral local homeomorphisms by the implication \eqref{i:nhb-spectral-lh-versus-strict-pm2}$\Rightarrow$\eqref{i:spectral-local} in \cref{thm.priest.loc}.

    Since $\pi_{Y}$ is a spectral local homeomorphism, and $Y$ is an Esakia space, we can therefore conclude that $Y\times_{X}Z$ is an Esakia space by \cref{prop: closure of Esakia under spectral}.
\end{proof}

\begin{theorem}[Finite limits in $\EsaEt(X)$] \label{t:the-triangle-property-is-not-that-bad}
    For every Esakia space $X$, the category $\EsaEt(X)$ is finitely complete:
    \begin{enumerate}
    
        \item \label{eq: terminal object} 
        \emph{(Terminal object)} 
        The identity $\id_{X} \colon X \to X$ is the terminal object in $\EsaEt(X)$.
        
        \item \label{eq: binary product} 
        \emph{(Binary product)} 
        If $f \colon Y\to X$ and $g \colon Z\to X$ are two objects in $\EsaEt(X)$, then the pullback $Y\times_{X}Z$ as ordered spaces with the two projection maps is the product in $\EsaEt(X)$.
        
        \item \label{eq: equaliser} 
        \emph{(Equaliser)} 
        Whenever $f \colon Y\to X$ and $g \colon Z\to X$ are two objects of $\EsaEt(X)$, and $p_{1},p_{2} \colon Y\to Z$ are two spectral local homeomorphisms making the two triangles in the diagram commute,
        \[
            \begin{tikzcd}
                Y \arrow[rr, shift left=.5ex, "p_{1}"] \arrow[rr, shift right=.5ex, swap, "p_{2}"] \arrow[dr, swap, "f"] && Z \arrow[dl, "g"] \\
                & X
            \end{tikzcd}
        \]
        then
        \[
            \mathsf{Eq}(p_{1},p_{2})=\{y\in Y : p_{1}(y)=p_{2}(y)\},
        \]
        together with the inclusion into $Y$, is the equaliser of $p_{1}$ and $p_{2}$ in $\EsaEt(X)$.   

        \item \label{i:pullbacks}
        \emph{(Pullbacks)} The  forgetful functor $\EsaEt(X) \to \Pries$ (i.e., the domain functor) creates pullbacks.
    \end{enumerate}
\end{theorem}

\begin{proof}
    First, it clearly follows from \cref{prop: Triangle property for spectral local homeomorphisms} that the category $\Esa_{\SLH}$ satisfies the triangle property in $\Pries$.  Moreover, $\Pries$ has limits (see \cref{r:limits-in-Pries}). By \cref{l:pullback}, the pullback cone in $\Pries$ of a cospan in $\Esa_{\SLH}$ is in $\Esa_{\SLH}$.
    The result now follows from \cref{t:the-triangle-property-is-not-that-bad-GENERAL}.
\end{proof}

\begin{corollary}[Mono $\Leftrightarrow$ injective]\label{lem: Monos are injective for spectral}
    Let $X$ be an Esakia space.
    A morphism in $\EsaEt(X)$ is a monomorphism if and only if it is injective. 
\end{corollary}
\begin{proof}
    By \cref{t:the-triangle-property-is-not-that-bad}\eqref{i:pullbacks}, the forgetful functor $\EsaEt(X) \to \Pries$  preserves pullbacks. Moreover, the forgetful functor $\Pries \to \Set$ preserves limits and hence the composite $\EsaEt(X) \to \Pries \to \Set$ preserves pullbacks. Therefore, the forgetful functor $\EsaEt(X) \to \Set$ preserves monomorphisms, and it also reflects them because it is faithful.
\end{proof}

\begin{corollary}[Subobject = clopen upset] \label{p:subobjects-EsaEt}
    Let $X$ be an Esakia space and $f \colon Y \to X$ an object of $\EsaEt(X)$.
    The posets $\ClopUp(Y)$ and $\Sub_{\EsaEt(X)}(f)$ are order-isomorphic, via the map
    sending $V\in \ClopUp(Y)$ to the subobject of $f$ represented by
    \[
        \begin{tikzcd}
            V \arrow[swap]{rd}{f{\restriction}_V} \arrow[hook]{rr} & & Y \arrow{ld}{f} \\
            & X.
        \end{tikzcd}
    \]
\end{corollary}

\begin{proof}
    If $V\in \ClopUp(Y)$, then $V$ is an Esakia space and the inclusion
    $V\hookrightarrow Y$ is an injective spectral local homeomorphism, hence a subobject
    of $f$.

    Conversely, let
    $m\colon (Z\xrightarrow{g}X)\to (Y\xrightarrow{f}X)$ represent a subobject of $f$.
    By \cref{lem: Monos are injective for spectral}, $m$ is injective. Its image
    $m[Z]$ is closed because $Z$ is compact and $Y$ is Hausdorff, open because $m$ is a
    local homeomorphism, and an upset because $m$ is a p-morphism. Thus
    $m[Z]\in \ClopUp(Y)$. Moreover, $m\colon Z\to m[Z]$ is an isomorphism in
    $\EsaEt(X)$: it is a homeomorphism onto its image and reflects the order by
    strictness and injectivity. Hence, $m$ represents the same subobject as $m[Z]\hookrightarrow Y$.

    These two assignments are inverse of each other and preserve inclusions.
\end{proof}

\begin{corollary}[Truth value = clopen upset of the base] \label{p:subterminals-EsaEt}
    Let $X$ be an Esakia space. Then
    \[
        \Sub_{\EsaEt(X)}(1) \cong \ClopUp(X).
    \]
\end{corollary}

\begin{proof}
    By \cref{t:the-triangle-property-is-not-that-bad}\eqref{eq: terminal object},
    the terminal object of $\EsaEt(X)$ is $\id_X\colon X\to X$.
    Apply \cref{p:subobjects-EsaEt} to $f=\id_X$.
\end{proof}

\begin{remark}[Priestley spectral local homeomorphisms]
    The results of this subsection hold already at the level of Priestley spaces, with a straightforward adaptation of the proofs.
    Indeed, let $\Pries_{\SLH}$ be the category of Priestley spaces and spectral local homeomorphisms. 
    Then, for every Priestley space $X$, the slice category $\Pries_{\SLH}/X$ is finitely complete, the monomorphisms are precisely the injective morphisms, and $\Sub_{\Pries_{\SLH}/X}(1)\cong \ClopUp(X)$.
    Note that, by \cref{prop: closure of Esakia under spectral}, whenever $X$ is an Esakia space, $\Pries_{\SLH}/X$ coincides with the category of our interest $\EsaEt(X)$.
\end{remark}

\subsection{Power objects}

We now describe power objects, which require a more intricate construction. Here, we will need to assume étale-finiteness.

\begin{definition}[The candidate power object]
    Let $f \colon Y\to X$ be a spectral local homeomorphism between Esakia spaces, with $X$ étale-finite.
    We define the set 
    \[
        P_{f}(Y)\coloneqq \{(x,K) : x\in X,\, K \text{ an upset of }Y \text{ with } K\subseteq f^{-1}[{\u}x]\}.
    \]
    On $P_{f}(Y)$, we define the partial order $\preceq$ by setting
    \[
        (x,K)\preceq (x',K') \iff x\leq x' \text{ and }K'=K\cap f^{-1}[{\u}x'].
    \]
    Note that $f^{-1}[{\u}x]$ is finite, and hence, for every $(x,K) \in P_f(Y)$, $K$ is closed.
    Then, by the definition of the Vietoris space (see \cref{d:Vietoris}),
    \[
        P_{f}(Y)\subseteq X\times \V(Y).
    \]
    We equip $P_{f}(Y)$ with two topologies, which we will show coincide.
    \begin{itemize}
        \item (The topology $\tau_\Viet$) We denote by $\tau_{\Viet}$ the subspace topology on $P_f(Y)$ inherited from $X\times \V(Y)$.
        \item (The topology $\tau_{\Pi, \Lambda}$)
        For every clopen $W\subseteq X$, we define
        \[
            \Pi(W) \coloneqq \{(x,K)\in P_{f}(Y) : x\in W\}.
        \]
        For every clopen upset $V\subseteq Y$, we define 
        \[
            \Lambda(V) \coloneqq \{(x,K)\in P_{f}(Y) : K=V\cap f^{-1}[{\u}x]\}.
        \]
        We define $\tau_{\Pi,\Lambda}$ as the topology on $P_f(Y)$ with subbasis consisting of the sets of the form $\Lambda(V)$ and $P_f(Y) \smallsetminus \Lambda(V)$ for a clopen upset $V \subseteq Y$, as well as the sets of the form $\Pi(W)$ for a clopen $W\subseteq X$.\footnote{Mnemonics: $\Pi$ should be reminiscent of the projection $(x, K) \mapsto x$, while the symbol $\Lambda$ is similar to the intersection symbol $\cap$ appearing in $V \cap f^{-1}[{\u}x]$.}
    \end{itemize}
    Finally, we define the (restriction of the) projection map
    \begin{align*}
        p \colon P_f(Y) & \longrightarrow X\\
        (x, K) & \longmapsto x.
    \end{align*}
\end{definition}

To develop some intuition about these constructions, we give some basic examples of power objects:

\begin{example}[Examples of power objects] \hfill
    \begin{enumerate}
        \item \label{eq: subobject classifier}
        (Subobject classifier) Let $X$ be an étale-finite Esakia space, and let $\acute{e}\colon  X\to X_{0}$ be a continuous strict p-morphism, where $X_{0}$ is finite. 
        The object $P_{\id_{X}}(X)$ is simply
        \begin{equation*}
            \{(x,K) : x\in X, K\subseteq {\u}x,\, K \text{ an upset} \}.
        \end{equation*}
        The intuition for this is that the pair $(x,K)$ is interpreted as saying that $K$ consists of the points in $\u x$ that are ``true''. Note that, since $\acute{e}$ is a strict p-morphism, the only possible truth values are, up to isomorphism, the upsets in $X_{0}$.
        
        The topology $\tau_{\Pi,\Lambda}$ on $P_{\id_{X}}(X)$ can be motivated from this point of view:
        \begin{itemize}
            \item 
            Since we want $(P_{\id_{X}}(X),p)$ to be an object in $\EsaEt(X)$, we need $\Pi(W)$ to be clopen, for each clopen $W\subseteq X$;
            \item 
            Since we want the possible ``truth values'' to form clopen subsets, we need to add, for each $x_{0}\in X_{0}$, the upset
            \begin{align*}
                \Lambda(\acute{e}^{-1}[{\u}x_{0}]) &= \{(x,K) \in P_{\id_X}(X) : K=\acute{e}^{-1}[{\u}x_{0}]\cap {\u}x\}\\
                &=\{(x,K) \in P_{\id_X}(X) : \acute{e}[K]={\u}\acute{e}(x)\cap {\u}x_{0}\};
            \end{align*}
            a natural generalisation of this is to take any clopen upset $V\subseteq X$.
        \end{itemize}
        On the other hand, we want $P_{\id_{X}}(X)$ to be an Esakia space, which is easiest to achieve by seeing it as a subspace of $X\times \V(X)$, since the latter is always compact. This motivates the introduction of the topology $\tau_{\Viet}$.
    
        \item (Boolean subobject classifier)
        As a special case of item \eqref{eq: subobject classifier}, if $X$ is a Stone space, then the space $P_{\id_{X}}(X)$ can equivalently be described as
        \begin{equation*}
            \{(x,1) : x\in X\}\sqcup \{(x,0) : x\in X\},
        \end{equation*}
        with the disjoint union topology and disjoint union partial order,
        where we identify $1=\{x\}$ and $0=\varnothing$; each summand is order-homeomorphic to $X$.
        It follows that $P_{\id_{X}}(X)\cong X\sqcup X$.
    \end{enumerate}
     
\end{example}

In this section, we prove that, if $f\colon Y\to X$ is a spectral local
homeomorphism between Esakia spaces and $X$ is étale-finite, then
\[
\tau_{\Pi,\Lambda}=\tau_{\Viet}.
\]
The two topologies are useful for different purposes: $\tau_{\Pi,\Lambda}$ makes Priestley separation transparent, while $\tau_{\Viet}$ makes compactness
transparent. Their coincidence then allows us to show that $P_f(Y)$ is a Priestley space, and that the projection
$p\colon P_f(Y)\to X$ is a spectral local homeomorphism, and hence that
$P_f(Y)$ is an Esakia space. We will then prove that $p$ is the power object of
$f$ in $\EsaEt(X)$.

\begin{lemma}[In principal fibres: upset = restriction of clopen] \label{lem: Separation for arbitrary elements}
    Let $f \colon Y \to X$ be a spectral local homeomorphism between Esakia spaces, with $X$ étale-finite.
    For each $(x,K)\in P_{f}(Y)$, there is a clopen upset $V\subseteq Y$ such that $K=V\cap f^{-1}[{\u}x]$.
\end{lemma}

\begin{proof}
    Let $(x, K) \in P_f(Y)$.
    The set $K$ is a closed upset of $Y$ contained in $f^{-1}[{\u} x]$.
    The set $f^{-1}[{\u} x]$ is finite by \cref{cor:finiteness}.
    Therefore, $f^{-1}[{\u} x] \smallsetminus K$ is finite and hence closed in $Y$.
    Since $f^{-1}[{\u} x] \smallsetminus K$ is disjoint from the upset $K$, also its down-closure ${\d} (f^{-1}[{\u} x] \smallsetminus K)$ is disjoint from $K$.
    By standard facts on Priestley spaces \cite[Lem.~11.21]{Davey2002-lr}, there is a clopen upset $V$ of $Y$ such that $K \subseteq V$ and $V \cap ({\d} (f^{-1}[{\u} x] \smallsetminus K))=\varnothing$.
    Therefore, $K = V \cap f^{-1}[{\u}x]$.
\end{proof}

\begin{remark}[$\Pi$ and $\Lambda$ of clopen upsets are clopen upsets] \label{r:Pi-Lambda-clopen-upset}
    If $W\subseteq X$ is a clopen upset, then $\Pi(W)$ is clearly a clopen upset.
    If $V\subseteq Y$ is a clopen upset, then $\Lambda(V)$ is clearly clopen by definition of the topology on $P_f(Y)$, and we show that it is an upset. Let $(x,K), (x', K') \in P_f(Y)$ be such that $(x',K') \succeq (x,K) \in \Lambda(V)$.
    Then, by definition, $K=V\cap f^{-1}[\u x]$, $x \leq x'$, and $K' = K \cap f^{-1}[\u x']$.
    Therefore,
    \[
    K' = K \cap f^{-1}[\u x'] = V\cap f^{-1}[\u x] \cap f^{-1}[\u x'] = V \cap f^{-1}[\u x'],
    \]
    i.e., $(x',K') \in \Lambda(V)$.
\end{remark}

From these results, we can deduce the following:

\begin{lemma}[Priestley separation for $P_f(Y)$] \label{lem: Priestley holds for P_f basic topology} 
    Let $f \colon Y \to X$ be a spectral local homeomorphism between Esakia spaces, with $X$ étale-finite.  Then, $(P_{f}(Y), \preceq, \tau_{\Pi, \Lambda})$ satisfies the Priestley separation axiom.
\end{lemma}
\begin{proof}
    Assume that $(x,K)\npreceq (x',K')$. 
    If $x \nleq x'$, then there is a clopen upset $V$ of $X$ such that $x\in V$ and $x'\notin V$, and thus
    \[
        (x, K) \in \Pi(V) \text{ and } (x',K')\notin \Pi(V);
    \]
    note that $\Pi(V)$ is a clopen upset by \cref{r:Pi-Lambda-clopen-upset}.
    Otherwise, $x\leq x'$, and so $K' \neq K \cap f^{-1}[{\u}x']$.
    By \cref{lem: Separation for arbitrary elements}, there is a clopen upset $V\subseteq Y$ such that $V \cap f^{-1}[{\u}x] = K$.
    Then, $(x, K) \in \Lambda(V)$ and $(x', K') \notin \Lambda(V)$ (since, otherwise, $K' = V \cap f^{-1}[{\u}x'] = V\cap f^{-1}[{\u}x] \cap f^{-1}[{\u}x'] = K \cap f^{-1}[{\u}x']$, a contradiction); 
    note that $\Lambda(V)$ is a clopen upset by \cref{r:Pi-Lambda-clopen-upset}.
\end{proof}

\begin{lemma}[Compactness of $P_f(Y)$] \label{lem: power object closed subspace}
     Let $f \colon Y \to X$ be a spectral local homeomorphism between Esakia spaces, with $X$ étale-finite.
     The space $(P_{f}(Y),\tau_{\Viet})$ is a closed subspace of $X\times \V(Y)$, and hence it is compact.
\end{lemma}

\begin{proof}
    First, note that, by definition, $P_{f}(Y)\subseteq X\times\Vu(Y)$; that is, for each $(x,K) \in P_f(Y)$, the set $K$ is assumed to be an upset. Because $Y$ is an Esakia space, by \cref{lem: basic facts about vietoris}\eqref{eq: ordered vietoris esakia} we have that $\Vu(Y)$ is a closed subset of $\V(Y)$.
    Thus, it suffices to show that for every $(x, K)\in (X \times \Vu(Y)) \smallsetminus P_f(Y)$ there is a neighbourhood containing $(x, K)$ entirely outside of $P_{f}(Y)$.
    
    Let $(x,K)\in (X \times \Vu(Y)) \smallsetminus P_f(Y)$.
    Then, $K\nsubseteq f^{-1}[{\u}x]$, i.e., there is $y\in K$ such that $x\nleq f(y)$. Let $V\subseteq X$ be a clopen upset such that $x\in V$ and $f(y)\notin V$. Consider the set
    \[
        V\times \langle f^{-1}[X \smallsetminus V]\rangle,
    \]
    which is open since it is the product of basic opens.    
    Certainly, $(x,K)$ belongs to this set; moreover, whenever $(z,K')$ belongs to it, we cannot have $K'\subseteq f^{-1}[{\u}z]$. 
    This shows the result.
\end{proof}

For the following two statements, we fix:
\[
    \expl{Y}{Esakia \\space} \expl{\xlongrightarrow{\quad\quad\quad\quad\quad\quad f \quad\quad\quad\quad\quad\quad}}{spectral \\local\\ homeomorphism} \expl{X}{étale-finite \\Esakia \\space} \expl{\xlongrightarrow{\quad\quad\quad\quad\quad\quad \acute{e} \quad\quad\quad\quad\quad\quad}}{continuous \\strict \\p-morphism} \expl{X_{0}}{finite\\ poset}
\]
Moreover, by the implication \eqref{i:spectral-local}$\Rightarrow$\eqref{i: Local homeomorphism decomposition} in \cref{thm.priest.loc}, we fix a decomposition
\[
Y=U_{1}\cup{\dots}\cup U_{n}
\]
of $Y$, where, for each $i \in \{1, \dots, n\}$, $U_{i}$ is a clopen upset of $Y$, and $f{\restriction}_{U_{i}}\colon U_{i}\to f[U_{i}]$ is an order-homeomorphism.

For each $n$-tuple $(T_{1},\dots,T_{n})$ of upsets of $X_0$ such that there is $x_{0}\in X_{0}$ with $T_{i}\subseteq {\u}x_{0}$ for all $i\in \{1,\dots,n\}$, consider the following subsets of $P_{f}(Y)$:
\[
    \Typ(T_{1},\dots,T_{n})=\{(x,K)\in P_f(Y) : \acute{e}f[K\cap U_i]=T_{i}\text{ for all }i\in \{1,\dots,n\}\}.
\]

\begin{lemma}[$\Typ$ is clopen]\label{lem: clopenness of the type sets}
     Let $f \colon Y \to X$ be a spectral local homeomorphism between Esakia spaces, with $X$ étale-finite.
     For each $n$-tuple $(T_{1},\dots,T_{n})$ of upsets as above, the set $\Typ(T_{1},\dots,T_{n})$ is clopen in $(P_{f}(Y),\tau_{\Viet})$.
\end{lemma}
\begin{proof}    
    Since $\acute{e} f \colon Y\to X_{0}$ is a continuous p-morphism, by \cref{lem: basic facts about vietoris} we have a continuous map:
    \begin{align*}
        \acute{e}f[-]\colon  \V(Y)&\longrightarrow \V(X_{0})\\
        C &\longmapsto (\acute{e}f)[C].
    \end{align*}
    Moreover, by \cref{lem: basic facts about vietoris}\eqref{eq: restriction maps are continuous}, for each $i\in \{1,\dots,n\}$ the map $r_{U_{i}}\colon \V(Y)\to \V(Y)$ sending $C$ to $C\cap U_{i}$ is continuous.
    
    For each $i \in \{1, \dots, n\}$, let $\alpha_{i}$ be defined by
    \begin{align*}
        \alpha_{i}\colon P_{f}(Y) & \longrightarrow \V(X_{0})\\
        (x,K) &\longmapsto \acute{e}f[K\cap U_{i}];
    \end{align*}
    note that $\alpha_{i}$ is continuous, since it can be factored as the composite of the following morphisms:
    \[
    P_f(Y) \xrightarrow{\pi_{\V(Y)}} \V(Y) \xrightarrow{r_{U_i}} \V(Y) \xrightarrow{\acute{e}f[-]} \V(X_0)
    \]

    Moreover, note that
    \begin{equation*}
        \Typ(T_{1},\dots,T_{n}) = \bigcap_{i=1}^{n}\alpha_{i}^{-1}[\{T_{i}\}];
    \end{equation*}
    indeed, given any $(x,K)\in P_{f}(Y)$, we have $\alpha_{i}(x,K)=(\acute{e}f[-])r_{U_{i}}\pi_{\V(Y)}(x,K)=\acute{e}f[K\cap U_{i}]$. This coincides with the definition of $\Typ(T_{1}, \dots, T_{n})$, and, since $\{T_{i}\}$ is clopen in $\V(X_{0})$ (since the latter is discrete), we have that $\Typ(T_{1}, \dots, T_{n})$ is clopen, as desired.
\end{proof}

Using this lemma, we can prove the equality of the topologies:

\begin{theorem}[The two topologies on $P_f(Y)$ coincide]\label{thm: Identification with Vietoris}
    $\tau_{\Pi,\Lambda} = \tau_{\Viet}$.
\end{theorem}

\begin{proof}
    First note that by \cref{lem: power object closed subspace} the topology $\tau_{\Viet}$ is compact, and by \cref{lem: Priestley holds for P_f basic topology} the topology $\tau_{\Pi, \Lambda}$ is Hausdorff. Therefore, if the inclusion $\tau_{\Pi,\Lambda}\subseteq\tau_{\Viet}$ holds, then the identity map $\id \colon  (P_{f}(Y),\tau_{\Viet})\to (P_{f}(Y),\tau_{\Pi,\Lambda})$ is a continuous bijection from a compact space to a Hausdorff space, and so it is a homeomorphism. Thus, it suffices to establish this inclusion.
    
    To show $\tau_{\Pi,\Lambda} \subseteq \tau_{\Viet}$, we will show (i) that every set of the form $\Pi(W)$ for $W$ a clopen of $X$ is clopen in the topology $\tau_\Viet$, and (ii) that every set of the form $\Lambda(V)$ for $V$ a clopen upset of $Y$ is clopen in the topology $\tau_{\Viet}$ (and so is its complement $P_f(Y) \smallsetminus \Lambda(V)$, too). 
    
    For every clopen $W$ of $X$, we have 
    \[
    \Pi(W) = \{(x,K)\in P_{f}(Y) : x\in W\} = P_f(Y) \cap (W \times \V(Y)),
    \]
    which is clearly clopen in $\tau_\Viet$.
    
    Now fix a clopen upset $V\subseteq Y$ and, for each $i \in \{1, \dots, n\}$, set $V_i \coloneqq V\cap U_i$. Consider:
    \[
        \Types(V) \coloneqq \{(T_{1},\dots,T_{n}) : \exists x\in X \text{ s.t.\ } \forall i\in\{1,\dots,n\},\,T_{i}\subseteq {\u}\acute{e}(x),\, \acute{e}f[V_{i}\cap f^{-1}[{\u}x]]=T_{i}\}.
    \]
    Note also that, for all $x \in X$ and for all subsets $W$ and $W'$ of $Y$, we have
    \begin{equation} \label{eq:iff-Box}
        W\cap f^{-1}[{\u}x]\subseteq W' \text{ iff } f^{-1} [{\u}x]\subseteq (W \Rightarrow W') \text{ iff } x\in \Box\forall_{f}[W\climp W'].
    \end{equation}

    Now, letting $\pi \colon P_f(Y) \to \V(Y)$ denote the restriction of the projection function $X \times \V(Y) \to \V(Y)$, we show:
    \begin{equation} \label{eq:lambda-eq}
        \Lambda(V)=\bigcup_{(T_{1},\dots,T_{n})\in \Types(V)}\Typ(T_{1},\dots,T_{n})\cap \pi^{-1}[\llbracket V \rrbracket] \cap \bigcap_{i=1}^{n}\Pi(\Box\forall_{f}[V_{i}\climp f^{-1}\acute{e}^{-1}[T_{i}]]).
    \end{equation}
    By \cref{lem: clopenness of the type sets}, the sets of the form $\Typ(T_{1}, \dots, T_{n})$ appearing in this finite union are clopen in $\tau_{\Viet}$.
    Moreover, the sets of the form $\forall_{f}[V_{i}\climp f^{-1}\acute{e}^{-1}[T_{i}]]$ are clopen because $f$ is closed and open; since $X$ is Esakia, it follows that also $\Box\forall_{f}[V_{i}\climp f^{-1}\acute{e}^{-1}[T_{i}]]$ is clopen.
    Therefore, the set on the right-hand side of \eqref{eq:lambda-eq} is clopen in $\tau_{\Viet}$. 
    
    To prove \eqref{eq:lambda-eq}, we first prove the inclusion $\supseteq$.
    Let $(T_1, \dots, T_n) \in \Types(V)$ and let $(x,K)\in \Typ(T_{1},\dots,T_{n})\cap \pi^{-1}[\llbracket V\rrbracket] \cap \bigcap_{i=1}^{n}\Pi(\Box\forall_{f}[V_{i}\climp f^{-1}\acute{e}^{-1}[T_{i}]])$. We need to show that
    \[
        K=V\cap f^{-1}[{\u}x].
    \]
    Since $Y = U_1 \cup \dots \cup U_n$, for this it suffices to show that, for each $i \in \{1, \dots, n\}$, upon setting $K_{i}\coloneqq K\cap U_{i}$, we have $K_{i}=V_{i}\cap f^{-1}[{\u}x]$.
    Let $i \in \{1, \dots, n\}$.
    The inclusion $K_{i}\subseteq V_{i}\cap f^{-1}[{\u}x]$ is immediate because $(x,K) \in \pi^{-1}[\llbracket V \rrbracket]$, $K_i = K \cap U_i \subseteq U_i$ and $K_{i}\subseteq f^{-1}[{\u}x]$.
    To prove the inclusion $V_{i}\cap f^{-1}[{\u}x] \subseteq K_i$, let $w\in V_{i}\cap f^{-1}[{\u}x]$. Since $x\in \Box\forall_{f}[V_{i}\climp f^{-1}\acute{e}^{-1}[T_{i}]]$, by \eqref{eq:iff-Box} we have $V_{i}\cap f^{-1}[{\u}x]\subseteq f^{-1}\acute{e}^{-1}[T_{i}]$, and so $\acute{e}f[V_{i}\cap f^{-1}[{\u}x]]\subseteq T_{i}=\acute{e}f[K_{i}]$ (since $(x,K)\in \Typ(T_{1},\dots,T_{n})$). Hence, there is $w'\in K_{i}$ such that $\acute{e}f(w)=\acute{e}f(w')$. Since $w,w'\in V\cap f^{-1}[{\u}x]$, it follows that $x\leq f(w),f(w')$; since $\acute{e}$ is a strict p-morphism and $f(w),f(w') \in \u x$, from $\acute{e}f(w)=\acute{e}f(w')$ we deduce $f(w)=f(w')$. Since $w,w'\in U_{i}$, and $f$ is an order-homeomorphism on this component, we have $w=w'$, i.e., $w\in K_{i}$, as desired.
    
    Conversely, assume that $(x,K)\in \Lambda(V)$, i.e., $K = V \cap f^{-1}[\u x]$.
    For each $i \in \{1, \dots, n\}$, set $K_i \coloneqq K \cap U_i$ and $T_{i} \coloneqq \acute{e}f[K_{i}]$.
    Then, $(T_{1},\dots,T_{n})\in \Types(V)$. 
    Clearly, $(x,K) \in \Typ(T_1, \dots, T_n)$ and $(x,K) \in \pi^{-1}[\llbracket V \rrbracket]$.
    Moreover, for every $i \in \{1, \dots, n\}$,
    \[
        V_{i}\cap f^{-1}[{\u}x] = U_i \cap V \cap f^{-1}[{\u}x] \expl{=}{$(x,K) \in \Lambda(V)$} U_i \cap K = K_{i} \expl{\subseteq}{$\acute{e}f[K_{i}]=T_{i}$}f^{-1}\acute{e}^{-1}[T_{i}],
    \]
    and so, by \eqref{eq:iff-Box}, $x\in \Box\forall_{f}[V_{i}\climp f^{-1}\acute{e}^{-1}[T_{i}]]$.
\end{proof}

In light of \cref{thm: Identification with Vietoris}, we will simply write $(P_{f}(Y),\preceq,\tau)$ to refer to the topology on $P_f(Y)$. We conclude with the following:

\begin{theorem}[$P_f(Y)$ is an Esakia étale space]\label{t:spectralochomeoforpower}
     Let $f \colon Y \to X$ be a spectral local homeomorphism between Esakia spaces, with $X$ étale-finite.
     Then $(P_{f}(Y),\preceq,\tau)$ is an Esakia space, and the map $p \colon P_{f}(Y)\to X$ is a spectral local homeomorphism.
\end{theorem}
\begin{proof}
    Because of \cref{lem: Priestley holds for P_f basic topology} we know that $(P_{f}(Y),\preceq,\tau_{\Pi,\Lambda})$ satisfies the Priestley separation axiom, and by \cref{lem: power object closed subspace}, that $(P_{f}(Y),\tau_{\Viet})$ is compact. By \cref{thm: Identification with Vietoris}, we know that $\tau_{\Pi,\Lambda}=\tau_{\Viet}$, and so $(P_{f}(Y),\preceq,\tau)$ is a Priestley space.
    
    We now show that $p$ is a spectral local homeomorphism. From this and \cref{prop: closure of Esakia under spectral} it will follow that $(P_{f}(Y),\preceq)$ is an Esakia space.
    
    By definition, $p$ is order-preserving, and it is continuous, since it is the restriction of the projection $\pi_{X}\colon X\times \mathcal{V}(Y)\to X$, which is continuous.

    Let $(x,K)\in P_{f}(Y)$ be arbitrary. By \cref{lem: Separation for arbitrary elements}, let $V$ be a clopen upset such that $V\cap f^{-1}[{\u} x] = K$. The set $\Lambda(V)$ is a clopen upset by \cref{r:Pi-Lambda-clopen-upset}.
    Then $(x,K)\in\Lambda(V)$, and note that $p[\Lambda(V)]=X$ because for each $x' \in X$ the element $(x', V \cap f^{-1}[\u x'])$ is mapped to $x'$ by $p$. Moreover, $p{\restriction}_{\Lambda(V)}$ is order-reflecting: if $p(x,K)\leq p(x',K')$ then $x\leq x'$ and
    \begin{equation*}
        K' = V \cap f^{-1}[{\u}x'] = V \cap f^{-1}[{\u}x] \cap f^{-1}[{\u}x'] = K \cap f^{-1}[{\u}x'],
    \end{equation*}
    and so by definition, $(x,K)\preceq (x',K')$. Thus, the function $p{\restriction}_{\Lambda(V)} \colon \Lambda(V) \to X$ is a continuous bijection between compact Hausdorff spaces, and hence a homeomorphism; since it is order-preserving and order-reflecting, it is an order-homeomorphism.

    We thus conclude that the set $p[\Lambda(V)] = X$ is a clopen upset, and $p{\restriction}_{\Lambda(V)}$ is an order-homeomorphism.   
    By the implication \eqref{i:in-terms-of-Priestley}$\Rightarrow$\eqref{i:spectral-local} in \cref{thm.priest.loc}, the map $p$ is a spectral local homeomorphism as desired.
\end{proof}

In light of \cref{t:spectralochomeoforpower}, we have that, for $X$ an étale-finite Esakia space and $f\colon Y\to X$ a spectral local homeomorphism, $(P_{f}(Y),p)$ is an object in $\EsaEt(X)$. We now show that it is indeed the power object of $f\colon Y\to X$.

\begin{definition}[$\in_{Y}$] \label{def: definition of the product in power object}
    By \cref{t:the-triangle-property-is-not-that-bad}\eqref{eq: binary product}, the product of $(Y,f)$ and $(P_{f}(Y),p)$ in $\EsaEt(X)$ is
    \[
        Y\times_{X}P_{f}(Y) = \{(y,(x,K)) \in Y \times P_{f}(Y) : f(y)=x\},
    \]
    together with the composite $Y \times_X P_f(Y) \xrightarrow{\pi_Y} Y \xrightarrow{f} X$.
    We define its subset
    \[
        {\in_{Y}} \coloneqq \{(y,(x,K))\in Y\times_{X}P_{f}(Y) : y\in K\}.
    \]
\end{definition}

\begin{lemma}[$\in_{Y}$ is a clopen upset]
    Let $f \colon Y \to X$ be a spectral local homeomorphism between Esakia spaces, with $X$ étale-finite.
    The set $\in_{Y}$ is a clopen upset of $Y\times_{X}P_{f}(Y)$.
\end{lemma}
\begin{proof}
    We first show that $\in_Y$ is an upset. 
    Let $(y,(x,K))$ be an element of ${\in_Y}$, and let $(y',(x',K')) \in Y\times_{X}P_{f}(Y)$ be such that $(y,(x,K)) \leq (y',(x',K'))$. Then $y\leq y'$ and $(x,K)\preceq (x',K')$; since $K$ is an upset and $y\in K$, it follows that $y'\in K$.
    Since
    \[
        K'=K\cap f^{-1}[{\u}x'],
    \]
    and $f(y')=x'$, we have $y'\in K\cap f^{-1}[{\u}x']$, and so $y'\in K'$.

    Let us now prove that $\in_Y$ is open.
    Let $(y,(x,K))$ be an element of ${\in_Y}$, and let us find an open neighbourhood (in $Y\times_{X}P_{f}(Y)$) of $(y,(x,K))$ contained in $\in_Y$.
    By \cref{lem: Separation for arbitrary elements}, there is a clopen upset $V$ of $Y$ such that $K=V\cap f^{-1}[{\u}x]$.
    We claim that the open set
    \[
        V\times_X \Lambda(V)
    \]
    contains the element $(y,(x,K))$ and is a subset of $\in_Y$.
    Since $y \in K = V \cap f^{-1}[{\u}x]$, we have $y \in V$.
    Moreover, $(x,K)$ belongs to $\Lambda(V)$ by definition of $\Lambda(V)$.
    Therefore, $(y,(x,K)) \in V\times_X \Lambda(V)$.
    We now prove that $V\times_X \Lambda(V)$ is a subset of $\in_{Y}$.
    Let $(y', (x', K')) \in V \times_X \Lambda(V)$.
    Since $f(y') = x'$ (and hence $y' \in f^{-1}[x'] \subseteq f^{-1}[{\u}x']$) and $y' \in V$, we have $y' \in V \cap f^{-1}[{\u}x'] = K'$. Therefore, $(y', (x', K'))$ belongs to $\in_Y$. This proves that $V \times_X \Lambda(V)$ is a subset of $\in_{Y}$.

    It remains to show that $\in_Y$ is closed.
    So, let 
    \[
    (y,(x,K)) \in (Y\times_{X}P_{f}(Y)) \smallsetminus {\in_Y},
    \]
    and let us find an open neighbourhood (in $Y\times_{X}P_{f}(Y)$) of $(y,(x,K))$ disjoint from $\in_Y$.
    By \cref{lem: Separation for arbitrary elements}, there is a clopen upset $V$ of $Y$ such that
    \[
    K=V\cap f^{-1}[{\u}x].
    \]
    We claim that the open set
    \[
        (Y \smallsetminus V)\times_X \Lambda(V)
    \]
    contains $(y,(x,K))$ and is disjoint from $\in_Y$.
    Let us first prove that it contains $(y, (x, K))$. Recall that $K = V \cap f^{-1}[{\u}x]$, that $f(y)=x$ (and hence $y \in f^{-1}[x] \subseteq f^{-1}[{\u}x]$) and that $y \notin K$.
    Thus, $y \notin V$.
    Moreover, $(x,K)$ belongs to $\Lambda(V)$ by definition of $\Lambda(V)$.
    Therefore, $(y,(x,K)) \in (Y \smallsetminus V)\times_X \Lambda(V)$.
    Let us now prove that $(Y \smallsetminus V)\times_X \Lambda(V)$ is disjoint from $\in_{Y}$.
    Let $(y', (x', K')) \in (Y \smallsetminus V)\times_X \Lambda(V)$.
    Since $f(y') = x'$ (and hence $y' \in f^{-1}[x'] \subseteq f^{-1}[{\u}x']$), we have $y' \notin V \cap f^{-1}[{\u}x'] = K'$. Therefore, $(y', (x', K'))$ does not belong to $\in_Y$. This proves that $(Y \smallsetminus V)\times_X \Lambda(V)$ is disjoint from $\in_{Y}$. This proves our claim, and hence $\in_Y$ is closed.
    
    So, $\in_{Y}$ is clopen, as desired.
\end{proof}

Consequently, $\in_{Y}$ (with the composite map ${\in_Y} \hookrightarrow Y \times_X P_f(Y) \xrightarrow{\pi_Y} Y \xrightarrow{f} X$) is an object of $\EsaEt(X)$.
To verify that it is the power object, we make use of the following lemma.

\begin{lemma}[Subobject pullbacks as inverse images] \label{l:pullback-mono-Pries}
    Let $g \colon S \to T$ be a morphism in $\Pries$, and let $A$ and $B$ be closed subsets of $S$ and $T$, respectively (and so Priestley spaces).
    There is a morphism $i \colon A \to B$ making the following square a pullback
    \[
    \begin{tikzcd}[ampersand replacement = \&]
        A \arrow{r}{i} \arrow[hook]{d}\& B \arrow[hook]{d}\\
        S \arrow[swap]{r}{g} \& T
    \end{tikzcd}
    \]
    if and only if $g^{-1}[B] = A$.
    The same holds when $\Pries$ is replaced by $\EsaEt(X)$ and $A$ and $B$ are assumed to be clopen upsets.
\end{lemma}
\begin{proof}
    Let $g \colon S \to T$ be a morphism in $\Pries$, and let $A$ and $B$ be closed subsets of $S$ and $T$, respectively.
    Suppose first that there is a morphism $i \colon A \to B$ making the square in the statement a pullback.
    Since the functor $\Pries\to\Set$ preserves limits, the square is also a pullback in $\Set$.
    Since pullbacks along inclusions in $\Set$ are given by inverse image, the conclusion follows.

    Conversely, assume that $g^{-1}[B]=A$. Then, the morphism $g$ restricts to a map $i\coloneqq g{\restriction}_A\colon A\to B$. Since $A$ carries the subspace topology and induced order from $S$, and $B$ carries the subspace topology and induced order from $T$, the restriction $i$ is continuous and order-preserving, and hence a morphism in $\Pries$ making the square in the statement commute. 
    Moreover, this square is a pullback in $\Set$. To prove that the square is a pullback in $\Pries$, let $C$ be a Priestley space, and let $h\colon C\to  S$ and $k\colon C\to B$ be continuous order-preserving maps such that $g h= i_B k$. Then, there is a unique map $p\colon C\to A$ in $\Set$ such that $ip=k$ and $i_A p=h$, as in the following diagram:
    \[
    \begin{tikzcd}
    	C && \\
    	& A & B \\
    	& S & T
    	\arrow["p"{description}, dashed, from=1-1, to=2-2]
    	\arrow["k", from=1-1, to=2-3, bend left = 2em]
    	\arrow["h"', from=1-1, to=3-2, bend right = 2em]
    	\arrow["i", from=2-2, to=2-3]
    	\arrow["{i_A}"', hook, from=2-2, to=3-2]
    	\arrow["{i_B}", hook, from=2-3, to=3-3]
    	\arrow["g"', from=3-2, to=3-3]
    \end{tikzcd}
    \]
    
    It remains to check that $p$ is a morphism in $\Pries$. The subspace topology on $A$ is the initial topology with respect to the inclusion $i_A\colon A\hookrightarrow S$.
    Therefore, the map $p$ is continuous because $i_Ap=h$ is continuous. Similarly, since the order on $A$ is induced from the order on $S$, the inclusion $i_A$ is an order-embedding. Hence, for $c,c'\in C$ with $c\leq c'$, we have
    \[
    i_Ap(c)=h(c)\leq h(c')=i_Ap(c'),
    \]
    and therefore $p(c)\leq p(c')$ in $A$. Thus $p$ is order-preserving, and so $p$ is a morphism in $\Pries$.
    
    The uniqueness of $p$ as a morphism in $\Pries$ follows from its uniqueness as a map in $\Set$. Hence, the square is a pullback in $\Pries$.
    
    Finally, let us prove the statement for $\EsaEt(X)$. Let $g\colon S\to T$ be a morphism in $\EsaEt(X)$, and let $A\subseteq S$ and $B\subseteq T$ be clopen upsets, equipped with the induced structure maps to $X$. By the first part, the corresponding square is a pullback in $\Pries$ if and only if $g^{-1}[B]=A$. Since the forgetful functor
    \[
    \EsaEt(X)\longrightarrow \Pries
    \]
    creates pullbacks by \cref{t:the-triangle-property-is-not-that-bad}\eqref{i:pullbacks}, the square is a pullback in $\EsaEt(X)$ if and only if its image is a pullback in $\Pries$. Therefore, the same equivalence holds in $\EsaEt(X)$.
\end{proof}

\begin{theorem}[$P_f(Y)$ is the power object] \label{thm: Power object theorem}
    Let $f \colon Y \to X$ be a spectral local homeomorphism between Esakia spaces, with $X$ étale-finite.
    In $\EsaEt(X)$, the object $p \colon P_{f}(Y) \to X$ is the power object of $f \colon Y\to X$.
\end{theorem}

\begin{proof}
    Let $m \colon Z\to X$ be an object in $\EsaEt(X)$.
    By \cref{p:subobjects-EsaEt}, the subobjects of $Y\times_{X}Z \to X$ are the clopen upsets of $Y\times_{X}Z$.
    So, let $R\subseteq Y\times_{X}Z$ be a clopen upset. 
    We shall show that there is a unique morphism $k \colon Z\to P_{f}(Y)$ such that there is a morphism $R \to {\in_Y}$ making the following square a pullback.
    \[
        \begin{tikzcd}
            R \arrow[d, tail] \arrow[r]                    & \in_{Y} \arrow[d, tail] \\
            Y\times_X Z \arrow[r, "\id_{Y}\times k"'] & Y\times_X P_f(Y)
        \end{tikzcd}
    \]
    By \cref{l:pullback-mono-Pries}, this amounts to showing that there is a unique morphism $k \colon Z\to P_{f}(Y)$ such that $(\id_Y \times k)^{-1}[\in_Y] = R$, where $\id_Y \times k$ is the product function $Y\times_X Z \to Y\times_X P_f(Y)$.
    In turn, $(\id_Y \times k)^{-1}[\in_Y] = R$ means: for all $(y,z) \in Y \times_X Z$, setting $(x_{z},K_{z}) \coloneqq k(z)$, 
    \begin{equation} \label{l:iff}
        y \in K_{z} \iff (y,z) \in R.
    \end{equation}

    We first prove the uniqueness of $k$.
    Let $k \colon Z\to P_{f}(Y)$ be a morphism such that $(\id_Y \times k)^{-1}[\in_Y] = R$.
    For each $z \in Z$, we write $(x_z,K_z)\coloneqq k(z)$.
    Since $k$ is a morphism, the following diagram commutes.
    \[
        \begin{tikzcd}
            Z \arrow[swap]{rd}{m}\arrow{rr}{k} & & P_f(Y)\arrow{ld}{p}\\
            & X
        \end{tikzcd}
    \]
    Therefore, $x_z = p(x_z, K_z) = pk(z) = m(z)$.
    This shows that we have at most one possible value for $x_z$.
    We are left to show that there is at most one possible value for $K_z$.
    Recall that, since $(x_z, K_z) \in P_f(Y)$, we have $K_{z}\subseteq f^{-1}[{\u}x_{z}]=f^{-1}[{\u}m(z)]$ (the equality following from what was just shown). 
    Therefore, we shall determine which elements of $f^{-1}[\u m(z)]$ belong to $K_z$.
    If we take $y' \in f^{-1}[\u m(z)]$, we have $f(y') \geq m(z)$.
    Then, by the strictness of $m$, there is a unique $z' \in \u z$ such that $m(z') = f(y')$, which means that $(y',z')\in Y\times_{X}Z$.
    By \eqref{l:iff}, $y' \in K_{z'}$ if and only if $(y', z') \in R$.
    Since $k$ is order-preserving and $z \leq z'$, we have $k(z) \preceq k(z')$.
    Therefore, by the definition of the partial order $\preceq$ on $P_f(Y)$, we have $K_{z'} = K_z \cap f^{-1}[\u m(z')]$.
    Since $y' \in f^{-1}[\u m(z')]$, we have
    \[
    y' \in K_z \iff y' \in K_{z'} \iff (y', z') \in R.
    \]
    In conclusion, for every $z \in Z$, $x_z = m(z)$, and $K_z$ is the set of $y' \in f^{-1}[\u m(z)]$ such that, for the unique $z' \in \u z$ with $f(y') = m(z')$, we have $(y', z') \in R$.
    This proves the uniqueness of $k$.

    Having shown that there is at most one possible candidate, let us prove the existence of the morphism $k$ with the desired properties.
 
    For each $z$, set $x_z \coloneqq m(z)$, and define $K_z$ as the set of elements $y' \in f^{-1}[\u m(z)]$ such that, for the unique $z' \in \u z$ with $(y',z') \in Y \times_X Z$, we have $(y',z') \in R$. In other words:
    \[
        K_{z} = f^{-1}[{\u}m(z)]\cap \pi_{Y}[R\cap \pi_{Z}^{-1}[{\u}z]].
    \]
    We show that the function
    \begin{align*}
        k \colon Z & \longrightarrow P_f(Y)\\
        z & \longmapsto (x_z, K_z)
    \end{align*}
    is well-defined.
    This is done by showing that, for each $z\in Z$, $k(z)$ is indeed an element of $P_f(Y)$, i.e., that $K_z \subseteq f^{-1}[\u x_z]$ and that $K_z$ is an upset of $Y$.
    Clearly, $K_z \subseteq f^{-1}[{\u}m(z)] = f^{-1}[\u x_z]$; we prove that $K_{z}$ is an upset of $Y$.
    Let $y' \in K_{z}$ and $y' \leq y''$.
    Let $z'$, resp.\ $z''$, be the unique element in $\u z$ such that $f(y') = m(z')$, resp.\ $f(y'') = m(z'')$.
    Since $y' \in K_z$, we have $(y', z') \in R$.
    Moreover, since $m{\restriction}_{\u z}$ is an order-isomorphism and $m(z') = f(y') \leq f(y'') = m(z'')$, we have $z' \leq z''$.
    Since $R$ is an upset of $Y \times_X Z$ and $(y'', z'') \geq (y', z') \in R$, we have $(y'', z'') \in R$, and hence $y'' \in K_z$.
    This shows that $K_z$ is an upset.
    Thus, $k$ is a well-defined function.
    
    It is immediate that the following diagram of functions commutes.
    \[
        \begin{tikzcd}
            Z \arrow[swap]{rd}{m}\arrow{rr}{k} & & P_f(Y)\arrow{ld}{p}\\
            & X
        \end{tikzcd}
    \]

    To prove that $k$ is a morphism in $\EsaEt(X)$, by the triangle property of $\Esa_{\SLH}$ in $\Pries$ (\cref{prop: Triangle property for spectral local homeomorphisms}) it suffices to show that $k$ is order-preserving and continuous.

    Let us prove that $k$ is order-preserving. Let $z,z' \in Z$ with $z\leq z'$, and let us prove that $k(z) \preceq k(z')$, i.e., that $(m(z), K_z) \preceq (m(z'), K_{z'})$, where $\preceq$ is the order on $P_f(Y)$.
    By monotonicity of $m$, we have $m(z)\leq m(z')$.
    We prove the equality $K_{z'}=K_z\cap f^{-1}[\u m(z')]$. 
    The inclusion $K_{z'}\subseteq K_z\cap f^{-1}[\u m(z')]$ holds because $z\leq z'$: indeed, if $y\in K_{z'}$, the unique point of $\u z'$ mapping to $f(y)$ is also the unique point of $\u z$ mapping to $f(y)$.
    Conversely, let $y\in K_z\cap f^{-1}[\u m(z')]$. Let $\bar z$ be the unique point in $\u z$ such that $m(\bar z)=f(y)$. Since $z'\in \u z$, $\bar z\in \u z$, and
    \[
        m(z')\leq f(y)=m(\bar z),
    \]
    the restriction $m\colon \u z\to \u m(z)$ reflects the order by strictness of $m$. Hence $z'\leq \bar z$. Therefore, $\bar z\in \u z'$, and so $\bar z$ is also the unique point of $\u z'$ mapping to $f(y)$. Since $y\in K_z$, we have $(y,\bar z)\in R$, and hence $y\in K_{z'}$. Thus
    \[
        K_{z'}=K_z\cap f^{-1}[\u m(z')],
    \]
    as required.
    Therefore, $k$ is order-preserving.

    Next, we prove continuity of $k$ by showing that preimages of subbasic clopens for the topology $\tau_{\Viet}$ are clopen; by \cref{thm: Identification with Vietoris}, this establishes continuity. Consider first a clopen $W\subseteq X$; then,
    \[
        k^{-1}[\Pi(W)]=\{z\in Z : m(z)\in W\}=m^{-1}[W],
    \]
    which is clopen.
    Now let $V$ be a clopen subset of $Y$; then
    \[
        k^{-1}[\llbracket V\rrbracket]=\{z\in Z : K_{z}\subseteq V\}.
    \]
    
    Let $\pi_Y \colon Y \times_X Z \to Y$ and $\pi_Z \colon Y \times_X Z \to Z$ denote the projections. To show that $k^{-1}[\llbracket V\rrbracket]$ is clopen, we claim
    \[
        k^{-1}[\llbracket V\rrbracket] = \Box \forall_{\pi_Z}[R \Rightarrow \pi_Y^{-1}[V]].
    \]
    (Note that $\Box \forall_{\pi_Z}[R \Rightarrow \pi_Y^{-1}[V]]$ is clopen because $R$ is clopen, $\pi_{Z}$ is an open map and $Z$ is Esakia.)
    We prove the claim. Let $z\in Z$. Then $z\in \Box \forall_{\pi_Z}[R \Rightarrow \pi_Y^{-1}[V]]$ if and only if, for every $z'\geq z$ and every $y\in Y$ such that $f(y)=m(z')$, we have
    \[
        (y,z')\in R \text{ implies } y\in V.
    \]
    We show that this is equivalent to $K_z\subseteq V$.
    Suppose first that $z\in \Box \forall_{\pi_Z}[R \Rightarrow \pi_Y^{-1}[V]]$, and let $y'\in K_z$. Let $z'$ be the unique point in $\u z$ such that $m(z')=f(y')$. By definition of $K_z$, we have $(y',z')\in R$. Since $z'\geq z$, the assumption gives $y'\in V$. Thus $K_z\subseteq V$.
    Conversely, suppose that $K_z\subseteq V$. Let $z'\geq z$, and let $y'\in Y$ be such that $f(y')=m(z')$. Assume that $(y',z')\in R$. Since $z'\geq z$, we have $f(y')=m(z')\geq m(z)$, and so $y'\in f^{-1}[\u m(z)]$. Moreover, by strictness of $m$, the point $z'$ is the unique element of $\u z$ mapped by $m$ to $f(y')$. Hence, by definition of $K_z$, from $(y',z')\in R$ we get $y'\in K_z$. Since $K_z\subseteq V$, it follows that $y'\in V$. Therefore
    \[
        z\in \Box \forall_{\pi_Z}[R \Rightarrow \pi_Y^{-1}[V]].
    \]
    This proves
    \[
        k^{-1}[\llbracket V\rrbracket]
        =
        \Box \forall_{\pi_Z}[R \Rightarrow \pi_Y^{-1}[V]].
    \]
    This proves that $k^{-1}[\llbracket V \rrbracket]$ is clopen.
    Clearly, also the set $k^{-1}[\langle V \rangle] = Z \smallsetminus k^{-1}[\tru{Y \smallsetminus V}]$ is clopen.
    We thus have that the preimage of subbasic clopens in $P_{f}(Y)$ is clopen, and hence $k$ is continuous. 
    
    Since $pk=m$ and $k$ is continuous and order-preserving, by the triangle property (\cref{prop: Triangle property for spectral local homeomorphisms}) $k$ is a spectral local homeomorphism.

    Finally, let $(y,z) \in Y\times_X Z$. We prove that
    \[
        y \in K_z \iff (y,z) \in R.
    \]
    Suppose first that $(y,z)\in R$. Since $z\in \u z$ and $m(z)=f(y)$, the point $z$ is the unique element of $\u z$ mapped by $m$ to $f(y)$. Hence, by the definition of $K_z$, we have $y\in K_z$.
    Conversely, suppose that $y\in K_z$. By definition of $K_z$, if $\bar z$ is the unique point in $\u z$ such that $m(\bar z)=f(y)$, then $(y,\bar z)\in R$. But $(y,z)\in Y\times_X Z$, so $m(z)=f(y)$; hence $z$ is also an element of $\u z$ mapped by $m$ to $f(y)$. By uniqueness, $\bar z=z$. Therefore, $(y,z)\in R$.
\end{proof}

We arrive at the main result of the paper:

\begin{theorem}[\'Etale-finite $\Rightarrow$ topos-admissible] \label{thm: elementary topos from etale-finite Heyting algebra}
    For every étale-finite Heyting algebra $H$, there is an elementary topos $\E$,
    namely
    \[
    \E \coloneqq \EsaEt(\Spec(H)),
    \]
    such that $\Sub_\E(1)\cong H$.
\end{theorem}
\begin{proof}
    By \cref{cor:etale-finite-iff-etale-finite}, $\Spec(H)$ is an étale-finite Esakia space.
    Set $\E \coloneqq \EsaEt(\Spec(H))$. 
    By \cref{t:the-triangle-property-is-not-that-bad}, this category has all finite limits, and, by \cref{thm: Power object theorem}, every object has a power object. Thus, it is an elementary topos. By \cref{p:subterminals-EsaEt}, $\Sub_{\E}(1) \cong \ClopUp(\Spec(H))$, and, by Esakia duality, this is isomorphic to $H$.
\end{proof}

\section{Finitely propositional toposes} \label{sec:finitely-propositional-toposes}

In this section, we exhibit an obstruction to extending our use of compact étale spaces beyond the étale-finite case. For this purpose, we introduce the notion of a \emph{finitely propositional} topos and show that the étale-finite Heyting algebras are precisely those Heyting algebras that arise as the lattice of truth values of such toposes.

\begin{definition}[(Finitely) propositional] \label{d:finitely propositional topos}
    We say that a category with a terminal object is \emph{(finitely) propositional} if every object admits a jointly epimorphic (finite) family of morphisms from subterminal objects; that is, for every object $Y$ there is a jointly epimorphic (finite) family
    \[
    \{V_i\longrightarrow Y\}_{i\in I}
    \]
    with each $V_i$ subterminal.
\end{definition}

Our interest in (finite) propositionality is, roughly speaking, that it is a property of categories of (compact) étale spaces.

Note that a category with a terminal object is propositional precisely when its
subterminal objects form a separating family (see
\cite[p.~12]{johnstone2002sketches}). For Grothendieck toposes, propositionality is characterised geometrically: it holds precisely for those Grothendieck toposes that are localic (i.e., equivalent to the
topos of sheaves on a locale), see \cite[Thm.~C1.4.7]{johnstone2002sketches}. Thus, for instance, the
presheaf topos $[\cat{C}^{\op},\Set]$ is non-propositional whenever $\cat{C}$
is not a preorder (see \cite[Cor.~3.2]{caramello}).

Outside the Grothendieck setting, the notion of propositionality for elementary
toposes has been studied in \cite{kenney}.

Finally, the conditions of propositionality and finite propositionality do not coincide: an example of a propositional topos that is not finitely propositional is the category $\E \coloneqq \mathrm{Sh}(\omega + 1)$ of sheaves on the locale $\omega + 1$. Indeed, since $\Sub_\E(1) \cong \omega + 1$ and the Heyting algebra $\omega + 1$ is not étale-finite, it will follow from \cref{cor: examples of non etale-finite algebras}\eqref{i:chain-non-etale} below that $\E$ is not finitely propositional.

\begin{theorem}[Esakia étale spaces are finitely propositional]\label{thm: etale-finite yield finitely propositional}
    For every Esakia space $X$, the category $\EsaEt(X)$ is finitely propositional.
\end{theorem}

\begin{proof}
    Let $f\colon Y\to X$ be an object of $\EsaEt(X)$.
    By the implication \eqref{i:spectral-local}$\Rightarrow$\eqref{i: Local homeomorphism decomposition} in \cref{thm.priest.loc}, there are clopen upsets $U_1,\dots,U_n$ of $Y$ such that $ Y = U_1\cup\cdots\cup U_n$ and, for each $i\in\{1,\dots,n\}$, the restriction $f{\restriction}_{U_i}\colon U_i \to f[U_i]$
    is an order-homeomorphism. Set $V_i \coloneqq f[U_i]$.
    Then, each $V_i$ is a clopen upset of $X$, and hence determines a subterminal object $V_i\rightarrowtail X$ of $\EsaEt(X)$. Let
    \[
        g_i\colon V_i \longrightarrow Y
    \]
    be the composite of $(f{\restriction}_{U_i})^{-1}$ with the inclusion $U_i \hookrightarrow Y$. Since $f g_i$ is the inclusion $V_i \hookrightarrow X$, each $g_i$ is a morphism in the slice category $\EsaEt(X)$.

    The family $\{ g_i : i\in\{1,\dots,n\}\}$ is jointly epimorphic. Indeed, any faithful functor reflects jointly epimorphic families, and so, by considering the composite
    \[
        \EsaEt(X) \longrightarrow \Esa_{\SLH} \longrightarrow \cat{Set},
    \]
    it suffices to check that $\{ g_i : i\in\{1,\dots,n\}\}$ is jointly surjective. But this follows immediately from $Y=U_1\cup \dots\cup U_n$.
\end{proof}

It is reasonable to ask whether one could extend our approach to Heyting algebras that are not étale-finite---such as the chains $\omega + 1$ or $\omega + \omega^\op$. We will now show that, for every finitely propositional topos $\E$, the Heyting algebra $\mathrm{Sub}_{\E}(1)$ is étale-finite.

\begin{proposition}[Finitely propositional $\Rightarrow$ Jibladze's law over a finite subset]\label{prop: weak etale condition}
    Let $\E$ be a finitely propositional elementary topos, and set $H \coloneqq \Sub_\E(1)$. There is a finite subset $S\subseteq H$ such that, for every $a\in H$,
    \[
    \bigvee_{s\in S} (a\leftrightarrow s) =1.
    \]
\end{proposition}

\begin{proof}
    By finite propositionality applied to the subobject classifier $\Omega$, let $V_1,\dots, V_n\in \Sub_\E(1)$ be such that the family $a_i\colon V_i\to \Omega$ is jointly epimorphic. Consider the subobject $U_i\rightarrowtail V_i$ classified by $a_i$, i.e., appearing in the following pullback diagram.
    \begin{equation}\label{pbb1}
        \begin{tikzcd}
        	{U_i} \arrow[dr, phantom, very near start, "\lrcorner"] & 1 
            \\
        	{V_i} & \Omega
        	\arrow[from=1-1, to=1-2]
        	\arrow[tail, from=1-1, to=2-1]
        	\arrow["t", from=1-2, to=2-2]
        	\arrow["{a_i}"', from=2-1, to=2-2]
        \end{tikzcd}
    \end{equation}
    For each $i \in \{1, \dots, n\}$, the composite $U_i\rightarrowtail V_i\rightarrowtail 1$ of monomorphisms shows that $U_i$ is also a subterminal object. We claim that $S=\{U_1,\dots, U_n\}$ works. 
    
    Let $W\rightarrowtail 1$ be any subterminal and let $p \colon 1 \to \Omega$ be its classifying morphism, i.e., the morphism for which there is a pullback diagram as follows.
    \begin{equation}\label{pbb2}
        \begin{tikzcd}
        	W \arrow[dr, phantom, very near start, "\lrcorner"]& 1 & \\
        	1 & \Omega 
        	\arrow[from=1-1, to=1-2]
        	\arrow[tail, from=1-1, to=2-1]
        	\arrow["t", from=1-2, to=2-2]
        	\arrow["p"', from=2-1, to=2-2]
        \end{tikzcd}
    \end{equation}
    For each $i \in \{1, \dots, n\}$, let $p^*(V_i)$ denote the pullback of $a_i\colon V_i\to\Omega$ along $p\colon 1\to\Omega$.
    \begin{equation}
        \label{pbb3}
        \begin{tikzcd}
        	{p^*(V_i)} \arrow[dr, phantom, very near start, "\lrcorner"] & 1 & \\
        	{V_i} & \Omega 
        	\arrow[from=1-1, to=1-2]
        	\arrow[tail, from=1-1, to=2-1]
        	\arrow["p", from=1-2, to=2-2]
        	\arrow["{a_i}"', from=2-1, to=2-2]
        \end{tikzcd}
    \end{equation}
    We now construct the following two diagrams:
    \[
    \begin{tikzcd}
    	{p^*(V_i)\wedge U_i}\arrow[dr, phantom, very near start, "\lrcorner"] & {U_i}\arrow[dr, phantom, very near start, "\lrcorner"] & 1 \\
    	{p^*(V_i)} & {V_i} & \Omega \\
    	{p^*(V_i)\wedge W} \arrow[dr, phantom, very near start, "\lrcorner"] & W\arrow[dr, phantom, very near start, "\lrcorner"] & 1 \\
    	{p^*(V_i)} & 1 & \Omega
    	\arrow[from=1-1, to=1-2, tail]
    	\arrow[tail, from=1-1, to=2-1]
    	\arrow[from=1-2, to=1-3]
    	\arrow[tail, from=1-2, to=2-2]
    	\arrow["t", from=1-3, to=2-3]
    	\arrow[tail, from=2-1, to=2-2]
    	\arrow["{a_i}", swap, from=2-2, to=2-3]
    	\arrow[tail, from=3-1, to=3-2]
    	\arrow[tail, from=3-1, to=4-1]
    	\arrow[from=3-2, to=3-3]
    	\arrow[tail, from=3-2, to=4-2]
    	\arrow["t", from=3-3, to=4-3]
    	\arrow[tail, from=4-1, to=4-2]
    	\arrow["p"', from=4-2, to=4-3]
    \end{tikzcd}
    \]
    Both left squares are pullback squares by the definition of the meet in subobject lattices. The right squares are pullback squares because of \eqref{pbb1} and \eqref{pbb2}, respectively. Hence, both outer squares are pullbacks. But note that the lower composites of both lower squares are equal by \eqref{pbb3}. Hence, both subobjects $p^*(V_i)\wedge U_i$ and $p^*(V_i)\wedge W$ of $p^*(V_i)$ are classified by the same morphism---i.e., $p^*(V_i)\wedge U_i=p^*(V_i)\wedge W$. It follows that
    \[
        p^*(V_i)\leq U_i\leftrightarrow W.
    \]
    
    Finally, construct the pullback
    \[
        \begin{tikzcd}
        	{p^*(V_1 \sqcup \dots \sqcup V_n)} \arrow[dr, phantom, very near start, "\lrcorner"] & 1 \\
        	{V_1\sqcup\dots\sqcup V_n} & \Omega
        	\arrow[two heads, from=1-1, to=1-2]
        	\arrow[from=1-1, to=2-1]
        	\arrow["p"', swap, from=1-2, to=2-2]
        	\arrow[two heads, from=2-1, to=2-2]
        \end{tikzcd}
    \]
    Since in a topos (regular) epimorphisms are pullback-stable, the morphism $p^*(V_1\sqcup\dots\sqcup V_n)\to 1$ in the diagram is an epimorphism. Moreover, recall that also coproducts are pullback-stable (see \cite[Lem.~1.51]{johnstonetopos}), and so we have an epimorphism $p^*(V_1)\sqcup \dots\sqcup  p^*(V_n)\twoheadrightarrow 1$. But joins in subobject lattices in a topos are computed by taking the monomorphic part of the image factorisation of the map from the coproduct (see \cite[p.~31]{johnstone2002sketches}); since this map is already an epimorphism, it follows that $p^*(V_1)\vee\dots \vee p^*(V_n)$ is the top element of $H$.  Hence, 
    \[
        1 = \bigvee_{i=1}^{n} p^*(V_i) \leq \bigvee_{i=1}^{n} (U_i\leftrightarrow W). \qedhere
    \]
\end{proof}

\begin{proposition}[Jibladze's law over a finite \emph{subset} $\Rightarrow$ étale-finite]\label{prop: finite set to etale-finite}
    Let $H$ be a Heyting algebra and $S\subseteq H$ a finite set such that, for all $x \in H$,
    \[
        \bigvee_{s\in S}(x \leftrightarrow s) =1.
    \]
    Then $H$ is locally finite, and so the Heyting subalgebra generated by $S$ is finite.
    Therefore, $H$ is étale-finite.
\end{proposition}
\begin{proof}
    Let $n$ be the cardinality of $S$.
    The first-order sentence
    \begin{equation}\label{eq:formulaforetaleness}
        \exists s_{1}\,\dots\,\exists s_{n}\,\forall x\,\Bigl(\bigvee_{i=1}^{n}(x \leftrightarrow s_{i})=1\Bigr)
    \end{equation}
    holds in $H$.
    Moreover, note that this sentence is positive (in the first-order sense: it is built from atomic formulas without use of negation) and hence, by \cite[Prop.~5.2.12]{Chang2012-ub}, it is preserved under surjective homomorphisms. Recall that the top element of a subdirectly irreducible Heyting algebra is completely join-irreducible \cite[Prop.~A.1.1]{Esakiach2019HeyAlg}; thus, by \cref{l:join-irreducible}, every subdirectly irreducible Heyting algebra satisfying \eqref{eq:formulaforetaleness} has at most $n$ elements. Therefore, every subdirectly irreducible homomorphic image of $H$ has cardinality at most $n$.
    Therefore, $H$ has, up to isomorphism, finitely many subdirectly irreducible homomorphic images. 
    By Birkhoff's subdirect representation theorem, $H$ embeds into a product of its subdirectly irreducible homomorphic images. Therefore:
    \begin{equation*}
        \mathsf{Var}(H)=\mathsf{Var}(\{H_{0} : H\twoheadrightarrow H_{0}, H_{0}\text{ is s.i.}\}),
    \end{equation*}
    and we conclude that $\mathsf{Var}(H)$ is finitely generated, i.e., generated by a finite family of finite algebras.

    It is a classical result \cite[Theorem~10.16]{BurrisSankappanavar} that finitely generated varieties are locally finite.
    Therefore, the Heyting subalgebra $\langle S\rangle$ of $H$ generated by the finite set $S$ is finite, as well. But note that, for all $x \in H$,
    \begin{equation*}
        \bigvee_{s\in \langle S\rangle}(x\leftrightarrow s)=1,
    \end{equation*}
    since the join is taken over a set that extends $S$. So, by \cref{r:etale-finite-basic}\eqref{i: finite algebra condition}, the inclusion $\langle S\rangle \hookrightarrow H$ satisfies Jibladze's law, and so $H$ is étale-finite.
\end{proof}

We can now conclude the following:

\begin{theorem}[Finitely propositional $\leftrightarrow$ étale-finite]\label{theorem.fp.etale}
    For a Heyting algebra $H$, the following are equivalent:
    \begin{enumerate}
        \item \label{eq: esa et is finitely propositional topos} $\EsaEt(\Spec(H))$ is a topos;
        \item \label{eq:isomorphictotruthvalues} $H$ is isomorphic to the lattice of truth values of some finitely propositional topos;
        \item \label{eq: H etale-finite} $H$ is étale-finite.
    \end{enumerate}
\end{theorem}
\begin{proof}
    \eqref{eq: esa et is finitely propositional topos}$\Rightarrow$\eqref{eq:isomorphictotruthvalues} By \cref{thm: etale-finite yield finitely propositional}, the category 
    $\EsaEt(\Spec(H))$ is finitely propositional. 
    By \eqref{eq: esa et is finitely propositional topos}, it is also a topos.
    By \cref{p:subterminals-EsaEt}, $\Sub_{\EsaEt(\Spec(H))}(1) \cong \ClopUp(\Spec(H))$. By Esakia duality, $\ClopUp(\Spec(H)) \cong H$. Thus, $H$ is isomorphic to the lattice of truth values of a finitely propositional topos.
    
    \eqref{eq:isomorphictotruthvalues}$\Rightarrow$\eqref{eq: H etale-finite} This follows from \cref{prop: weak etale condition,prop: finite set to etale-finite}. 
    
    \eqref{eq: H etale-finite}$\Rightarrow$\eqref{eq: esa et is finitely propositional topos} If $H$ is étale-finite, then by \cref{thm: elementary topos from etale-finite Heyting algebra} $\EsaEt(\Spec(H))$ is an elementary topos.
\end{proof}

\begin{remark}[Functoriality?]
    We have shown that étale-finite Heyting algebras give rise to finitely propositional toposes, and, conversely, the Heyting algebra of truth values of every finitely propositional topos is étale-finite. We leave open whether these assignments can be made functorial, and what categorical relationship, if any, they induce.
\end{remark}

From \cref{theorem.fp.etale} we can deduce the following:

\begin{corollary}[Heyting algebras with no finitely propositional topos]\label{cor: examples of non etale-finite algebras}
    None of the following algebras arises as the lattice of truth values of a finitely propositional topos.
    \begin{enumerate}
        \item \label{i:chain-non-etale} Any infinite bounded chain (such as $\omega+1$ or $\omega+\omega^{\op}$).
        
        \item The free Heyting algebra $\mathcal{F}_{\HA}(X)$ over any nonempty set $X$.
        
        \item \label{eq: algebra H}
        The dual of the following Esakia space, where every point is isolated, except for $\infty$ which is an accumulation point:
        \[
            \begin{tikzpicture}
                \node at (0,0) {$\bullet$};
                \node at (1,0) {$\bullet$};
                \node at (2,0) {$\bullet$};
                \node at (3,0) {$\bullet$};
                \node at (4,0) {$\dots$};

                \foreach \x in {0,1,2,3} {
                    \draw (\x,0) -- (2,1);
                }
                
                \node at (2,1) {$\bullet$};
                \node at (2,1.3) {$\infty$};
                
                \node at (3,1) {$\dots$};
                \node at (4,1) {$\bullet$};
                \node at (5,1) {$\bullet$};
                \node at (6,1) {$\bullet$};
                \node at (7,1) {$\bullet$};
            \end{tikzpicture}
        \]
        that is, the subalgebra of $\cofnaturals\times \fincofnaturals$ given by the subset
        \[
        \{(A,B) : A=\varnothing \Leftrightarrow B \text{ is finite}\}.
        \]
    \end{enumerate}
\end{corollary}
\begin{proof}
    By \cref{ex: non-example}, all infinite bounded chains and all free Heyting algebras over a nonempty set are not étale-finite. Therefore, by \cref{theorem.fp.etale}, they cannot arise as the lattices of truth values of a finitely propositional topos. The algebra $H$ in item \eqref{eq: algebra H} is not étale-finite, since the projection onto the first coordinate is a surjective homomorphism to $\cofnaturals$, which is not étale-finite; since étale-finite Heyting algebras are closed under quotients (\cref{rem:closure-properties}\eqref{eq: closure properties of etale-finite}), it follows that $H$ is not étale-finite.
\end{proof}

The examples in \cref{cor: examples of non etale-finite algebras} are quite different in nature:
\begin{enumerate}
    \item 
    The chain $\omega+1$ is complete, which implies that there is a (localic) topos $\E$ such that $\Sub_{\E}(1)\cong \omega+1$. Instead, we do not know whether $\omega+\omega^{\op}$ arises as the lattice of truth values of some elementary topos. 
    \item 
    A solution for all free algebras would imply a solution for all algebras (because of the filter-quotient construction and the fact that every Heyting algebra is the quotient of a free Heyting algebra).
    \item
    The algebra $\cofnaturals\times \fincofnaturals$ is not étale-finite; however, since it is the product of a complete Heyting algebra and a Boolean algebra, we know that it is topos-admissible.\footnote{Note that, given two toposes $\E_{0}$ and $\E_{1}$, the lattice of truth values of $\E_{0}\times \E_{1}$ is isomorphic to $\Sub_{\E_{0}}(1)\times \Sub_{\E_{1}}(1)$ (since the terminal object is simply the pair $(1_{\E_{0}},1_{\E_{1}})$ and subobjects are computed pointwise).}
    On the other hand, we do not know whether its subalgebra $\{(A,B) : A=\varnothing \Leftrightarrow B$ is finite$\}$ is topos-admissible, although it is similar to the examples studied in this paper (for example, every principal upset in the dual space is finite).
\end{enumerate}

\section{Conclusions and future work}

We have extended the class of Heyting algebras for which a positive answer to the Heyting-to-topos problem is known. From our point of view, this raises some questions that deserve to be analysed in greater depth. 

It is well known that Pitts' uniform interpolation theorem \cite{Pitts1992} has connections to the model theory of Heyting algebras \cite{Ghilardi2002}: the study of model completions of Heyting algebras \cite{darnierejunkermodelcompletionofcoheyting}, unification properties \cite{ghilardi_unificationintuitionisticlogic} and the existence of fixed points \cite{Santocanale2018RuitenburgsTV}. It would be interesting to see whether the existence of a topos $\E$ whose lattice of truth values is isomorphic to a given Heyting algebra $H$, with $\E$ enjoying properties such as finite propositionality, could similarly be studied from this point of view (i.e., in terms of the consequences of these properties for the model theory of the Heyting algebra $H$).

Finally, the obvious question left by the present work is how to solve the problem for Heyting algebras that are not étale-finite. In \cref{sec:finitely-propositional-toposes}, we noted that it is necessary to extend the current techniques beyond the use of local-homeomorphism-like constructions on compact spaces.
In \cite[Ch.~6]{almeidaphdthesis}, this situation and possible paths forward are discussed. One interesting line of development might be to consider using \emph{locally Esakia spaces} (in the sense of \cite{locallyesakiaspaces}).
We hope that the present paper stimulates work in this direction and that the methods outlined can ultimately contribute to solvidng the Heyting-to-topos problem.

\appendix
\section{A characterisation of étale-finite Esakia spaces}
\label{s:app-char}

\begin{proposition}[Characterisation of étale-finite Esakia spaces] \label{p:finite-principal-upset-types}
    An Esakia space $X$ is étale-finite if and only if the following conditions hold:
    \begin{enumerate}
        \item for every $x \in X$, the principal upset $\u x$ is finite.
        \item for every finite poset $P$, the set
        \[
            X_P \coloneqq \{x\in X : {\u}x\cong P\}
        \]
        is clopen.
    \end{enumerate}
\end{proposition}

\begin{proof}
    ($\Rightarrow$) Suppose $\acute e \colon X\to X_0$ is a continuous strict p-morphism with $X_0$ finite. For each $x \in X$, the restriction of $\acute{e}$ gives $\u x\cong\u\acute e(x)$; hence, each principal upset is finite. Moreover, for a finite poset $P$,
    \[
    X_P=\bigcup_{\{x_0\in X_0:\u x_0\cong P\}}\acute e^{-1}[\{x_0\}],
    \]
    which is clopen because $X_0$ is finite (and hence discrete) and $\acute e$ is continuous.

    ($\Leftarrow$)
    For each $n\in\N\smallsetminus\{0\}$, set
    \[
        X_n\coloneqq \{x\in X:\lvert{\u}x\rvert=n\}.
    \]
    Each $X_n$ is clopen: indeed, it is the finite union of the sets $X_P$, with $P$
    ranging over representatives of the finitely many isomorphism classes of posets of
    cardinality $n$.

    \begin{claim} \label{claim:finite-clopen-partition-separating-principal-upsets}
        There is a finite clopen partition
        \[
            B_1,\dots,B_r
        \]
        of $X$ such that, for every $x\in X$, distinct points of ${\u}x$ lie in distinct
        blocks of the partition.
    \end{claim}

    \begin{proof}[Proof of the claim]
        Let $x\in X$, and write $n\coloneqq \lvert{\u}x\rvert$. Since the underlying topological space of a Priestley space is a Stone space, we may choose pairwise
        disjoint clopen neighbourhoods
        \[
            U^x_y\qquad(y\in{\u}x)
        \]
        such that $y\in U^x_y$ for each $y\in{\u}x$. Define
        \[
            \mathcal{O}_x
            \coloneqq
            X_n
            \cap
            \Box\left(\bigcup_{y\in{\u}x}U^x_y\right)
            \cap
            \bigcap_{y\in{\u}x}{\d}U^x_y.
        \]
        Then $\mathcal{O}_x$ is clopen. Indeed, $X_n$ is clopen, the set
        $\bigcup_{y\in{\u}x}U^x_y$ is clopen, and both $\Box W$ and ${\d}W$ are clopen
        whenever $W$ is clopen in an Esakia space.

        Moreover, $x\in\mathcal{O}_x$. Thus the family $(\mathcal{O}_x)_{x\in X}$
        covers $X$, and hence, by compactness, there are $x^1,\dots,x^m\in X$ such that
        \[
            X=\mathcal{O}_{x^1}\cup\cdots\cup\mathcal{O}_{x^m}.
        \]

        Collect the finitely many clopen subsets
        \[
            U^{x^a}_y
            \qquad
            (1\leq a\leq m,\ y\in{\u}x^a),
        \]
        and let $B_1,\dots,B_r$ be the finite clopen partition of $X$ generated by them.

        We show that this partition has the desired property. Let $z\in X$, and choose
        $a$ such that $z\in\mathcal{O}_{x^a}$. Set
        \[
            n\coloneqq \lvert{\u}x^a\rvert.
        \]
        Since $z\in X_n$, we have $\lvert{\u}z\rvert=n$. Moreover, by the definition of
        $\mathcal{O}_{x^a}$, the set ${\u}z$ is contained in
        \[
            \bigcup_{y\in{\u}x^a}U^{x^a}_y
        \]
        and meets each of the pairwise disjoint clopens $U^{x^a}_y$. Since both ${\u}z$
        and ${\u}x^a$ have cardinality $n$, the set ${\u}z$ meets each $U^{x^a}_y$ in
        exactly one point. Hence, any two distinct points of ${\u}z$ are separated by one of
        the chosen clopens, and therefore lie in distinct blocks of the generated partition.
        This proves the claim.
    \end{proof}

    We now construct the finite target poset. For $x\in X$, define
    \[
        S_x\coloneqq \{i\in\{1,\dots,r\}:{\u}x\cap B_i\neq\varnothing\}.
    \]
    By \cref{claim:finite-clopen-partition-separating-principal-upsets}, for each $i\in S_x$ there is a unique point
    \[
        x_i\in{\u}x\cap B_i.
    \]
    Define a partial order on $S_x$ by
    \[
        i\leq_x j
        \quad\Longleftrightarrow\quad
        x_i\leq x_j.
    \]
    Let
    \[
        q_x\coloneqq (S_x,\leq_x).
    \]
    Thus, $q_x$ is a finite poset, and the map
    \begin{align*}
        {\u}x &\longrightarrow q_x\\
        y &\longmapsto \text{the unique }i\text{ such that }y\in B_i
    \end{align*}
    is an order-isomorphism.

    Let
    \[
        Q\coloneqq \{q_x:x\in X\}.
    \]
    This is finite, since there are only finitely many partial orders on subsets of
    $\{1,\dots,r\}$.

    If $q=(S,\leq_q)\in Q$ and $i\in S$, write
    \[
        q_{\geq i}\coloneqq \{j\in S:i\leq_q j\},
    \]
    with the induced order. Notice that $q_{\geq i}\in Q$: indeed, if $q=q_x$ and
    $y\in{\u}x$ is the unique point lying in $B_i$, then
    \[
        q_y=(q_x)_{\geq i}.
    \]

    Define an order on $Q$ by declaring, for $q=(S,\leq_q)$,
    \[
        q\leq_Q q'
        \quad\Longleftrightarrow\quad
        q'=q_{\geq i}\text{ for some }i\in S.
    \]
    This is a partial order. Reflexivity follows because each $q_x$ has a least element,
    namely the block containing $x$. Transitivity is immediate. For antisymmetry, if
    $q\leq_Q q'$ and $q'\leq_Q q$, then the underlying set of $q'$ is contained in the
    underlying set of $q$, and conversely. Hence, the two underlying sets are equal; since
    principal upsets carry the induced order, it follows that $q=q'$.

    We regard $Q$ as a finite Esakia space with the discrete topology. Define
    \begin{align*}
        \acute{e}\colon X &\longrightarrow Q,\\
        x & \longmapsto q_x.
    \end{align*}
    
    We prove that $\acute{e}$ is continuous. Since $Q$ is finite and discrete, it suffices to show that
    each fibre of $\acute{e}$ is clopen. For $i\in\{1,\dots,r\}$, put
    \[
        D_i\coloneqq \{x\in X:{\u}x\cap B_i\neq\varnothing\}
        =
        {\d}B_i.
    \]
    This is clopen because $X$ is Esakia. For $i,j\in\{1,\dots,r\}$, put
    \[
        E_{ij}\coloneqq
        \{x\in X:\text{there are }y\in{\u}x\cap B_i
        \text{ and }z\in{\u}x\cap B_j\text{ with }y\leq z\}.
    \]
    Equivalently,
    \[
        E_{ij}={\d}(B_i\cap{\d}B_j).
    \]
    Hence $E_{ij}$ is clopen, since $B_j$ is clopen, ${\d}B_j$ is clopen, and $X$ is
    Esakia.

    Fix $q=(S,\leq_q)\in Q$. Then
    \[
        \acute{e}^{-1}[\{q\}]
        =\bigcap_{i\in S}D_i
        \cap
        \bigcap_{i\notin S}(X\smallsetminus D_i)\cap \bigcap_{\substack{i,j\in S\\ i\leq_q j}}E_{ij}
        \cap \bigcap_{\substack{i,j\in S\\ i\not\leq_q j}}(X\smallsetminus E_{ij}).
    \]
    Indeed, the first two intersections express exactly that the set of colours appearing above a point is $S$, while the last two intersections express exactly that the order between those colours is $\leq_q$. Therefore, each fibre of $\acute{e}$ is clopen, and so $\acute{e}$ is continuous.

    It remains to show that $\acute{e}$ is a strict p-morphism. Let $x\in X$, and set
    $q\coloneqq \acute{e}(x)$. Consider the map
    \begin{align*}
        \alpha_x\colon {\u}x &\longrightarrow q\\
        y &\longmapsto \text{the unique }i\text{ such that }y\in B_i.
    \end{align*}
    By construction, $\alpha_x$ is an order-isomorphism. Also, the map
    \begin{align*}
        \beta_q\colon q &\longrightarrow {\u}_Q q\\
        i &\longmapsto q_{\geq i}
    \end{align*}
    is an order-isomorphism. It is surjective by the definition of $\leq_Q$. It is injective
    because, if $q_{\geq i}=q_{\geq j}$, then $i\in q_{\geq j}$ and $j\in q_{\geq i}$, hence
    $j\leq_q i$ and $i\leq_q j$, and therefore $i=j$.

    We now check that $\beta_q$ preserves and reflects the order. If $i\leq_q j$, then
    $q_{\geq j}=(q_{\geq i})_{\geq j}$, and hence $q_{\geq i}\leq_Q q_{\geq j}$. Conversely,
    if $q_{\geq i}\leq_Q q_{\geq j}$, then
    \[
        q_{\geq j}=(q_{\geq i})_{\geq k}
    \]
    for some $k\in q_{\geq i}$. Thus $q_{\geq j}=q_{\geq k}$, and by injectivity $j=k$.
    Since $k\in q_{\geq i}$, it follows that $i\leq_q j$.

    For every $y\in{\u}x$ we have
    \[
        \acute{e}(y)=\beta_q(\alpha_x(y)).
    \]
    Hence
    \[
        \acute{e}\!\restriction_{{\u}x}\colon {\u}x\longrightarrow {\u}_Q \acute{e}(x)
    \]
    is an order-isomorphism.

    By \cref{rem: external characterisation of etale maps}, $\acute{e}$ is a strict p-morphism.
    Thus $\acute{e}\colon X\to Q$ is a continuous strict p-morphism to a finite Esakia space.
    Therefore, $X$ is étale-finite.
\end{proof}

\begin{remark}[\'Etale-finite $\Leftrightarrow$ principal upsets are finite and locally constant]
    The characterisation in \cref{p:finite-principal-upset-types} immediately entails the following similar characterisation: 
    \begin{quotation}
        An Esakia space $X$ is étale-finite if and only if, for every $x \in X$, $\u x$ is finite and there is an open neighbourhood $U$ of $x$ such that, for every $y \in U$, the posets $\u x$ and $\u y$ are isomorphic.
    \end{quotation}
    Indeed, it is clear that the condition in \cref{p:finite-principal-upset-types} implies that principal upsets are finite and locally constant.
    Conversely, if principal upsets are finite and locally constant, then all the sets of the form $X_P$ ($P$ a finite poset) are open and, since they partition $X$, also closed.
\end{remark}

\section{\texorpdfstring{$\KM$}{KM}-algebras}\label{sec: a further example}

As a further example of étale-finite Heyting algebras, we show that the Heyting reducts of those $\KM$-algebras (a well-studied class of modal Heyting algebras \cite{Muravitsky2014,Esakia2006}) that satisfy linearity and bounded-depth axioms (see e.g.\ \cite{Litak2014}) are étale-finite. This provides an especially natural class of examples for which our \cref{thm: elementary topos from etale-finite Heyting algebra} provides a positive answer to the Heyting-to-topos problem.

\begin{definition}[$\KM$-algebras, $\KM$ Heyting algebras] \label{KM-algebra definition}
    Let $H$ be a Heyting algebra. We say that a pair $(H,\Box)$, where $\Box \colon H\to H$ is a unary function, is a \emph{$\KM$-algebra}\footnote{Such algebras are also known as \emph{frontons} \cite{Esakia2006}, \emph{$\Delta$-pseudo-Boolean algebras} \cite{Muravitsky2014}, and by other names in the provability logic literature \cite{Visser1982}. Given their logical counterpart in the Kuznetsov-Muravitsky calculus, we adopt the present terminology, given in \cite[Def.~9]{Muravitsky2014}.}
    if the following conditions hold for all $a,b\in H$:
    \begin{enumerate}
        \item (Normality) $\Box(a\wedge b)=\Box a\wedge \Box b$ and $\Box\top=\top$;
        \item (Inflationarity) $a\leq \Box a$;
        \item (Minimality) $\Box a\leq b\vee (b\rightarrow a)$;
        \item (Density) $\Box a\rightarrow a\leq a$.
    \end{enumerate}
    We say that $H$ is a \emph{$\KM$ Heyting algebra} if it is the Heyting reduct of a (necessarily unique\footnote{This uniqueness stems from the fact that, for each $a\in H$, $\Box a$ is the least dense element above $a$, as explained for example in \cite[Prop.~5]{Esakia2006}.}) $\KM$-algebra. 
\end{definition}

$\KM$-algebras have recently been the subject of attention due to their intimate connection with Heyting algebras; for example, every variety of Heyting algebras is generated by its $\KM$ Heyting algebras\footnote{This result is due to Alexander Kuznetsov. An algebraic proof is the main result of \cite{explicitkuznetsovmuravitsky}.}. Note also that a locale (i.e., a complete Heyting algebra) is a $\KM$ Heyting algebra if and only if it is scattered (i.e., if the Booleanisation of every open sublocale is open), see \cite{Esakia2000}. The $\KM$-Gödel algebras---i.e., the $\KM$-algebras satisfying Gödel's axiom $(p\rightarrow q)\vee (q\rightarrow p)$---appear naturally in the context of the topos of trees, used in functional programming (see e.g.\ \cite{Litak2014}).

$\KM$ Heyting algebras admit the following well-known dual characterisation (see e.g.\ \cite[Prop.~4.8]{CastiglioniSagastumeEtAl2010}):

\begin{proposition}[Dual characterisation of $\KM$ Heyting algebras]\label{prop: characterisation of KM}
    A Heyting algebra $H$ is a $\KM$ Heyting algebra if and only if $X=\Spec(H)$ has the property that, for each clopen $W$, the set of its maximal elements
    \[
    \max(W) \coloneqq \Bigl\{x\in W : \text{ for all } y \in W \text{ if }x\leq y \text{ then } x = y\Bigr\}
    \]
    is clopen.
\end{proposition}

From this characterisation, we can see that every étale-finite Esakia space is dual to a $\KM$ Heyting algebra:

\begin{proposition}[\'Etale-finite $\Rightarrow$ $\KM$]\label{prop: etale-finite yields KM}
    For every étale-finite Esakia space $X$, if $W\subseteq X$ is clopen, then $\max (W)$ is clopen. 
    Dually, every étale-finite Heyting algebra is a $\KM$ Heyting algebra.
\end{proposition}
\begin{proof}
    Let $\acute{e}\colon X\to X_{0}$ be a continuous strict p-morphism. By \cref{prop: characterisation of KM}, it suffices to show that, for each clopen subset $W$ of $X$, $\max(W)$ is clopen.
    
    Let $W$ be a clopen subset of $X$, and $x\in \max(W)$. Then, $W\cap \Box(W\climp \acute{e}^{-1}[\acute{e}(x)])$ is a clopen subset (since $X_0$ is discrete, $\acute{e}$ is continuous and $X$ is an Esakia space), and $x\in W\cap \Box(W\climp \acute{e}^{-1}[\acute{e}(x)])$ since $x$ is maximal. Moreover, $W \cap \Box(W\climp \acute{e}^{-1}[\acute{e}(x)])$ is contained in $\max(W)$: if $y \in W \cap \Box(W\climp \acute{e}^{-1}[\acute{e}(x)])$ and $y\leq y'\in W$, then, by definition, $\acute{e}(y)=\acute{e}(y')=\acute{e}(x)$, and, since $y\leq y'$, we have $y=y'$. Thus, $\max(W)$ is open. It is also closed because $X$ is an Esakia space: by \cite[Thm.~3.2.6]{Esakiach2019HeyAlg}, $W$ is an Esakia space with the induced topology and order, and so, by \cite[Thm.~3.2.3]{Esakiach2019HeyAlg}, $\max(W)$ is closed in the subspace topology; since $W$ is clopen, $\max(W)$ is closed in $X$. Thus $\max(W)$ is clopen, as desired.
\end{proof}

We recall that a \emph{Gödel algebra} is a Heyting algebra that validates the Gödel--Dummett axiom $(p\rightarrow q)\vee (q\rightarrow p)$.
It is known that a Heyting algebra is a Gödel algebra if and only if it belongs to a variety generated by a class of chains.

\begin{definition}[$\KM$-tabular Gödel algebra]
     A $\KM$-\emph{tabular Gödel algebra} is a $\KM$ Heyting algebra $H$ for which there is a finite chain $[k]\coloneqq 1<\dots<k$ for $k\geq 1$ such that $H\in \mathsf{Var}_{\mathbf{HA}}([k])$.
\end{definition}
Every $\KM$-tabular Gödel algebra is in particular a G\"odel algebra.

We can thus show that all algebras belonging to a natural class of $\KM$ Heyting algebras are étale-finite:

\begin{proposition}[For G\"odel algebras, étale-finite $\Leftrightarrow$ $\KM$-tabular]
    For every Heyting algebra $H$, the following are equivalent:
    \begin{enumerate}
        \item \label{eq: Gödel + etale-finite} $H$ is a Gödel algebra and is étale-finite;
        \item \label{eq: KM-tabular Godel} $H$ is a $\KM$-tabular Gödel algebra.
    \end{enumerate}
\end{proposition}
\begin{proof}
    It is well known that a Heyting algebra is a Gödel algebra if and only if every principal upset in its dual Esakia space is a chain; see \cite[Thm.~1.5]{hornlinearheyting}.
    Moreover, it was essentially shown by Hosoi \cite{Hosoi1967-HOSOIL-2} that, for every $k\geq 0$, a Heyting algebra belongs to the variety generated by $[k+1]$ if and only if every principal upset in its dual Esakia space is a chain of cardinality at most $k$ (see also \cite{Dunn1971}).

    In the following, set $X \coloneqq \Spec(H)$.

    \eqref{eq: Gödel + etale-finite}$\Rightarrow$\eqref{eq: KM-tabular Godel} Since $H$ is a Gödel algebra, every principal upset in $X$ is a chain. Since $H$ is étale-finite, there is a finite poset $X_{0}$ and a continuous strict p-morphism $\acute{e}\colon X\to X_{0}$.
    Up to replacing $X_0$ by the image of $\acute{e}$, we can assume that $\acute{e}$ is surjective.
    Therefore, every principal upset of $X_0$ is a finite chain.
    Let $X_0'$ be the quotient of $X_0$ obtained by identifying points of the same depth; then, $X_0'$ is a finite chain.
    Then, the quotient $\pi \colon X_0 \twoheadrightarrow X_0'$ is a strict p-morphism, and so the composite $X \xrightarrow{\acute{e}}X_0 \xrightarrow{\pi} X_0'$ is a continuous strict p-morphism.
    By \cref{t:strict-dual-to-Jib}, the dual homomorphism
    \(H_0\to H\), where $H_0 =\ClopUp(X'_0)$, satisfies Jibladze's law. Hence $H$ is
    $H_0$-étale, and, by \cref{prop: characterisation of etale maps}, $H$ is the Heyting reduct of an algebra in $\mathsf{Var}(H_{0})_{h\in H_0}$, where $H_{0}=\ClopUp(X_{0}')$; thus, $H$ belongs to the variety generated by a finite chain.  Moreover, $H$ is a $\KM$ Heyting algebra by \cref{prop: etale-finite yields KM}. Thus $H$ is a $\KM$-tabular Gödel algebra as desired.
    
    \eqref{eq: KM-tabular Godel}$\Rightarrow$\eqref{eq: Gödel + etale-finite} As noted above, there is $k \geq 0$ such that every principal upset in $X$ is a chain of cardinality at most $k$. We show that $X$ is étale-finite. For this purpose, note that to each point $x\in X$ we can assign its depth $d(x) \coloneqq \lvert {\uparrow}x\rvert$. Let $k^{*}=k<\dots<1$ be the $k$-element chain with the reverse of the natural order. Then the map
    \begin{align*}
        d\colon X&\longrightarrow k^{*}\\
        x&\longmapsto d(x)
    \end{align*}
    is a continuous strict p-morphism. The fact that it is a strict p-morphism follows from simple cardinality considerations, since, for each $y\in {\uparrow}x$, one has $\lvert{\uparrow}y\rvert=\lvert{\uparrow}x\rvert-\lvert [x,y)\rvert$ where $[x,y)=\{z\in X : x\leq z<y\}$. 
    
    We finally show that the map $d$ is continuous: for each $j\in \{1,\dots,k\}$, consider the set $L_{j}\coloneqq\{x\in X : d(x)=j\}$. Note that $L_{1}=\max(X)$, since maximal elements are exactly the elements such that their upset is a singleton. 
    Since $X$ is a clopen subset of itself, \cref{prop: characterisation of KM} gives that $\max(X)$ is clopen. For each $1\leq j<k$ we have $L_{j+1}=\max(X\smallsetminus \bigcup_{i = 1}^{j}L_{i})$; so, by induction, all such subsets are clopen. This shows that $d$ is continuous and concludes the proof.
\end{proof}

\makeatletter
\begingroup
\let\addcontentsline\@gobblethree
\section*{Acknowledgements} 
We would like to thank Mamuka Jibladze, Evgeny Kuznetsov, and Benno van den Berg for useful discussions on the subject of this paper, and Andrew Pitts for providing some insights into the problem.

\emph{Funding.}
The first author was funded by UK Research and Innovation (UKRI) under the UK government's Horizon Europe funding guarantee (grant number EP/Y015029/1, Project ``DCPOS'') during his affiliation at the University of Birmingham, and by an FSR Incoming Postdoctoral Fellowship during his affiliation at the Université catholique de Louvain.
The third author acknowledges support from the Basque Government through grants IT1483-22 and IT1913-26, and a postdoctoral fellowship (grant POS-2022-1-0015 and POS-2025-2-0019).
\endgroup
\makeatother

\bibliographystyle{alpha}
\bibliography{Biblio}

\section*{Author information}

\authorwithphoto
{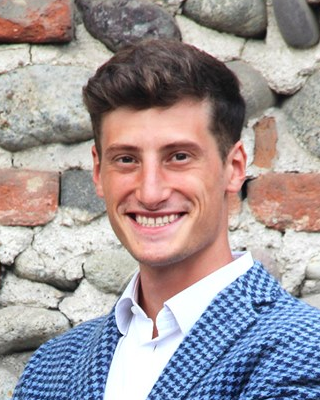}
{Marco Abbadini}
{Research Institute in Mathematics and Physics\\ Université catholique de Louvain, Louvain-la-Neuve, 1348, Belgium}
{marco.abbadini@uclouvain.be}{https://marcoabbadini-uni.github.io}

\authorwithphoto
{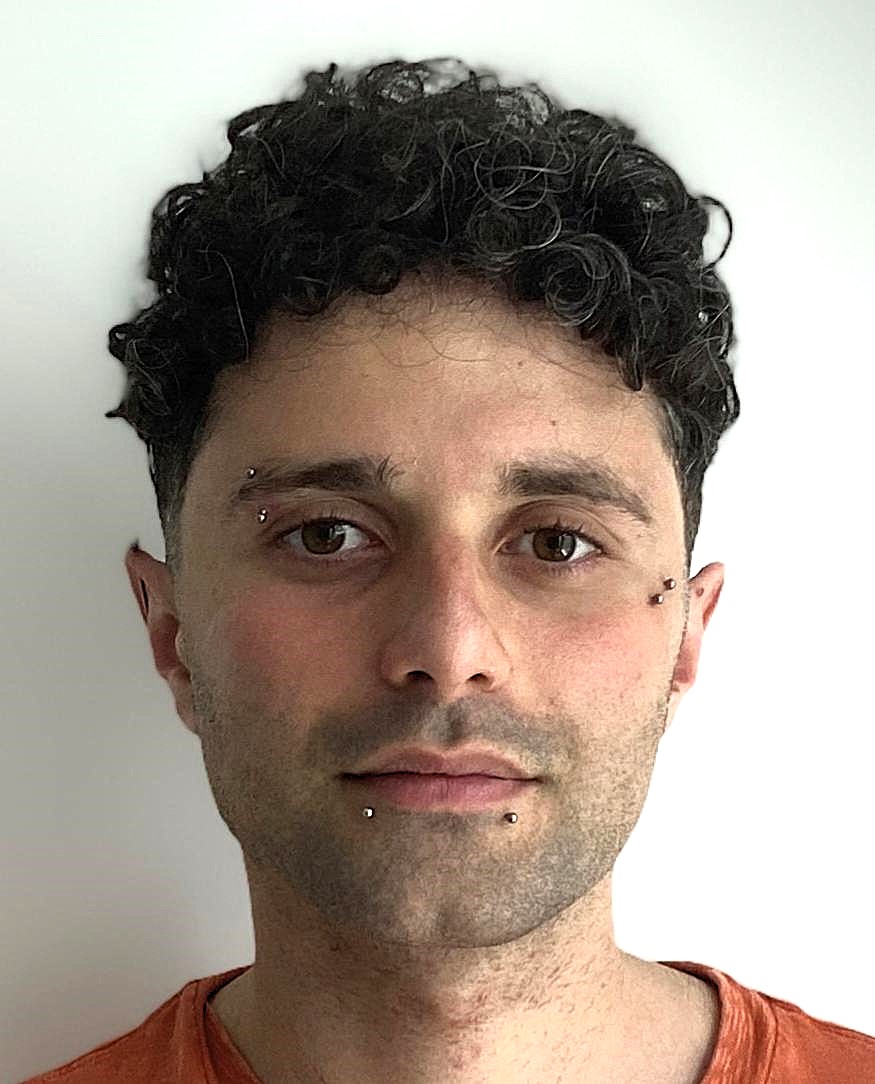}
{Rodrigo Nicolau Almeida}
{Institute for Logic, Language and Computation\\ University of Amsterdam, Amsterdam, 1098XH, The Netherlands}
{r.dacruzsilvapinadealmeida@uva.nl}{https://rodrigonalmeida.github.io}

\authorwithphoto
{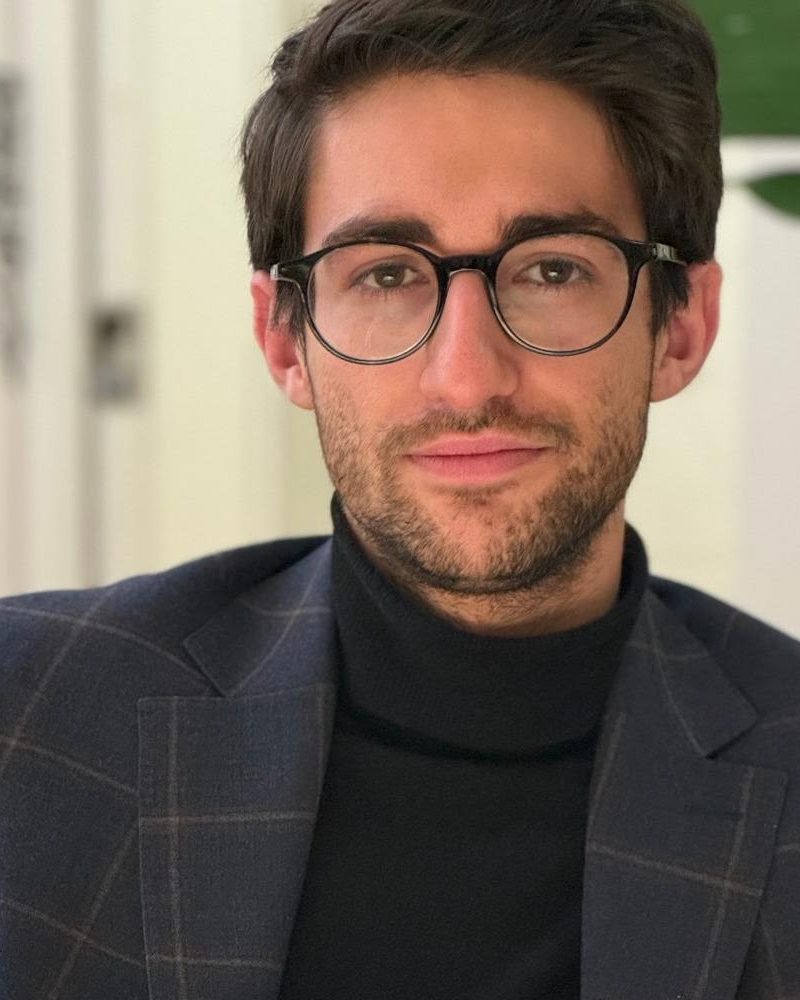}
{Igor Arrieta}
{Department of Mathematics\\ University of the Basque Country UPV/EHU, Bilbao, 48080, Spain}
{igor.arrieta@ehu.eus}
{https://sites.google.com/view/igorarrieta}

\end{document}